\documentclass[a4paper]{amsart}
\usepackage{a4wide,amsfonts,amssymb,amsmath,enumerate,amscd,color,tikz,tikz-3dplot,tikz-cd,todonotes}
\usepackage[english]{babel}
\usepackage{hyperref}
\allowdisplaybreaks
\usepackage{comment}
\newtheorem{lem}{Lemma}[section]

\newtheorem{theo}[lem]{Theorem}
\newtheorem{cor}[lem]{Corollary}
\newtheorem{rem}[lem]{Remark}

\newcommand{\Hmmc}[2]{\leavevmode{\marginpar{\tiny%
\color{#1}$\hbox to 0mm{\hspace*{-0.5mm}$\leftarrow$\hss}%
\vcenter{\vrule depth 0.1mm height 0.1mm width \the\marginparwidth}%
\hbox to
0mm{\hss$\rightarrow$\hspace*{-0.5mm}}$\\\relax\raggedright #2}}}

\def\bs{\boldsymbol}

\def\ol{\overline}

\def\incl{\hookrightarrow}
\def\reals{\mathbb{R}}
\def\nat{\mathbb{N}}

\def\rt{\reals^{3}}
\def\ga{\Gamma}
\def\gap{\ga_{\Phi}}
\def\gat{\ga_{\!t}}
\def\gan{\ga_{\!n}}
\def\gatp{\ga_{\!t,\Phi}}

\def\ganp{\ga_{\!n,\Phi}}

\def\om{\Omega}
\def\omp{\om_{\Phi}}
\def\omb{\ol{\om}}
\def\ombp{\omb_{\Phi}}
\def\n{\mathrm{n}}

\def\sfL{\mathsf{L}}
\def\sfH{\mathsf{H}}
\def\sfC{\mathsf{C}}

\def\sfR{\mathsf{R}}
\def\sfD{\mathsf{D}}

\renewcommand{\L}{\sfL}
\renewcommand{\H}{\sfH}
\newcommand{\bH}{\bs\H}
\newcommand{\C}{\sfC}
\newcommand{\R}{\sfR}
\newcommand{\D}{\sfD}
\newcommand{\bR}{\bs\R}
\newcommand{\bD}{\bs\D}

\newcommand{\eps}{\varepsilon}
\DeclareMathOperator{\dist}{dist}
\DeclareMathOperator{\supp}{supp}

\DeclareMathOperator{\A}{A}
\DeclareMathOperator{\cA}{\mathcal{A}}
\DeclareMathOperator{\cL}{\mathcal{L}}

\DeclareMathOperator{\p}{\partial}
\DeclareMathOperator{\id}{id}
\DeclareMathOperator{\adj}{adj}
\DeclareMathOperator{\sym}{sym}

\DeclareMathOperator{\tr}{tr}

\DeclareMathOperator{\symtr}{symtr}
\DeclareMathOperator{\ed}{d}

\DeclareMathOperator{\na}{\nabla}

\DeclareMathOperator{\rot}{rot}

\DeclareMathOperator{\divergence}{div}
\def\div{\divergence}

\newcommand{\wt}[1]{\widetilde{#1}}

\newcommand{\norm}[1]{|#1|}

\newcommand{\scp}[2]{\langle#1,#2\rangle}
\newcommand{\bscp}[2]{\big\langle#1,#2\big\rangle}
\newcommand{\Bscp}[2]{\Big\langle#1,#2\Big\rangle}

\title[Shape Derivatives of the  Eigenvalues of the de Rham Complex: Part I]
{Shape Derivatives\\ of the  Eigenvalues of the de Rham Complex\\  
for Lipschitz deformations and variable coefficients:\\
Part I}
\author{Pier Domenico Lamberti}
\author{Dirk Pauly}
\author{Michele Zaccaron}
\address{Dipartimento di Tecnica e Gestione dei Sistemi Industriali, University of Padova, Italy}
\email[Pier Domenico Lamberti]{lamberti@math.unipd.it}
\address{Institut f\"ur Analysis, Technische Universit\"at Dresden, Germany}
\email[Dirk Pauly]{dirk.pauly@tu-dresden.de}
\address{Aix Marseille Univ , CNRS, Centrale Marseille, Institut Fresnel, Marseille, France}
\email[Michele Zaccaron]{michele.zaccaron@univ-amu.fr}
\date{\today}
\thanks{{\it Corresponding Author}: Dirk Pauly}

\begin{document}


\begin{abstract}
We study eigenvalue problems for the de Rham complex on varying three dimensional domains. 
Our analysis includes the Helmholtz equation as well as the Maxwell system with mixed boundary conditions and non-constant coefficients. 
We provide Hadamard-type formulas for the shape derivatives under weak regularity assumptions on the domain and its perturbations. 
Our proofs are based on abstract results adapted to varying Hilbert complexes. 
As a bypass product of our analysis we give a proof of the celebrated Helmann-Feynman theorem 
both for simple and multiple eigenvalues of suitable families 
of self-adjoint operators in Hilbert space depending on possibly infinite dimensional parameters. 
This series of papers consists of Parts I and II.
\end{abstract}


\maketitle
\setcounter{tocdepth}{3}
{\small
\tableofcontents}


\section{Introduction}

The analysis of the dependence of the eigenvalues and eigenfunctions of elliptic operators upon variation of the underlying domain is a classical problem considered in many papers in the literature with applications in approximation, optimization, homogenization, control theory and mathematical physics. It is impossible to give an account of all contributions in the literature and we refer to the monograph \cite{henry} for an introduction to this topic in particular to the method of transplantation used in this paper. Needless to say that the  Laplace operator and  other second order partial differential equations have received much more attention than higher order operators and systems, the analysis of which often leads to various technical and theoretical  obstructions  as well as paradoxes, see for instance  \cite{arferlamtri, arferlamdumb, buosopoly, ferraressotri, ferlam} for polyharmonic operators and  to \cite{buososystems, buosoreissner} for elliptic systems. 
From this point of view, the case of the Maxwell system has been investigated even less, cf.  \cite{fermar, jimbo, johnson, LZ2021a, lamzac23, zac23}.  In particular we note that  differentiability results and Hadamard-type formulas for shape derivatives are proved in   \cite{jimbo, LZ2021a, LZ2022a, zac23}   under suitable regularity assumptions on the domains and the corresponding perturbations.  

The main aim of the present series of papers is to relax those regularity assumptions gaining one degree of smoothness and to provide a unified approach including both the Helmholz equation and the Maxwell system.  This is done by analyzing the corresponding  de Rham complex and its domain perturbations. A further contribution of our papers consists in the fact that  we consider nonconstant physical parameters such as the electric permittivity  $\eps$ and the magnetic permeability $\mu$.     In particular we give a rigorous proof of a formula found by  Hiromasa Hirakawa in \cite[pp.~91-93]{hira}  which is a Hadamard-type formula for the Maxwell system.      Moreover, we consider  the general case of mixed Dirichlet-Neumann boundary conditions.

We note that the proof of the Hadamard formulas can be obtained at a formal level  by applying the Hellmann-Feymann Theorem, a classical result is quantum mechanics that reduces here to a straightforward differentiation of the Rayleigh quotients depending on a parameter (see  \cite{este} for a recent discussion on this theorem and references).  However, in order to discuss the dependence of the eigenvalues on infinite dimensional parameters and to consider multiple eigenvalues, we follow the approach developed in \cite{LLdC2004a}  and in particular we consider the elementary symmetric functions of the eigenvalues since  these functions depend smoothly on the parameters as simple eigenvalues do.  The results in  \cite{LLdC2004a} concern general families of compact selfadjoint operators in Hilbert space with variable scalar product and are applied in \cite{LZ2021a} to the Maxwell problem. To do so, the authors of \cite{LZ2021a} have to consider a penalised problem  which requires $\C^{1,1}$ regularity assumptions on the domain perturbations. Here, in order to consider domain perturbations of class $\C^{0,1}$,  we do not penalise the problem  but this prevents us from using the results of \cite{LLdC2004a}   in a direct way because the operators under consideration are selfadjoint but not compact\footnote{the so-called reduced operators are compact but this does not help too much because  their domain depends too heavily on the perturbations}. Thus we are forced to give new proofs of abstract theorems concerning families of selfadjoint operators in Hilbert space. As a bypass product of our analysis, we provide a proof of 
a general version of the  Helmann-Feymann Theorem for families of operators  suitable for de Rham complexes in Hilbert spaces, 
see Part II of this paper at hand.

This Part I of the paper series is organised as follows. 
Section~\ref{presec} is devoted to notations and preliminaries on the Functional Analysis Toolbox.  
Section~\ref{eigensec} is devoted to the analysis of the eigenvalue problem for the de Rham complex on transplanted domains.
In Section \ref{sec:conclu} we conclude this first part 
with some formal computations to derive the shape derivatives of the eigenvalues 
assuming that those are simple and the corresponding eigenvectors are differentiable.

In Part II of this series of papers we present 
Hadamard type formulas and related findings
obtained by a direct application of the  Helmann-Feymann Theorem
together with sound proofs of all results.

We conclude this introduction with two subsections 
where we highlight the main problems under consideration  
and briefly discuss the approach of domain transplantation used in this paper. 

Until stated otherwise, let $\om$ be a \emph{bounded open set} in $\rt$ with boundary $\ga$
and let $\lambda_{0},\lambda_{1}>0$. Moreover, let $\nu\in\L^{\infty}(\om,\reals)$
be positive with respect to the $\L^{2}(\om)$-inner product,
and let $\eps$ and $\mu$ be \emph{admissible} symmetric matrix fields, i.e.,
$\eps$ and $\mu$ belong to $\L^{\infty}(\om,\reals^{3\times3}_{\sym})$
and are positive with respect to the $\L^{2}(\om)$-inner product, cf.~\cite{BPS2016a,P2017a,P2019b,PS2022a}.   
The required regularity of $\ga, \nu, \eps, \mu$ will be specified along the paper.

\subsection{Eigenvalues of the De Rham Complex}

We shall consider the Dirichlet Maxwell eigenvalue problem 
\begin{align}
\begin{aligned}
\label{eq:maxev1}
\eps^{-1}\rot\mu^{-1}\rot E&=\lambda_{1}E
&\text{in }&\om,\\
\n\times E&=0
&\text{on }&\ga.
\end{aligned}
\end{align}
The corresponding Neumann Maxwell eigenvalue problem reads
\begin{align}
\begin{aligned}
\label{eq:maxev2}
\eps^{-1}\rot\mu^{-1}\rot E&=\lambda_{1}E
&\text{in }&\om,\\
\n\times\mu^{-1}\rot E&=0
&\text{on }&\ga.
\end{aligned}
\end{align}
Note that any solution of \eqref{eq:maxev1} or \eqref{eq:maxev2} automatically satisfies $\div\eps E=0$ in $\om$,
and that in \eqref{eq:maxev1} and \eqref{eq:maxev2} we have additional 
$\n\cdot\rot E=0$ and $\n\cdot\eps E=0$ on $\ga$, respectively.

We investigate also mixed boundary conditions, i.e., 
the Maxwell eigenvalue problem with mixed Dirichlet/Neumann boundary conditions 
\begin{align}
\begin{aligned}
\label{eq:maxev3}
\eps^{-1}\rot\mu^{-1}\rot E&=\lambda_{1}E
&\text{in }&\om,\\
\n\times E&=0
&\text{on }&\gat,\\
\n\times\mu^{-1}\rot E&=0
&\text{on }&\gan,
\end{aligned}
\end{align}
where $\ga$ is decomposed into two relatively open subsets
$\emptyset\subset\gat\subset\ga$ and $\gan:=\ga\setminus\ol{\gat}$. 
Note that again we have $\div\eps E=0$ and $\n\cdot\rot E|_{\gat}=0$ and $\n\cdot\eps E|_{\gan}=0$.

Moreover, we shall discuss the full spectrum of the de Rham complex.
Hence, we also investigate the scalar Laplacian and its dual, i.e.,
\begin{align}
\begin{aligned}
\label{eq:lapev1}
-\nu^{-1}\div\eps\na u&=\lambda_{0}u
&\text{in }&\om,\\
u&=0
&\text{on }&\gat,\\
\n\cdot\eps\na u&=0
&\text{on }&\gan,
\end{aligned}
\intertext{and}
\begin{aligned}
\label{eq:lapev2}
-\na\nu^{-1}\div\eps H&=\lambda_{0}H
&\text{in }&\om,\\
\nu^{-1}\div\eps H&=0
&\text{on }&\gat,\\
\n\cdot\eps H&=0
&\text{on }&\gan.
\end{aligned}
\end{align}
As in \eqref{eq:maxev3} it holds $\rot H=0$ and $\n\times H|_{\gat}=0$ but only
$\int_{\ga}\nu^{-1}\div\eps H=0$ if $\gan=\ga$.

In view of \eqref{eq:maxev3} and \eqref{eq:lapev2} we shall also discuss 
the generalised vector Laplacian
\begin{align}
\begin{aligned}
\label{eq:lapev3}
(\eps^{-1}\rot\mu^{-1}\rot-\na\nu^{-1}\div\eps)E&=\lambda_{0,1}E
&\text{in }&\om,\qquad
\lambda_{0,1}\in\{\lambda_{0},\lambda_{1}\},\\
\n\times E=0,\quad
\nu^{-1}\div\eps E&=0
&\text{on }&\gat,\\
\n\times\mu^{-1}\rot E=0,\quad
\n\cdot\eps E&=0
&\text{on }&\gan.
\end{aligned}
\end{align}
Note that for $\eps$, $\mu$, and $\nu$ being the identity mappings we have
$$\eps^{-1}\rot\mu^{-1}\rot-\na\nu^{-1}\div\eps=\rot\rot-\na\div =         -\vec\Delta.$$

\subsection{Shape Derivatives of Eigenvalues}

We intend to study variations of the domain and the boundary conditions
by replacing $\om$ and the boundary parts $\gat$, $\gan$ with 
$$\omp:=\Phi(\om),\qquad
\gap:=\Phi(\ga),\qquad
\gatp:=\Phi(\gat),\qquad
\ganp:=\Phi(\gan),$$
respectively, where 
$$\Phi:\om\to\omp$$ 
is a bi-Lipschitz transformation.
In particular, for $\ell\in\{0,1\}$,
we are interested in the variations of the eigenvalues   
$$0<\lambda_{\ell,1}(\Phi)\leq\lambda_{\ell,2}(\Phi)\leq\dots<\lambda_{\ell,k-1}(\Phi)\leq\lambda_{\ell,k}(\Phi)\leq\dots\to\infty$$  
in the domain $\omp$ and their elementary symmetric functions
with respect to changing transformations $\Phi$.

For simplicity, assume here that  $\eps$, $\mu$ are  the identity matrices  and that $\nu=1$. 
Let $\lambda_{0,k}(\Phi)$ and $\lambda_{1,k}(\Phi)$ be eigenvalues
with  eigenvectors $u$ and $E$
of \eqref{eq:lapev1} and \eqref{eq:maxev3}, respectively. As is well-known these eigenvalues can be written by means of Rayleigh quotients as 
\begin{equation}
\label{introray}
\lambda_{0,k}(\Phi)
=\frac{\norm{\na u}_{\L^{2}(\omp)}^{2}}{\norm{ u}_{\L^{2}(\omp)}^2},\qquad
\lambda_{1,k}(\Phi)
=\frac{\norm{\rot E}_{\L^{2}(\omp)}^{2}}{  \norm{ E}_{\L^{2}(\omp)}^{2}}.
\end{equation}
In particular, assuming  that the eigenvectors are normalized in $\L^{2}(\omp)$ we have 
$$
\lambda_{0,k}(\Phi)
=\norm{\na u}_{\L^{2}(\omp)}^{2},\qquad
\lambda_{1,k}(\Phi)
= \norm{\rot E}_{\L^{2}(\omp)}^{2}.
$$
Then also the dual eigenvectors 
\begin{align*}
H&:=\lambda_{0,k}^{-1/2}(\Phi)\na u,
&
B&:=\lambda_{1,k}^{-1/2}(\Phi)\rot E
\intertext{are $\L^{2}(\omp)$-normalised eigenvectors
of \eqref{eq:lapev2} and the respective dual of \eqref{eq:maxev3}. Moreover, we have the dualities}
u&=-\lambda_{0,k}^{-1/2}(\Phi)\div H,
&
E&=\lambda_{1,k}^{-1/2}(\Phi)\rot B,\\
\lambda_{0,\Phi}
&=\norm{\div H}_{\L^{2}(\omp)}^{2},
&
\lambda_{1,\Phi}
&=\norm{\rot B}_{\L^{2}(\omp)}^{2}.
\end{align*}

Note that the dual of \eqref{eq:maxev3} reads 
\begin{align}
\begin{aligned}
\label{eq:maxev3dual}
\mu^{-1}\rot\eps^{-1}\rot B&=\lambda_{1}B
&\text{in }&\om,\\
\n\times\eps^{-1}\rot B&=0   
&\text{on }&\gat,\\
\n\times B&=0
&\text{on }&\gan,
\end{aligned}
\end{align}
which incorporates also the conditions $\div\mu B=0$ and $\n\cdot \mu B  |_{\gat}=0$ and $\n\cdot \rot B |_{\gan}=0$.

In this paper, among other results, we  prove  Hadamard type formulas for the directional derivatives of the maps
$\Phi\mapsto \lambda_{l,k}(\Phi)$ for $l\in \{0,1\}$. 
This means that, given a fixed direction $\wt{\Psi}$ in the space of Lipschitz transformations $\Phi$, we compute the limit 
\begin{equation}
\label{eq:intro-deriv}
\p_{\wt{\Psi}}\lambda_{l,k}(\Phi)
=\lim_{h\to0}\frac{  \lambda_{l,k}(\Phi+h\wt{\Psi})-\lambda_{l,k}(\Phi)}{h}.
\end{equation}
Recall
$\p_{\wt{\Psi}}\lambda_{l,k}(\Phi)
=\lambda_{l,k}'(\Phi)\wt{\Psi}$.
As customary, it is convenient to consider $\wt{\Psi}$ as the pull-back of a transformation $\Psi$ defined on $\omp$, that is  
$\wt{\Psi}=\Psi\circ\Phi$, and to express the formulas for the derivatives as volume or boundary integrals on $\omp$.  
At a formal level, assuming
the differentiability of the eigenvalues and eigenvectors with respect to $\Phi$
(which might fail for multiple eigenvalues, cf.~Part II), the Hellman-Feynman Theorem  
allows to obtain the formulas for $\p_{\wt{\Psi}}\lambda_{l,k}(\Phi)$ by differentiating the Rayleigh quotients \eqref{introray} 
with respect to $\Phi$ (in direction $\wt{\Psi}$ and keeping fixed the eigenvectors involved).  By doing so, we obtain
\begin{subequations}
\label{eq:derivintro}
\begin{align}
\label{eq:deriv0intro}
-\frac{\p_{\wt{\Psi}}\lambda_{0,k}(\Phi)}{\lambda_{0,k}(\Phi)}
&=\bscp{(\symtr\na\Psi)H}{H}_{\L^{2}(\omp)}
+\bscp{(\div\Psi)u}{u}_{\L^{2}(\omp)},\\
\label{eq:deriv1intro}
\frac{\p_{\wt{\Psi}}\lambda_{1,k}(\Phi)}{\lambda_{1,k}(\Phi)}
&=\bscp{(\symtr\na\Psi)B}{B}_{\L^{2}(\omp)}
+\bscp{(\symtr\na\Psi)E}{E}_{\L^{2}(\omp)},
\end{align}
\end{subequations}
cf.~\eqref{eq:derevformal3} and Part II, where 
$$\symtr\na\Psi
:=2\sym\na\Psi-\tr\na\Psi
=\na\Psi+(\na\Psi)^{\top}-\div\Psi.$$ 
Then, under more regularity assumptions on eigenvectors, 
it is possible to integrate by parts and write these formulas 
by means of surface integrals as follows
\begin{subequations}
\label{eq:derivintrobis}
\begin{align}
\label{eq:deriv0introbis}
\frac{\p_{\wt{\Psi}}\lambda_{0,k}(\Phi)}{\lambda_{0,k}(\Phi)}
&=\int_{\ganp} \big(  |H|^2-|u|^2 \big)\Psi \cdot n\,  d\sigma 
-\int_{\gatp}\big(|H|^2-|u|^2\big)  \Psi \cdot n\, d\sigma,\\
\label{eq:deriv1introbis}
\frac{\p_{\wt{\Psi}}\lambda_{1,k}(\Phi)}{\lambda_{1,k}(\Phi)}
&=  \int_{\ganp}\big( |B|^2-|E|^2\big) \Psi \cdot n\,  d\sigma 
-\int_{\gatp}\big( |B|^2-|E|^2 \big)\Psi \cdot n\, d\sigma,
\end{align}
\end{subequations}
see Part II. 
These computations are quite involved.

Note that formula \eqref{eq:deriv0introbis} is well known at least for non-mixed boundary conditions, cf. e.g.~\cite{LLdC2004a, lalaneu},
and formula \eqref{eq:deriv1introbis} has  been recently proved in \cite{LZ2021a, LZ2022a} for sufficiently regular perturbations.  Formula~\eqref{eq:deriv1introbis} was found in a heuristic way in \cite{hira} with $\Gamma_n=\emptyset$  for arbitrary $\eps$ and $\mu$ and  another interesting equivalent  formula  was proved in \cite{jimbo}, see also~\cite{zac23}.

It is important to observe that the left and right  derivatives in \eqref{eq:intro-deriv} 
, i.e., $h\to0^{\mp}$,
coincide if the eigenvalue under consideration is simple, while  they might be different  if the eigenvalue is  multiple   (the difference corresponds to the choice of different eigenvectors in the formulas). 
This phenomenon  is well-known for many eigenvalue problems associated with families of self-adjoint operators depending on some parameters.  It is also well-known  that for perturbations  depending on one scalar parameter, it is possible to apply the Rellich theorem
and relabel the eigenvalues in order to guarantee their differentiability. On the other hand it was proved in \cite{LLdC2004a} that the elementary symmetric functions of the eigenvalues which bifurcate at a multiple eigenvalue are differentiable no matter whether the parameter involved is one dimensional or not.

\section{Preliminaries}
\label{presec}

\subsection{Sobolev Spaces and Boundary Conditions}

Let $k\in\nat_{0}\cup\{\infty\}$.
We define (for scalar, vector, or tensor fields)   
\begin{align*}
\C^{k}_{\gat}(\om)
&:=\big\{\psi|_{\om}\,:\,\psi\in\C^{k}(\rt),\;\supp\psi\text{ compact},\;\dist(\supp\psi,\gat)>0\big\},\\
\C^{0,1}_{\gat}(\om)
&:=\big\{\psi|_{\om}\,:\,\psi\in\C^{0,1}(\rt),\;\supp\psi\text{ compact},\;\dist(\supp\psi,\gat)>0\big\}.
\end{align*}
 Recall that $\Gamma_t$ is a relatively open subset of $\Gamma$.
Note that $\C^{k}_{\emptyset}(\om)$ and $\C^{0,1}_{\emptyset}(\om)$ 
are often denoted by $\C^{k}(\omb)$ and $\C^{0,1}(\omb)$, respectively.
With the Lebesgue space $\L^{2}(\om)$ we have the standard Sobolev spaces in the weak sense
\begin{align*}
\bH^{k}(\om)
&:=\big\{\psi\in\L^{2}(\om)\,:\,\p^{\alpha}\psi\in\L^{2}(\om)\;\;\forall\,|\alpha|\leq k\big\}\\
&\;=\big\{\psi\in\L^{2}(\om)\,:\,
\forall\,|\alpha|\leq k\;\;
\exists\Psi_{\alpha}\in\L^{2}(\om)\;\;
\forall\,\theta\in\C^{\infty}_{\ga}(\om)\\
&\hspace{100pt}
\scp{\psi}{\p^{\alpha}\theta}_{\L^{2}(\om)}=(-1)^{|\alpha|}\scp{\psi_{\alpha}}{\theta}_{\L^{2}(\om)}\big\},\\
\bR(\om)
&:=\big\{\Psi\in\L^{2}(\om)\,:\,\rot\Psi\in\L^{2}(\om)\}\\
&\;=\big\{\Psi\in\L^{2}(\om)\,:\,
\exists\Psi_{\rot}\in\L^{2}(\om)\;\;
\forall\,\Theta\in\C^{\infty}_{\ga}(\om)\quad
\scp{\Psi}{\rot\Theta}_{\L^{2}(\om)}=\scp{\Psi_{\rot}}{\Theta}_{\L^{2}(\om)}\big\},\\
\bD(\om)
&:=\big\{\Psi\in\L^{2}(\om)\,:\,\div\Psi\in\L^{2}(\om)\}\\
&\;=\big\{\Psi\in\L^{2}(\om)\,:\,
\exists\psi_{\div}\in\L^{2}(\om)\;\;
\forall\,\theta\in\C^{\infty}_{\ga}(\om)\quad
\scp{\Psi}{\na\theta}_{\L^{2}(\om)}=-\scp{\psi_{\div}}{\theta}_{\L^{2}(\om)}\big\}.
\end{align*}
Note that $\bR(\om)$ and $\bD(\om)$ are the well known spaces 
$\bH(\rot,\om)$ and $\bH(\div,\om)$, respectively.
We introduce boundary conditions in the strong sense by
\begin{align*}
\H^{k}_{\gat}(\om)&:=\ol{\C^{\infty}_{\gat}(\om)}^{\bH^{k}(\om)},
&
\R_{\gat}(\om)&:=\ol{\C^{\infty}_{\gat}(\om)}^{\bR(\om)},
&
\D_{\gat}(\om)&:=\ol{\C^{\infty}_{\gat}(\om)}^{\bD(\om)}.
\end{align*}
By standard Friedrichs' mollification we observe
\begin{align*}
\H^{1}_{\gat}(\om)&=\ol{\C^{0,1}_{\gat}(\om)}^{\bH^{1}(\om)},
&
\R_{\gat}(\om)&=\ol{\C^{0,1}_{\gat}(\om)}^{\bR(\om)},
&
\D_{\gat}(\om)&=\ol{\C^{0,1}_{\gat}(\om)}^{\bD(\om)}.
\end{align*}
Also boundary conditions in the weak sense are introduced by
\begin{align*}
\bH^{k}_{\gat}(\om)
&:=\big\{\psi\in\bH^{k}(\om)\,:\,
\forall\,|\alpha|\leq k\;\;
\forall\,\theta\in\C^{\infty}_{\gan}(\om)\quad
\scp{\psi}{\p^{\alpha}\theta}_{\L^{2}(\om)}=(-1)^{|\alpha|}\scp{\p^{\alpha}}{\theta}_{\L^{2}(\om)}\big\},\\
\bR_{\gat}(\om)
&:=\big\{\Psi\in\bR(\om)\,:\,
\forall\,\Theta\in\C^{\infty}_{\gan}(\om)\quad
\scp{\Psi}{\rot\Theta}_{\L^{2}(\om)}=\scp{\rot\Psi}{\Theta}_{\L^{2}(\om)}\big\},\\
\bD_{\gat}(\om)
&:=\big\{\Psi\in\bD(\om)\,:\,
\forall\,\theta\in\C^{\infty}_{\gan}(\om)\quad
\scp{\Psi}{\na\theta}_{\L^{2}(\om)}=-\scp{\div\Psi}{\theta}_{\L^{2}(\om)}\big\}.
\end{align*}
Note that for $\gat=\emptyset$ we have
$$\bH^{k}_{\emptyset}(\om)=\bH^{k}(\om),\qquad
\bR_{\emptyset}(\om)=\bR(\om),\qquad
\bD_{\emptyset}(\om)=\bD(\om).$$

\subsection{Weak and Strong Boundary Conditions Coincide}

For full boundary conditions there is a simple result
that `weak equals strong' holds without any additional assumptions
in a certain sense (the test fields can be chosen from a possibly larger space).
Unfortunately, the proof does not allow for mixed boundary conditions.

\begin{lem}[weak equals strong for full boundary conditions]
\label{lem:densefullbc1}
It holds
\begin{align*}
\H^{1}_{\ga}(\om)
&=\big\{\psi\in\bH^{1}(\om)\,:\,
\forall\,\Theta\in\bD(\om)\quad
\scp{\psi}{\div\Theta}_{\L^{2}(\om)}=-\scp{\na\psi}{\theta}_{\L^{2}(\om)}\big\},\\
\R_{\ga}(\om)
&=\big\{\Psi\in\bR(\om)\,:\,
\forall\,\Theta\in\bR(\om)\quad
\scp{\Psi}{\rot\Theta}_{\L^{2}(\om)}=\scp{\rot\Psi}{\Theta}_{\L^{2}(\om)}\big\},\\
\D_{\ga}(\om)
&=\big\{\Psi\in\bD(\om)\,:\,
\forall\,\theta\in\bH^{1}(\om)\quad
\scp{\Psi}{\na\theta}_{\L^{2}(\om)}=-\scp{\div\Psi}{\theta}_{\L^{2}(\om)}\big\}.
\end{align*}
\end{lem}

See the appendix for a proof.  For a definition of the segment property used in the following lemma we refer to \cite{adams}.

\begin{lem}[weak equals strong for no boundary conditions/density of smooth fields]
\label{lem:densenobc}
Let $\om$ have additionally the segment property. Then 
$$\H^{k}_{\emptyset}(\om)=\bH^{k}(\om),\qquad
\R_{\emptyset}(\om)=\bR(\om),\qquad
\D_{\emptyset}(\om)=\bD(\om).$$
In other words $\C^{\infty}_{\emptyset}(\om)=\C^{\infty}(\omb)$ is dense in 
$\bH^{k}(\om)$, $\bR(\om)$, and $\bD(\om)$, respectively.
\end{lem}

\begin{proof}
The proof for $\H^{1}(\om)$ can be found, e.g., in Agmon \cite{A1965a} or in Wloka \cite[Theorem 3.6]{W1982a},
and it literally carries over to $\R(\om)$ and $\D(\om)$
as the mollifiers work similarly for $\rot$ and $\div$.
The result for $\H^{k}(\om)$ follows by induction.
\end{proof}

In case of Lemma \ref{lem:densenobc}, we set 
$$\H^{k}(\om):=\H^{k}_{\emptyset}(\om)=\bH^{k}(\om),\qquad
\R(\om):=\R_{\emptyset}(\om)=\bR(\om),\qquad
\D(\om):=\D_{\emptyset}(\om)=\bD(\om).$$

\begin{lem}[weak equals strong for full boundary conditions]
\label{lem:densefullbc2}
Let $\om$ have additionally the segment property. Then
$$\H^{k}_{\ga}(\om)=\bH^{k}_{\ga}(\om),\qquad
\R_{\ga}(\om)=\bR_{\ga}(\om),\qquad
\D_{\ga}(\om)=\bD_{\ga}(\om).$$
\end{lem}

\begin{proof}
This follows by the same technique used in the proof of Lemma \ref{lem:densefullbc1}
in combination with the density results from Lemma \ref{lem:densenobc}, e.g.,
$\bR(\om)=\R(\om)=\R_{\emptyset}(\om)=\ol{\C^{\infty}_{\emptyset}(\om)}^{\R(\om)}$.
\end{proof}

For mixed boundary conditions, i.e., $\emptyset\neq\gat\neq\ga$,
the question `weak equals strong' is more delicate.  The equality can be proved under the assumption that  $\Omega$ has a Lipschitz boundary in the weak sense  and $\Gamma_t$ has a relative boundary in $\Gamma$ which is also Lipschitz in the weak sense. In particular, $\Gamma$ is a Lipschitz manifold of codimension one in $\mathbb{R}^3$ and the relative boundary of $\Gamma_t$ in $\Gamma$ is a Lipschitz submanifold  of codimension one in $\Gamma$. In this case we say that $(\om,\gat)$ is a weak Lipschitz pair. 
Recall that usually ``Lipschitz in the weak sense" means that the open set can be locally flattened near the boundary by means of a Lipschitz diffeomorphism. This  condition is weaker than  ``Lipschitz in the strong sense'' in which case the open set can be locally represented near boundary as a subgraph of a Lipschitz function.

A proof of the following lemma and the precise definition of a weak Lipschitz pair 
can be found in \cite{BPS2016a} or \cite{PS2022a}.

\begin{lem}[weak equals strong for mixed boundary consitions]
\label{lem:densebc}
Let $(\om,\gat)$ be additionally a weak Lipschitz pair. Then
$$\H^{k}_{\gat}(\om)=\bH^{k}_{\gat}(\om),\qquad
\R_{\gat}(\om)=\bR_{\gat}(\om),\qquad
\D_{\gat}(\om)=\bD_{\gat}(\om).$$
\end{lem}

\subsection{The Transformation Theorem}

Let $\Phi\in\C^{0,1}(\rt,\rt)$ be such that
$$\Phi:\om\to\Phi(\om)=\omp$$      
is bi-Lipschitz, and regular, i.e.,
$\Phi\in\C^{0,1}(\omb,\ombp)$ and $\Phi^{-1}\in\C^{0,1}(\ombp,\omb)$ with\footnote{Note that the Jacobian determinant of a bi-Lipschitz diffeomorphism has a constant sign on the connected components of the domain, see \cite[Lemma~6.7]{res}, hence it is not restrictive to assume that it is positive almost everywhere.} 
$$J_{\Phi}=\Phi'=(\na\Phi)^{\top},\qquad
{\rm ess}\inf\det J_{\Phi}>0.$$ 
Such regular bi-Lipschitz transformations will be called admissible and we write
$$\Phi\in\cL(\om).$$
For $\Phi\in\cL(\om)$ the inverse and adjunct matrix of $J_{\Phi}$ shall be denoted by
$$J_{\Phi}^{-1},\qquad
\adj J_{\Phi}:=(\det J_{\Phi})J_{\Phi}^{-1},$$
respectively.
We denote the composition with $\Phi$
by tilde, i.e., for any tensor field $\psi$ we define 
$$\wt{\psi}:=\psi\circ\Phi.$$
Moreover, let
$$\gap:=\Phi(\ga),\qquad
\gatp:=\Phi(\gat),\qquad
\ganp:=\Phi(\gan).$$

A proof of the following theorem for  differential forms 
can be found in the appendix of \cite{BPS2019a}. 
Here we focus on the special case used in this paper.

\begin{theo}[transformation theorem]   
\label{theo:transtheo}
Let $u\in\H^{1}_{\gatp}(\omp)$, $E\in\R_{\gatp}(\omp)$, and $H\in\D_{\gatp}(\omp)$.
Then 
\begin{align*}
\tau^{0}_{\Phi}u:=\wt{u}&\in\H^{1}_{\gat}(\om)
&&\text{and}&
\na\tau^{0}_{\Phi}u&=\tau^{1}_{\Phi}\na u,\\
\tau^{1}_{\Phi}E:=J_{\Phi}^{\top}\wt{E}&\in\R_{\gat}(\om)
&&\text{and}&
\rot\tau^{1}_{\Phi}E&=\tau^{2}_{\Phi}\rot E,\\
\tau^{2}_{\Phi}H:=(\adj J_{\Phi})\wt{H}&\in\D_{\gat}(\om)
&&\text{and}&
\div\tau^{2}_{\Phi}H&=\tau^{3}_{\Phi}\div H.
\end{align*}
with $\tau^{3}_{\Phi}f:=(\det J_{\Phi})\tau^{0}_{\Phi}f=(\det J_{\Phi})\wt{f}\in\L^{2}(\om)$
for $f\in\L^{2}(\omp)$
Moroever, 
\begin{align*}
\tau^{0}_{\Phi}:\H^{1}_{\gatp}(\omp)&\to\H^{1}_{\gat}(\om),
&
\tau^{1}_{\Phi}:\R_{\gatp}(\omp)&\to\R_{\gat}(\om),\\
\tau^{3}_{\Phi}:\L^{2}(\omp)&\to\L^{2}(\om),
&
\tau^{2}_{\Phi}:\D_{\gatp}(\omp)&\to\D_{\gat}(\om)
\end{align*}
are topological isomorphisms with norms depending on $\om$
and $J_{\Phi}$. The inverse operators and the $\L^{2}$-adjoints, 
i.e., the Hilbert space adjoints of 
$\tau^{q}_{\Phi}:\L^{2}(\omp)\to\L^{2}(\om)$, are given by
$$(\tau^{q}_{\Phi})^{-1}=\tau^{q}_{\Phi^{-1}},\qquad
(\tau^{0}_{\Phi})^{*}=\tau^{3}_{\Phi^{-1}},\qquad
(\tau^{1}_{\Phi})^{*}=\tau^{2}_{\Phi^{-1}},\qquad
(\tau^{2}_{\Phi})^{*}=\tau^{1}_{\Phi^{-1}},\qquad
(\tau^{3}_{\Phi})^{*}=\tau^{0}_{\Phi^{-1}},$$
respectively.
\end{theo}

\begin{proof} 
If $u\in\C^{0,1}_{\gatp}(\omp)$
we have by Rademacher's theorem 
$\wt{u}\in\C^{0,1}_{\gat}(\om)$
and the standard chain rule $(\wt{u})'=\wt{u'}\Phi'$ holds, i.e.,
\begin{align}
\label{transtheo:gradformulaone}
\na\wt{u}=\na\Phi\wt{\na u}=J_{\Phi}^{\top}\wt{\na u}.
\end{align}
Then we use an approximation argument.
For $u\in\H^{1}_{\gatp}(\omp)$ we pick a sequence $(u^{\ell})\subset\C^{0,1}_{\gatp}(\omp)$ 
such that $u^{\ell}\to E$ in $\H^{1}_{\gatp}(\omp)$. 
Then $\wt{u^{\ell}}\to\wt{E}$ and $\wt{\na u^{\ell}}\to\wt{\na u}$ in $\L^{2}(\om)$
by the standard transformation theorem. By \eqref{transtheo:gradformulaone} we have
$\wt{u^{\ell}}\in\C^{0,1}_{\gat}(\om)\subset\H^{1}_{\gat}(\om)$ with
$$\wt{u^{\ell}}\to\wt{u},\quad
\na\wt{u^{\ell}}=J_{\Phi}^{\top}\wt{\na u^{\ell}}\to J_{\Phi}^{\top}\wt{\na u}\qquad
\text{in }\L^{2}(\om).$$
Since $\na_{\gat}:\H^{1}_{\gat}(\om)\subset\L^{2}(\om)\to\L^{2}(\om)$ is closed,
we conclude $\wt{u}\in\H^{1}_{\gat}(\om)$ and 
\begin{align}
\label{transtheo:gradformulatwo}
\na\wt{u}=J_{\Phi}^{\top}\wt{\na u}.
\end{align}
For the classical chain rule in Sobolev spaces see, e.g., \cite{LLdC2004a}.

Assume $E\in\C^{0,1}_{\gatp}(\omp)$. 
Then $\wt{E}\in\C^{0,1}_{\gat}(\om)$ and
$$J_{\Phi}^{\top}\wt{E}=\na\Phi\wt{E}
=[\na\Phi_{1}\;\na\Phi_{2}\;\na\Phi_{3}]\wt{E}
=\sum_{j}\wt{E}_{j}\na\Phi_{j}.$$
As $\na\Phi_{j}\in\na\H^{1}(\om)\subset\R(\om)$ we conclude
$J_{\Phi}^{\top}\wt{E}\in\R(\om)$ and 
\begin{align}
\nonumber
\rot(J_{\Phi}^{\top}\wt{E})
&=\sum_{j}\na\wt{E}_{j}\times\na\Phi_{j}
=\sum_{j}(J_{\Phi}^{\top}\wt{\na E_{j}})\times\na\Phi_{j}\\
\label{transtheo:rotformulaLipschitz}
&=\sum_{j}\big([\na\Phi_{1}\;\na\Phi_{2}\;\na\Phi_{3}]\,\wt{\na E_{j}}\big)\times\na\Phi_{j}\\
\nonumber
&=\sum_{j,m}\wt{\p_{m}E_{j}}\na\Phi_{m}\times\na\Phi_{j}
=\sum_{j<m}(\wt{\p_{m}E_{j}}-\wt{\p_{j}E_{m}})\na\Phi_{m}\times\na\Phi_{j}\\
\nonumber
&=[\na\Phi_{2}\times\na\Phi_{3}\quad\na\Phi_{3}\times\na\Phi_{1}\quad\na\Phi_{1}\times\na\Phi_{2}]\,\wt{\rot E}
=(\adj J_{\Phi})\,\wt{\rot E}.
\end{align}
Moreover, by a mollification argument it follows that $J_{\Phi}^{\top}\wt{E}\in\R_{\gat}(\om)$.
Again, the general case $E\in\R_{\gatp}(\omp)$ can be treated by an approximation argument. 
For this, we consider a sequence $(E^{\ell})_{l\in \mathbb{N}}\subset\C^{0,1}_{\gatp}(\omp)$ 
such that $E^{\ell}\to E$ in $\R(\omp)$. 
Then $\wt{E^{\ell}}\to\wt{E}$ and $\wt{\rot E^{\ell}}\to\wt{\rot E}$ in $\L^{2}(\om)$.
Hence by the previous argument it follows that 
$J_{\Phi}^{\top}\wt{E^{\ell}}\in\R_{\gat}(\om)$ with
$$J_{\Phi}^{\top}\wt{E^{\ell}}\to J_{\Phi}^{\top}\wt{E},\quad
\rot(J_{\Phi}^{\top}\wt{E^{\ell}})=(\adj J_{\Phi})\,\wt{\rot E^{\ell}}\to(\adj J_{\Phi})\,\wt{\rot E}\qquad
\text{in }\L^{2}(\om).$$
Since $\rot_{\gat}:\R_{\gat}(\om)\subset\L^{2}(\om)\to\L^{2}(\om)$ is a closed operator,
we conclude $J_{\Phi}^{\top}\wt{E}\in\R_{\gat}(\om)$ and 
\begin{align}
\label{transtheo:rotformula}
\rot(J_{\Phi}^{\top}\wt{E})=(\adj J_{\Phi})\,\wt{\rot E},
\end{align}
which completes the proof of the transformation rule for $\tau^{1}_{\Phi}$. 

We now consider the case of $\tau^{2}_{\Phi}$.  
Assume $H\in\C^{0,1}_{\gatp}(\omp)$. 
Then $\wt{H}\in\C^{0,1}_{\gat}(\om)$ and
\begin{align*}
(\adj J_{\Phi})\wt{H}
&=[\na\Phi_{2}\times\na\Phi_{3}\quad\na\Phi_{3}\times\na\Phi_{1}\quad\na\Phi_{1}\times\na\Phi_{2}]\wt{H}
=\sum_{(j,m,l)}\wt{H}_{j}\na\Phi_{m}\times\na\Phi_{l},
\end{align*}
cf.~\eqref{transtheo:rotformulaLipschitz},
where the summation is over the three even permutations $(j,m,l)$ of $(1,2,3)$.
Since $\na\Phi_{m}\times\na\Phi_{l}=\rot(\Phi_{m}\na\Phi_{l})\in\rot\R(\om)\subset\D(\om)$ we conclude that
$(\adj J_{\Phi})\wt{H}\in\D(\om)$ and
\begin{align}
\nonumber
\div\big((\adj J_{\Phi})\wt{H}\big)
&=\sum_{(j,m,l)}\na\wt{H}_{n}\cdot(\na\Phi_{m}\times\na\Phi_{l})
=\sum_{(j,m,l)}(J_{\Phi}^{\top}\wt{\na H_{j}})\cdot(\na\Phi_{m}\times\na\Phi_{l})\\
\label{transtheo:divformulaLipschitz}
&=\sum_{(j,m,l)}\big([\na\Phi_{1}\;\na\Phi_{2}\;\na\Phi_{3}]\,\wt{\na H_{j}}\big)\cdot(\na\Phi_{m}\times\na\Phi_{l})\\
\nonumber
&=\sum_{(j,m,l,k)}\wt{\p_{k}H_{j}}\na\Phi_{k}\cdot(\na\Phi_{m}\times\na\Phi_{l})\\
\nonumber
&\overset{k=j}{=}(\det\na\Phi)\,\wt{\div H}
=(\det J_{\Phi})\,\wt{\div H}.
\end{align}
Moreover, by a mollification  argument we deduce that 
$(\adj J_{\Phi})\wt{H}\in\D_{\gat}(\om)$.
The general case $H\in\D_{\gatp}(\omp)$ can be discussed by an approximation argument as above. 
Consider a sequence $(H^{\ell})_{l\in \mathbb{N}}\subset\C^{0,1}_{\gatp}(\omp)$ 
such that $H^{\ell}\to H$ in $\D(\omp)$. 
Then $\wt{H^{\ell}}\to\wt{H}$ and $\wt{\div H^{\ell}}\to\wt{\div H}$ in $\L^{2}(\om)$. 
Hence by the previous argument we know that 
$(\adj J_{\Phi})\wt{H^{\ell}}\in\D_{\gat}(\om)$ with
$(\adj J_{\Phi})\wt{H^{\ell}}\to(\adj J_{\Phi})\wt{H}$ and 
$\div\big((\adj J_{\Phi})\wt{H^{\ell}}\big)=(\det J_{\Phi})\wt{\div H^{\ell}}\to(\det J_{\Phi})\wt{\div H}$
in $\L^{2}(\om)$.
Since $\div_{\gat}:\D_{\gat}(\om)\subset\L^{2}(\om)\to\L^{2}(\om)$ is a closed operator,
we conclude that $(\adj J_{\Phi})\wt{H}\in\D_{\gat}(\om)$ and 
$$\div\big((\adj J_{\Phi})\wt{H}\big)=(\det J_{\Phi})\wt{\div H},$$
which completes the proof of the transformation rule for  $\tau^{2}_{\Phi}$. 

Concerning the inverse operators and $\L^{2}$-adjoints we consider, e.g., 
the case $q=1$ since the other cases can be discussed in a similar way. 
As
$$\tau^{1}_{\Phi^{-1}}\tau^{1}_{\Phi}E
=\tau^{1}_{\Phi^{-1}}J_{\Phi}^{\top}\wt{E}
=J_{\Phi^{-1}}^{\top}\big((J_{\Phi}^{\top}\wt{E})\circ\Phi^{-1}\big)
=\big(J_{\Phi}^{-\top}J_{\Phi}^{\top}\wt{E})\circ\Phi^{-1}
=E$$
we have $(\tau^{1}_{\Phi})^{-1}=\tau^{1}_{\Phi^{-1}}$. 
Moreover, observing that $J_{\Phi^{-1}}=J_{\Phi}^{-1}\circ\Phi^{-1}$ we get
\begin{align*}
\scp{\tau^{1}_{\Phi}E}{\Psi}_{\L^{2}(\om)}
=\scp{J_{\Phi}^{\top}\wt{E}}{\Psi}_{\L^{2}(\om)}
&=\bscp{E}{(\det J_{\Phi^{-1}})(J_{\Phi}\Psi)\circ\Phi^{-1}}_{\L^{2}(\omp)}\\
&=\bscp{E}{(\det J_{\Phi^{-1}})J_{\Phi^{-1}}^{-1}(\Psi\circ\Phi^{-1})}_{\L^{2}(\omp)}\\
&=\bscp{E}{(\adj J_{\Phi^{-1}})(\Psi\circ\Phi^{-1})}_{\L^{2}(\omp)}
=\scp{E}{\tau^{2}_{\Phi^{-1}}\Psi}_{\L^{2}(\omp)}
\end{align*}
and hence $(\tau^{1}_{\Phi})^{*}=\tau^{2}_{\Phi^{-1}}$.
\end{proof}

\begin{rem}[transformation theorem]
\label{transtheo:rem2}
For the divergence there is also a duality argument 
leading to the result of Theorem \ref{theo:transtheo}.
For this, let $H\in\D(\omp)$ and pick some $\psi\in\C_{\ga}^{0,1}(\om)$.
Then
$\phi:=\psi\circ\Phi^{-1}\in\C_{\ga , \Phi}^{0,1}(\omp)$ and $\wt{\phi}=\psi$.
By the chain rule we compute
\begin{align*}
&\qquad\scp{H}{\na\phi}_{\L^{2}(\omp)}
=-\scp{\div H}{\phi}_{\L^{2}(\omp)}
=-\bscp{(\det J_{\Phi})\wt{\div H}}{\psi}_{\L^{2}(\om)}\\
&=\bscp{(\det J_{\Phi})\wt{H}}{\wt{\na\phi}}_{\L^{2}(\om)}
=\bscp{(\det J_{\Phi})\wt{H}}{J_{\Phi}^{-\top}\na\wt{\phi}}_{\L^{2}(\om)}
=\bscp{(\adj J_{\Phi})\wt{H}}{\na\psi}_{\L^{2}(\om)}.
\end{align*}
Hence, $(\adj J_{\Phi})\wt{H}\in\D(\om)$ and
$\div\big((\adj J_{\Phi})\wt{H}\big)=(\det J_{\Phi})\wt{\div H}$.
Note that this duality argument does not apply for the $\rot$ operator.
\end{rem}

\begin{cor}[transformation theorem]
\label{cor:transtheo}
Let $E\in\R_{\gatp}(\omp)\cap\eps^{-1}\D_{\ganp}(\omp)$. Then 
$$\tau^{1}_{\Phi}E\in\R_{\gat}(\om)\cap\eps_{\Phi}^{-1}\D_{\gan}(\om)$$
and it holds
$$\rot\tau^{1}_{\Phi}E=\tau^{2}_{\Phi}\rot E,\qquad
\div\eps_{\Phi}\tau^{1}_{\Phi}E=\tau^{3}_{\Phi}\div\eps E,\qquad
\eps_{\Phi}\tau^{1}_{\Phi}
=\tau^{2}_{\Phi}\eps$$
with 
$\eps_{\Phi}
:=\tau^{2}_{\Phi}\eps\tau^{1}_{\Phi^{-1}}
=(\det J_{\Phi})J_{\Phi}^{-1}\wt{\eps}J_{\Phi}^{-\top}
=(\adj J_{\Phi})\wt{\eps}J_{\Phi}^{-\top}$.
Moroever, 
$$\tau^{1}_{\Phi}:\R_{\gatp}(\omp)\cap\eps^{-1}\D_{\ganp}(\omp)
\to\R_{\gat}(\om)\cap\eps_{\Phi}^{-1}\D_{\gan}(\om)$$
is a topological isomorphism with norm depending on $\om$, $\eps$,
and $J_{\Phi}$. The inverse is given by $\tau^{1}_{\Phi^{-1}}$.
\end{cor}

\begin{proof}
Using Theorem \ref{theo:transtheo} we compute for $\eps E\in\D_{\ganp}(\omp)$
$$\tau^{3}_{\Phi}\div\eps E
=\div\tau^{2}_{\Phi}\eps E
=\div\tau^{2}_{\Phi}\eps\tau^{1}_{\Phi^{-1}}\tau^{1}_{\Phi}E
=\div\eps_{\Phi}\tau^{1}_{\Phi}E$$
with 
$\eps_{\Phi}
=\tau^{2}_{\Phi}\eps\tau^{1}_{\Phi^{-1}}
=(\adj J_{\Phi})\wt{\eps\tau^{1}_{\Phi^{-1}}}
=(\adj J_{\Phi})\wt{\eps}J_{\Phi}^{-\top}
=(\det J_{\Phi})J_{\Phi}^{-1}\wt{\eps}J_{\Phi}^{-\top}$.
\end{proof}

\begin{rem}[transformation theorem]
\label{rem:transtheo1}
More explicitly,
in Theorem \ref{theo:transtheo} and Corollary \ref{cor:transtheo} it holds
\begin{align*}
\forall\,u&\in\H^{1}_{\gatp}(\omp)
&
\na\wt{u}
&=J_{\Phi}^{\top}\wt{\na u},\\
\forall\,E&\in\R_{\gatp}(\omp)
&
\rot(J_{\Phi}^{\top}\wt{E})
&=(\adj J_{\Phi})\,\wt{\rot E},\\
\forall\,H&\in\D_{\gatp}(\omp)
&
\div\big((\adj J_{\Phi})\,\wt{H}\big)
&=(\det J_{\Phi})\,\wt{\div H},\\
\forall\,E&\in\eps^{-1}\D_{\ganp}(\omp)
&
\div(\eps_{\Phi}J_{\Phi}^{\top}\wt{E})
&=(\det J_{\Phi})\,\wt{\div\eps E}.
\end{align*}
\end{rem}

\begin{rem}[transformation theorem]
\label{rem:transtheo2}
The transformations $\tau^{q}_{\Phi}$ are just 
the well known pullback maps for differential $q$-forms
applied to the corresponding vector proxies, i.e.,
$\tau^{q}_{\Phi}=\Phi^{*}$ on differentials forms $F$ of degree $q$.
Using the exterior derivative $\ed$ the latter formulas reduce to
$$\ed\tau^{q}_{\Phi}F=\ed\Phi^{*}F=\Phi^{*}\ed F=\tau^{q+1}_{\Phi}\ed F.$$
\end{rem}

\subsection{Functional Analysis Toolbox}

We collect and cite some parts from
\cite{P2017a,P2019b,P2019a,P2020a,PS2022a,PW2020a},
cf.~\cite{PS2022b,PS2022d,PW2021a,PZ2020a,PZ2022a},
of the so-called functional analysis toolbox (FA-ToolBox).

\subsubsection{Single Operators and Hilbert Space Adjoints}
\label{singleoperators}

Let $\A:D(\A)\subset\H_{0}\to\H_{1}$ be a densely defined and closed 
(unbounded\footnote{The related bounded linear operator, where the domain  $D(\A)$ is endowed with the graph norm,
shall be denoted by $\A:D(\A)\to\H_{1}$.}) 
linear operator
with domain of definition $D(\A)$ on two Hilbert spaces $\H_{0}$ and $\H_{1}$. 
Then the Hilbert space adjoint $\A^{*}:D(\A^{*})\subset\H_{1}\to\H_{0}$
is well defined and characterised by
$$\forall\,x\in D(\A)\quad
\forall\,y\in D(\A^{*})\qquad
\scp{\A x}{y}_{\H_{1}}=\scp{x}{\A^{*} y}_{\H_{0}}.$$
The operators $\A$ and $\A^{*}$ are both densely defined, closed, and typically unbounded.
We call $(\A,\A^{*})$ a dual pair as $(\A^{*})^{*}=\ol{\A}=\A$.
The projection theorem shows
\begin{align}
\label{helm}
\H_{0}=N(\A)\oplus_{\H_{0}}\ol{R(\A^{*})},\quad
\H_{1}=N(\A^{*})\oplus_{\H_{1}}\ol{R(\A)},
\end{align}
often called Helmholtz/Hodge/Weyl decompositions,
where we introduce the notation $N$ for the kernel (or null space)
and $R$ for the range of a linear operator.
These orthogonal decompositions reduce the operators $\A$ and $\A^{*}$,
leading to the injective operators $\cA:=\A|_{\ol{R(\A^{*})}}$ and $\cA^{*}:=\A^{*}|_{\ol{R(\A)}}$, i.e.
\begin{align*}
\cA:D(\cA)\subset\ol{R(\A^{*})}&\to\ol{R(\A)},
&
D(\cA)&=D(\cA)\cap\ol{R(\A^{*})},\\
\cA^{*}:D(\cA^{*})\subset\ol{R(\A)}&\to\ol{R(\A^{*})},
&
D(\cA^{*})&=D(\cA^{*})\cap\ol{R(\A)},
\end{align*}
which are again densely defined and closed (unbounded) linear operators.
Note that 
$$\ol{R(\A^{*})}=N(\A)^{\bot_{\H_{0}}},\quad
\ol{R(\A)}=N(\A^{*})^{\bot_{\H_{1}}}$$
and that $\cA$ and $\cA^{*}$ are indeed adjoint to each other, i.e.,
$(\cA,\cA^{*})$ is a dual pair as well. Then the inverse operators 
$$\cA^{-1}:R(\A)\to D(\cA),\quad
(\cA^{*})^{-1}:R(\A^{*})\to D(\cA^{*})$$
are well defined and bijective, but possibly unbounded.
Furthermore, by \eqref{helm} we have the refined Helmholtz type decompositions
\begin{align}
\label{helmdecoAcA}
D(\A)=N(\A)\oplus_{\H_{0}}D(\cA),\quad
D(\A^{*})=N(\A^{*})\oplus_{\H_{1}}D(\cA^{*})
\end{align}
and thus we obtain for the ranges
$$R(\A)=R(\cA),\quad
R(\A^{*})=R(\cA^{*}).$$
Note that $D(\A),\,D(\cA)$ and $D(\A^{*}),\,D(\cA^{*})$ equipped with the 
respective graph norms are Hilbert spaces.

The following result is a well known and direct consequence of
the closed graph theorem and the closed range theorem.

\begin{lem}[fa-toolbox lemma 1]
\label{fatbl1}
The following assertions are equivalent:
\begin{itemize}
\item[\bf(i)] 
$\exists\,c_{\A}\in(0,\infty)$ \quad 
$\forall\,x\in D(\cA)$ \qquad
$\norm{x}_{\H_{0}}\leq c_{\A}\norm{\A x}_{\H_{1}}$
\item[\bf(i${}^{*}$)] 
$\exists\,c_{\A^{*}}\in(0,\infty)$ \quad 
$\forall\,y\in D(\cA^{*})$ \qquad
$\norm{y}_{\H_{1}}\leq c_{\A^{*}}\norm{\A^{*} y}_{\H_{0}}$
\item[\bf(ii)] 
$R(\A)=R(\cA)$ is closed in $\H_{1}$.
\item[\bf(ii${}^{*}$)] 
$R(\A^{*})=R(\cA^{*})$ is closed in $\H_{0}$.
\item[\bf(iii)] 
$\cA^{-1}:R(\A)\to D(\cA)$ is bounded.
\item[\bf(iii${}^{*}$)] 
$(\cA^{*})^{-1}:R(\A^{*})\to D(\cA^{*})$ is bounded.
\item[\bf(iv)] 
$\cA:D(\cA)\to R(\A)$ is a topological isomorphism.
\item[\bf(iv${}^{*}$)] 
$\cA^{*}:D(\cA^{*})\to R(\A^{*})$ is a topological isomorphism.
\end{itemize}
\end{lem}

The latter inequalities will be called Friedrichs-Poincar\'e type estimates.

\begin{lem}[fa-toolbox lemma 2]
\label{fatbl2}
The following assertions are equivalent:
\begin{itemize}
\item[\bf(i)]
$D(\cA)\incl\H_{0}$ is compact.
\item[\bf(i${}^{*}$)]
$D(\cA^{*})\incl\H_{1}$ is compact.
\item[\bf(ii)]
$\cA^{-1}:R(\A)\to R(\A^{*})$ is compact.
\item[\bf(ii${}^{*}$)]
$(\cA^{*})^{-1}:R(\A^{*})\to R(\A)$ is compact.
\end{itemize}
\end{lem}

\begin{rem}[sufficient assumptions for the first fa-toolbox lemmas]
\label{fatbr}
\mbox{}
\begin{itemize}
\item[\bf(i)]
If $R(\A)$ is closed, then the assertions of Lemma \ref{fatbl1} hold.
\item[\bf(ii)]
If $D(\cA)\incl\H_{0}$ is compact,
then the assertions of Lemma \ref{fatbl1} (and Lemma \ref{fatbl2}) hold.
In particular, the Friedrichs-Poincar\'e type estimates hold,
all ranges are closed and the inverse operators are compact.
\end{itemize}
\end{rem}

\subsubsection{Spectra and Point Spectra}
\label{spectra}

We emphasise that
\begin{align}
\label{eq:saops}
\A^{*}\A\geq0,\qquad
\A\A^{*}\geq0
\end{align}
are self-adjoint with essentially (except of $0$) the same non-negative spectrum.
The same holds true for the reduced operators $\cA^{*}\cA,\cA\cA^{*}>0$.
We shall give more details for the point spectrum in the next lemma.

\begin{lem}[fa-toolbox lemma 3/eigenvalues]
\label{lemconstev}
Let $D(\cA)\incl\H_{0}$ be compact. 
Then the operators in \eqref{eq:saops} are self-adjoint, non-negative, and have pure and discrete point spectra
with no accumulation point in $\reals$. Moreover,
$$\sigma(\cA^{*}\cA)
=\sigma(\A^{*}\A)\setminus\{0\}
=\sigma(\A\A^{*})\setminus\{0\}
=\sigma(\cA\cA^{*})
=\{\lambda_{k}\}_{k\in\nat}
\subset(0,\infty)$$
with eigenvalues 
$0<\lambda_{1}\leq\lambda_{2}\leq\dots\leq\lambda_{k-1}\leq\lambda_{k}\leq\dots\to\infty$.
Only finitely many eigenvalues coincide, they are repeated according  their multiplicity, 
and it holds
$$N(\A^{*}\A-\lambda_{k})=N(\cA^{*}\cA-\lambda_{k}),\qquad
N(\A\A^{*}-\lambda_{k})=N(\cA\cA^{*}-\lambda_{k}).$$
\end{lem}


\begin{rem}[variational formulations]
\label{remconstev}
For any eigenvector $x$ of $\A^{*}\A$ associated with  an eigenvalue $\lambda_{k}$ we have
$$(\A^{*}\A-\lambda_{k})x=0,\qquad
x\in D(\A^{*}\A)\cap R(\A^{*})=D(\cA^{*}\cA)\subset D(\cA),$$
and the variational formulation
$$\forall\,\phi\in D(\A)\qquad
\scp{\A x}{\A\phi}_{\H_{1}}
=\lambda_{k}\scp{x}{\phi}_{\H_{0}}$$
holds. The corresponding results hold for any eigenvector $y$ of $\A\A^{*}$ to $\lambda_{k}$.
Note that, e.g., $y=\A x$.
\end{rem}

For $x\in N(\A^{*}\A-\lambda_{k})$
and $y:=\lambda_{k}^{-1/2}\A x$
we observe that 
$$x
=\lambda_{k}^{-1}\A^{*}\A x
=\lambda_{k}^{-1/2}\A^{*}y,\qquad
\A\A^{*}y
=\lambda_{k}^{1/2}\A x
=\lambda_{k}y$$
and 
$$\norm{y}_{\H_{1}}^{2}
=\lambda_{k}^{-1}\norm{\A x}_{\H_{1}}^{2}
=\lambda_{k}^{-1}\bscp{\A^{*}\A x}{x}_{\H_{0}}
=\norm{x}_{\H_{0}}^{2},$$
which shows the following: 

\begin{lem}[eigenvectors]
\label{lem:eigenvec1}
The following statements hold:
\begin{itemize}
\item[\bf(i)]
If $x$ is an eigenvector of $\A^{*}\A$ for the eigenvalue $\lambda_{k}$, 
then $y:=\A x$ is an eigenvector of $\A\A^{*}$ for the same eigenvalue $\lambda_{k}$.
\item[\bf(i${}^{*}$)]
If $y$ is an eigenvector of $\A\A^{*}$ for the eigenvalue $\lambda_{k}$, 
then $x:=\A^{*}y$ is an eigenvector of $\A^{*}\A$ for the same eigenvalue $\lambda_{k}$.
\end{itemize}
\end{lem}

\begin{lem}[eigenvalues, Friedrichs-Poincar\'e type constants, and Rayleigh quotients]
\label{lem:eigenval2}
The `best' constants in Lemma \ref{fatbl1} (i) and (i${}^{*}$)
are given by the Rayleigh quotients and equal each other 
and the inverse of the square root the first positive eigenvalue of $\A^{*}\A$ and $\A\A^{*}$, i.e.,
$$\lambda_{1}^{1/2}
=\frac{1}{c_{\A}}
=\inf_{0\neq x\in D(\cA)}\frac{\norm{\A x}_{\H_{1}}}{\norm{x}_{\H_{0}}}
=\inf_{0\neq y\in D(\cA^{*})}\frac{\norm{\A^{*} y}_{\H_{0}}}{\norm{y}_{\H_{1}}}
=\frac{1}{c_{\A^{*}}}.$$
Note that similar formulas hold for all eigenvalues, i.e., 
$$\lambda_{k}^{1/2}
=\inf_{x}\frac{\norm{\A x}_{\H_{1}}}{\norm{x}_{\H_{0}}}
=\inf_{y}\frac{\norm{\A^{*} y}_{\H_{0}}}{\norm{y}_{\H_{1}}},$$
where the infima are taken over all 
$0\neq x\in D(\cA)$ and $0\neq y\in D(\cA^{*})$ with 
$\displaystyle x\bot_{\H_{0}}\bigoplus_{\ell=1}^{k-1}N(\A^{*}\A-\lambda_{\ell})$
and $\displaystyle y\bot_{\H_{1}}\bigoplus_{\ell=1}^{k-1}N(\A\A^{*}-\lambda_{\ell})$.
All infima are minima and are attained at the corresponding eigenvectors, i.e.,
for all $k$ and all eigenvectors $x_{k}\in N(\A^{*}\A-\lambda_{k})$ and $y_{k}\in N(\A\A^{*}-\lambda_{k})$
we have 
$$\frac{\scp{\A^{*}\A x_{k}}{x_{k}}_{\H_{0}}}{\norm{x_{k}}_{\H_{0}}^{2}}
=\frac{\norm{\A x_{k}}_{\H_{1}}^{2}}{\norm{x_{k}}_{\H_{0}}^{2}}
=\lambda_{k}
=\frac{\norm{\A^{*} y_{k}}_{\H_{0}}^{2}}{\norm{y_{k}}_{\H_{1}}^{2}}
=\frac{\scp{\A\A^{*}y_{k}}{y_{k}}_{\H_{1}}}{\norm{y_{k}}_{\H_{1}}^{2}}.$$
\end{lem}

\subsubsection{Hilbert Complexes}

Now, let 
$$\A_{0}\!:\!D(\A_{0})\subset\H_{0}\to\H_{1},\quad
\A_{1}\!:\!D(\A_{1})\subset\H_{1}\to\H_{2}$$
be two densely defined and closed linear operators 
on three Hilbert spaces $\H_{0}$, $\H_{1}$, and $\H_{2}$ with adjoints 
$$\A_{0}^{*}\!:\!D(\A_{0}^{*})\subset\H_{1}\to\H_{0},\quad
\A_{1}^{*}\!:\!D(\A_{1}^{*})\subset\H_{2}\to\H_{1}$$
as well as reduced operators $\cA_{0}$, $\cA_{0}^{*}$, and $\cA_{1}$, $\cA_{1}^{*}$.
Furthermore, we assume the complex property
of $\A_{0}$ and $\A_{1}$, that is $\A_{1}\A_{0}=0$, i.e.,
\begin{align}
\label{sequenceprop}
R(\A_{0})\subset N(\A_{1}),
\end{align}
being equivalent to $R(\A_{1}^{*})\subset N(\A_{0}^{*})$.
Recall that
$$R(\A_{0})=R(\cA_{0}),\quad
R(\A_{0}^{*})=R(\cA_{0}^{*}),\quad
R(\A_{1})=R(\cA_{1}),\quad
R(\A_{1}^{*})=R(\cA_{1}^{*}).$$
From the Helmholtz type decompositions \eqref{helm} for $\A=\A_{0}$ and $\A=\A_{1}$ we get in particular
\begin{align}
\label{helmappclone}
\H_{1}=\ol{R(\A_{0})}\oplus_{\H_{1}}N(\A_{0}^{*}),\quad
\H_{1}=\ol{R(\A_{1}^{*})}\oplus_{\H_{1}}N(\A_{1}).
\end{align}
Introducing the cohomology group 
$$N_{0,1}:=N(\A_{1})\cap N(\A_{0}^{*}),$$
we obtain the refined Helmholtz type decompositions 
\begin{align}
\label{helmrefinedone}
\begin{aligned}
N(\A_{1})&=\ol{R(\A_{0})}\oplus_{\H_{1}}N_{0,1},
&
N(\A_{0}^{*})&=\ol{R(\A_{1}^{*})}\oplus_{\H_{1}}N_{0,1},\\
D(\A_{1})&=\ol{R(\A_{0})}\oplus_{\H_{1}}\big(D(\A_{1})\cap N(\A_{0}^{*})\big),\quad
&
D(\A_{0}^{*})&=\ol{R(\A_{1}^{*})}\oplus_{\H_{1}}\big(D(\A_{0}^{*})\cap N(\A_{1})\big),
\end{aligned}
\end{align}
and therefore the Helmholtz type decomposition
\begin{align}
\label{helmrefinedtwo}
\H_{1}&=\ol{R(\A_{0})}\oplus_{\H_{1}}N_{0,1}\oplus_{\H_{1}}\ol{R(\A_{1}^{*})}
\end{align}
follows. Let us remark that the first line of \eqref{helmrefinedone} can also be written as
$$\ol{R(\A_{0})}=N(\A_{1})\cap N_{0,1}^{\bot_{\H_{1}}},\quad
\ol{R(\A_{1}^{*})}=N(\A_{0}^{*})\cap N_{0,1}^{\bot_{\H_{1}}}.$$
Note that \eqref{helmrefinedtwo} can be further refined and specialised, e.g., to
\begin{align}
\begin{split}
\label{helmrefinedthree}
D(\A_{1})&=\ol{R(\A_{0})}\oplus_{\H_{1}}N_{0,1}\oplus_{\H_{1}}D(\cA_{1}),\\
D(\A_{0}^{*})&=D(\cA_{0}^{*})\oplus_{\H_{1}}N_{0,1}\oplus_{\H_{1}}\ol{R(\A_{1}^{*})},\\
D(\A_{1})\cap D(\A_{0}^{*})&=D(\cA_{0}^{*})\oplus_{\H_{1}}N_{0,1}\oplus_{\H_{1}}D(\cA_{1}).
\end{split}
\end{align}
We observe
\begin{align*}
D(\cA_{1})
&=D(\A_{1})\cap\ol{R(\A_{1}^{*})}
\subset D(\A_{1})\cap N(\A_{0}^{*})
\subset D(\A_{1})\cap D(\A_{0}^{*}),\\
D(\cA_{0}^{*})
&=D(\A_{0}^{*})\cap\ol{R(\A_{0})}
\subset D(\A_{0}^{*})\cap N(\A_{1})
\subset D(\A_{0}^{*})\cap D(\A_{1}),
\end{align*}
and using the refined Helmholtz type decompositions \eqref{helmrefinedtwo} and \eqref{helmrefinedthree}
as well as the results of Lemma~\ref{fatbl2} we immediately see:

\begin{lem}[fa-toolbox lemma 4/compact embeddings]
\label{compemblem}
The following assertions are equivalent: 
\begin{itemize}
\item[\bf(i)]
$D(\cA_{0})\incl\H_{0}$, $D(\cA_{1})\incl\H_{1}$,
and $N_{0,1}\incl\H_{1}$ are compact.
\item[\bf(ii)]
$D(\A_{1})\cap D(\A_{0}^{*})\incl\H_{1}$ is compact.
\end{itemize}
In this case, the cohomology group $N_{0,1}$ has finite dimension.
\end{lem}

We summarise:

\begin{lem}[fa-toolbox lemma 5]
\label{lem:compemblem}
Let the ranges $R(\A_{0})$ and $R(\A_{1})$ be closed.
Then $R(\A_{0}^{*})$ and $R(\A_{1}^{*})$ are also closed, and
the corresponding Friedrichs-Poincar\'e type estimates hold, i.e.,
there exists a positive constant $c$ such that
\begin{align*}
\forall\,z&\in D(\cA_{0})=D(\A_{0})\cap R(\A_{0}^{*})
&
\norm{z}_{\H_{0}}&\leq c\norm{\A_{0} z}_{\H_{1}},\\
\forall\,x&\in D(\cA_{0}^{*})=D(\A_{0}^{*})\cap R(\A_{0})=D(\A_{0}^{*})\cap N(\A_{1})\cap N_{0,1}^{\bot_{\H_{1}}}
&
\norm{x}_{\H_{1}}&\leq c\norm{\A_{0}^{*} x}_{\H_{0}},\\
\forall\,x&\in D(\cA_{1})=D(\A_{1})\cap R(\A_{1}^{*})=D(\A_{1})\cap N(\A_{0}^{*})\cap N_{0,1}^{\bot_{\H_{1}}}
&
\norm{x}_{\H_{1}}&\leq c\norm{\A_{1} x}_{\H_{2}},\\
\forall\,y&\in D(\cA_{1}^{*})=D(\A_{1}^{*})\cap R(\A_{1})
&
\norm{y}_{\H_{2}}&\leq c\norm{\A_{1}^{*} y}_{\H_{1}},
\end{align*}
and
$$\forall\,x\in D(\A_{1})\cap D(\A_{0}^{*})\cap N_{0,1}^{\bot_{\H_{1}}}\qquad
\norm{x}_{\H_{1}}
\leq c\big(\norm{\A_{1} x}_{\H_{2}}+\norm{\A_{0}^{*} x}_{\H_{0}}\big).$$
Moreover, all Helmholtz type decompositions \eqref{helmappclone}-\eqref{helmrefinedthree}
hold with closed ranges, in particular
$$\H_{1}=R(\A_{0})\oplus_{\H_{1}}N_{0,1}\oplus_{\H_{1}}R(\A_{1}^{*}).$$
\end{lem}

In other words, the primal and dual complex
\begin{equation}
\label{complex1}
\def\arrowlength{10ex}
\def\arrowdistance{.8}
\begin{tikzcd}[column sep=\arrowlength]
\H_{0}
\ar[r, rightarrow, shift left=\arrowdistance, "\A_{0}"] 
\ar[r, leftarrow, shift right=\arrowdistance, "\A_{0}^{*}"']
&
[-1em]
\H_{1}
\arrow[r, rightarrow, shift left=\arrowdistance, "\A_{1}"] 
\arrow[r, leftarrow, shift right=\arrowdistance, "\A_{1}^{*}"']
& 
[-1em]
\H_{2}
\end{tikzcd}
\end{equation}
is a Hilbert complex of closed and densely defined linear operators.
We call the complex \emph{closed} if the ranges $R(\A_{0})$ and $R(\A_{1})$ are closed.
The complex is \emph{exact} if $N_{0,1}=\{0\}$.
The complex is called \emph{compact}, if the embedding
\begin{align}
\label{cptembcrucial}
D(\A_{1})\cap D(\A_{0}^{*})\incl\H_{1}
\end{align} 
is compact.

\subsubsection{Generalised Laplacian}

Finally, we present some results for the densely defined and closed generalised Laplacian
$$\rho\A_{0}\A_{0}^{*}+\A_{1}^{*}\A_{1}:
D(\tau\A_{0}\A_{0}^{*}+\A_{1}^{*}\A_{1})\subset\H_{1}\to\H_{1},\qquad
\rho>0,$$
with
$D(\rho\A_{0}\A_{0}^{*}+\A_{1}^{*}\A_{1})
:=D(\A_{1}^{*}\A_{1})\cap D(\A_{0}\A_{0}^{*})
\subset D(\A_{1})\cap D(\A_{0}^{*})$.

\begin{lem}[fa-toolbox lemma 6/eigenvalues]
\label{lemconstevcomplex}
Let $D(\A_{1})\cap D(\A_{0}^{*})\incl\H_{1}$ be compact. Then
$\rho\A_{0}\A_{0}^{*}+\A_{1}^{*}\A_{1}$
is self-adjoint, non-negative, and has pure and discrete point spectrum with no accumulation point, i.e.,
$$\sigma(\rho\A_{0}\A_{0}^{*}+\A_{1}^{*}\A_{1})\setminus\{0\}
=\big(\rho\sigma(\A_{0}^{*}\A_{0})\setminus\{0\}\big)
\cup\big(\sigma(\A_{1}^{*}\A_{1})\setminus\{0\}\big)
=\rho\{\lambda_{0,k}\}_{k\in\nat}\cup\{\lambda_{1,k}\}_{k\in\nat}$$
with eigenvalues 
$0<\lambda_{\ell,1}\leq\lambda_{\ell,2}\leq\dots\leq\lambda_{\ell,k-1}\leq\lambda_{\ell,k}\leq\dots\to\infty$
of $\cA_{\ell}^{*}\cA_{\ell}$ for $\ell\in\{0,1\}$. 
Only finitely many eigenvalues coincide  and they are repeated according to their multiplicity.
Moreover, 
$$\rho\A_{0}\A_{0}^{*}+\A_{1}^{*}\A_{1}:
D(\rho\A_{0}\A_{0}^{*}+\A_{1}^{*}\A_{1})\cap N_{0,1}^{\bot_{\H_{1}}}\to N_{0,1}^{\bot_{\H_{1}}}$$
is a topological isomorphism.
\end{lem}

\begin{rem}[Helmholtz decomposition]
\label{lemconstevcomplexremhelm}
Let $\rho=1$. Then
$\A_{0}\A_{0}^{*}+\A_{1}^{*}\A_{1}$ provides the Helmholtz decomposition from Lemma \ref{lem:compemblem}.
To see this, let us denote the orthonormal projector onto the cohomology group $N_{0,1}$ by
$\pi_{N_{0,1}}:\H_{1}\to N_{0,1}$. Then, for $x\in\H_{1}$ we have
$(1-\pi_{N_{0,1}})x\in N_{0,1}^{\bot_{\H_{1}}}$ and
\begin{align*}
x&=\pi_{N_{0,1}}x+(1-\pi_{N_{0,1}})x\\
&=\pi_{N_{0,1}}x+(\A_{0}\A_{0}^{*}+\A_{1}^{*}\A_{1})(\A_{0}\A_{0}^{*}+\A_{1}^{*}\A_{1})^{-1}(1-\pi_{N_{0,1}})x
\in N_{0,1}\oplus_{\H_{1}}R(\A_{0})\oplus_{\H_{1}}R(\A_{1}^{*}).
\end{align*}
\end{rem}

\subsection{De Rham Complex}
\label{sec:derhamcomplex}

In this subsection, let additionally $(\om,\gat)$ be a weak Lipschitz pair. 
Let us consider the densely defined and closed (unbounded) linear operators
\begin{align*}
\A_{0}:=\na_{\gat}:\H^{1}_{\gat}(\om)\subset\L^{2}_{\nu}(\om)&\to\L^{2}_{\eps}(\om),\\
\A_{1}:=\mu^{-1}\rot_{\gat}:\R_{\gat}(\om)\subset\L^{2}_{\eps}(\om)&\to\L^{2}_{\mu}(\om),\\
\A_{2}:=\kappa^{-1}\div_{\gat}\mu:\mu^{-1}\D_{\gat}(\om)\subset\L^{2}_{\mu}(\om)&\to\L^{2}_{\kappa}(\om)
\intertext{together with their densely defined and closed (unbounded) adjoints}
\A_{0}^{*}=-\nu^{-1}\div_{\gan}\eps:\eps^{-1}\D_{\gan}(\om)\subset\L^{2}_{\eps}(\om)&\to\L^{2}_{\nu}(\om),\\
\A_{1}^{*}=\eps^{-1}\rot_{\gan}:\R_{\gan}(\om)\subset\L^{2}_{\mu}(\om)&\to\L^{2}_{\eps}(\om),\\
\A_{2}^{*}=-\na_{\gan}:\H^{1}_{\gan}(\om)\subset\L^{2}_{\kappa}(\om)&\to\L^{2}_{\mu}(\om),
\end{align*}
where we introduce the weighted Lebesgue space $\L^{2}_{\eps}(\om)$ 
as $\L^{2}(\om)$ equipped with the weighted and equivalent inner product
$$\scp{\,\cdot\,}{\,\cdot\,}_{\L^{2}_{\eps}(\om)}
:=\scp{\eps\,\cdot\,}{\,\cdot\,}_{\L^{2}(\om)}$$
(same for $\mu$, $\nu$, and $\kappa$).
For the adjoints we refer to \cite{BPS2016a} and \cite{PS2022a} (weak and strong boundary conditions coincide). 
Recall that $\A_{\ell}^{**}=\ol{\A_{\ell}}=\A_{\ell}$ and
that $(\A_{\ell},\A_{\ell}^{*})$ are dual pairs.

\begin{rem}
\label{rem:derham1}
The latter operators satisfy the complex properties 
$R(\A_{\ell})\subset N(\A_{\ell+1})$ and build the well known de Rham Hilbert complex
$$\small
\def\arrowlength{12ex}
\def\arrowdistance{.8}
\begin{tikzcd}[column sep=\arrowlength]
N(\A_{0})
\arrow[r, rightarrow, shift left=\arrowdistance, "\A_{-1}:=\iota_{N(\A_{0})}"] 
\arrow[r, leftarrow, shift right=\arrowdistance, "\A_{-1}^{*}=\pi_{N(\A_{0})}"']
& 
\L^{2}_{\nu}(\om)
\ar[r, rightarrow, shift left=\arrowdistance, "\A_{0}=\na_{\gat}"] 
\ar[r, leftarrow, shift right=\arrowdistance, "\A_{0}^{*}=-\nu^{-1}\div_{\gan}\eps"']
&
\L^{2}_{\eps}(\om)
\arrow[r, rightarrow, shift left=\arrowdistance, "\A_{1}=\mu^{-1}\rot_{\gat}"] 
\arrow[r, leftarrow, shift right=\arrowdistance, "\A_{1}^{*}=\eps^{-1}\rot_{\gan}"']
& 
\L^{2}_{\mu}(\om)
\arrow[r, rightarrow, shift left=\arrowdistance, "\A_{2}=\kappa^{-1}\div_{\gat}\mu"] 
\arrow[r, leftarrow, shift right=\arrowdistance, "\A_{2}^{*}=-\na_{\gan}"']
&
\L^{2}_{\kappa}(\om)
\arrow[r, rightarrow, shift left=\arrowdistance, "\A_{3}:=\pi_{N(\A_{2}^{*})}"] 
\arrow[r, leftarrow, shift right=\arrowdistance, "\A_{3}^{*}=\iota_{N(\A_{2}^{*})}"']
&
N(\A_{2}^{*}).
\end{tikzcd}$$
Here, $\iota$ and $\pi$ denote the cannocial embedding and the orthogonal projector,
respectively.
\end{rem}

\begin{theo}[Weck's selection theorem] 
\label{theo:maxcpt}
The embedding
$$D(\A_{1})\cap D(\A_{0}^{*})
=\R_{\gat}(\om)\cap\eps^{-1}\D_{\gan}(\om)
\incl\L^{2}_{\eps}(\om)=\H_{1}$$
is compact.
\end{theo}

A proof can be found in \cite{BPS2016a} and \cite{PS2022a}.

\begin{rem}[Weck's selection theorem]
\label{rem:maxcpttheo}
Note that by Theorem \ref{theo:maxcpt} also
$$D(\A_{2})\cap D(\A_{1}^{*})
=\mu^{-1}\D_{\gat}(\om)\cap\R_{\gan}(\om)
\incl\L^{2}_{\mu}(\om)=\H_{2}$$
is compact. Moreover, $D(\A_{0})=\H^{1}_{\gat}(\om)\incl\L^{2}_{\nu}(\om)=\H_{0}$
and $D(\A_{2}^{*})=\H^{1}_{\gan}(\om)\incl\L^{2}_{\kappa}(\om)=\H_{3}$ are trivially compact
by Rellich's selection theorem.
The first compact embedding result for non-smooth domains, i.e.,
for piecewise smooth and globally strong Lipschitz boundaries and full boundary conditions, 
was given by Weck \cite{W1974a}. 
First results for strong Lipschitz boundaries and full boundary conditions  
had been proved in \cite{P1984a,W1980a}. First results for strong Lipschitz boundaries 
and mixed boundary conditions can be found in \cite{FG1997a,J1997a}.
\end{rem}

Theorem \ref{theo:maxcpt} together with Lemma \ref{compemblem} shows that    
$D(\cA_{1})\incl\H_{1}$ and $D(\cA_{0}),D(\A_{0})\incl\H_{1}$  are compact.
Hence by Lemma \ref{lemconstev} we see that 
\begin{align}
\label{eq:AsAsa1}
\begin{aligned}
\A_{0}^{*}\A_{0}&=-\nu^{-1}\div_{\gan}\eps\na_{\gat},
&
\A_{1}^{*}\A_{1}&=\eps^{-1}\rot_{\gan}\mu^{-1}\rot_{\gat},\\
\A_{0}\A_{0}^{*}&=-\na_{\gat}\nu^{-1}\div_{\gan}\eps,
&
\A_{1}\A_{1}^{*}&=\mu^{-1}\rot_{\gat}\eps^{-1}\rot_{\gan}
\end{aligned}
\end{align}
are self-adjoint, non-negative, and
have pure and discrete point spectrum with no accumulation point. Moreover, as 
$$\sigma(\A_{\ell}^{*}\A_{\ell})\setminus\{0\}
=\sigma(\cA_{\ell}^{*}\cA_{\ell})=\sigma(\cA_{\ell}\cA_{\ell}^{*})
=\sigma(\A_{\ell}\A_{\ell}^{*})\setminus\{0\}\subset(0,\infty),\qquad
\ell\in\{0,1\},$$ 
we get:

\begin{theo}[eigenvalues of the de Rham complex]
\label{theo:evderham}
It holds
\begin{align*}
\sigma(\mu^{-1}\rot_{\gat}\eps^{-1}\rot_{\gan})\setminus\{0\}
=\sigma(\eps^{-1}\rot_{\gan}\mu^{-1}\rot_{\gat})\setminus\{0\}
&=\{\lambda_{1,k}\}_{k\in\nat}
\subset(0,\infty),\\
\sigma(-\na_{\gat}\nu^{-1}\div_{\gan}\eps)\setminus\{0\}
=\sigma(-\nu^{-1}\div_{\gan}\eps\na_{\gat})\setminus\{0\}
&=\{\lambda_{0,k}\}_{k\in\nat}
\subset(0,\infty)
\end{align*}
with eigenvalues 
$0<\lambda_{\ell,1}\leq\lambda_{\ell,2}\leq\dots\leq\lambda_{\ell,k-1}\leq\lambda_{\ell,k}\leq\dots\to\infty$
for $\ell\in\{0,1\}$.
Only finitely many eigenvalues coincide  and they are repeated according to their multiplicity.
\end{theo}

For the generalised Laplacian 
$$\rho\A_{0}\A_{0}^{*}+\A_{1}^{*}\A_{1}
=-\rho\na_{\gat}\nu^{-1}\div_{\gan}\eps
+\eps^{-1}\rot_{\gan}\mu^{-1}\rot_{\gat}$$
we have the following result:

\begin{theo}[eigenvalues of the generalised Laplacian]
\label{theo:evlap}
$\eps^{-1}\rot_{\gan}\mu^{-1}\rot_{\gat}
-\rho\na_{\gat}\nu^{-1}\div_{\gan}\eps$
is self-adjoint, non-negative, and has pure and discrete point spectrum with no accumulation point, i.e.,
\begin{align*}
&\qquad\sigma(-\rho\na_{\gat}\nu^{-1}\div_{\gan}\eps
+\eps^{-1}\rot_{\gan}\mu^{-1}\rot_{\gat})\setminus\{0\}\\
&=\big(\rho\sigma(-\na_{\gat}\nu^{-1}\div_{\gan}\eps)\setminus\{0\}\big)
\cup\sigma\big(\eps^{-1}\rot_{\gan}\mu^{-1}\rot_{\gat})\setminus\{0\}\big)
=\rho\{\lambda_{0,k}\}_{k\in\nat}\cup\{\lambda_{1,k}\}_{k\in\nat}. 
\end{align*}
Only finitely many eigenvalues coincide and they  are repeated according to their multiplicity. 
\end{theo}

\section{Eigenvalues}
\label{eigensec}

Let 
$$\Phi:\om\to\omp=\Phi(\om)$$ 
be a bi-Lipschitz transformation.
We are interested in the eigenvalue problems \eqref{eq:maxev1}-\eqref{eq:lapev3},
in particular, in the dependence of the eigenvalues 
and related symmetric functions on the domain $\om$,
more precisely on the domain $\omp$, when $\Phi$ is varying.
For this, we consider unbounded linear operators of the de Rham complex
in $\L^{2}(\omp)$ together with their Lipschitz transformed relatives in $\L^{2}(\om)$.

From now on, let additionally $(\om,\gat)$ be a weak Lipschitz pair.

\subsection{Operators of the De Rham Complex}
\label{sec:derhamop}

Let us define the densely defined and closed (unbounded) linear operators
\begin{align*}
\A_{0,\Phi}:=\na_{\gatp}:\H^{1}_{\gatp}(\omp)\subset\L^{2}_{\nu}(\omp)&\to\L^{2}_{\eps}(\omp),\\
\A_{1,\Phi}:=\mu^{-1}\rot_{\gatp}:\R_{\gatp}(\omp)\subset\L^{2}_{\eps}(\omp)&\to\L^{2}_{\mu}(\omp)
\intertext{together with their densely defined and closed (unbounded) adjoints}
\A_{0,\Phi}^{*}=-\nu^{-1}\div_{\ganp}\eps:\eps^{-1}\D_{\ganp}(\omp)\subset\L^{2}_{\eps}(\omp)&\to\L^{2}_{\nu}(\omp),\\
\A_{1,\Phi}^{*}=\eps^{-1}\rot_{\ganp}:\R_{\ganp}(\omp)\subset\L^{2}_{\mu}(\omp)&\to\L^{2}_{\eps}(\omp).
\end{align*}
Recall that $(\A_{\ell,\Phi},\A_{\ell,\Phi}^{*})$ are dual pairs. Moreover, let
\begin{align*}
\A_{0}:=\na_{\gat}:\H^{1}_{\gat}(\om)\subset\L^{2}_{\nu_{\Phi}}(\om)&\to\L^{2}_{\eps_{\Phi}}(\om),
&
\eps_{\Phi}&:=\tau^{2}_{\Phi}\eps\tau^{1}_{\Phi^{-1}}
=(\det J_{\Phi})J_{\Phi}^{-1}\wt{\eps}J_{\Phi}^{-\top},\\
\A_{1}:=\mu_{\Phi}^{-1}\rot_{\gat}:\R_{\gat}(\om)\subset\L^{2}_{\eps_{\Phi}}(\om)&\to\L^{2}_{\mu_{\Phi}}(\om),
&
\mu_{\Phi}&:=\tau^{2}_{\Phi}\mu\tau^{1}_{\Phi^{-1}}
=(\det J_{\Phi})J_{\Phi}^{-1}\wt{\mu}J_{\Phi}^{-\top},\\
\A_{0}^{*}=-\nu_{\Phi}^{-1}\div_{\gan}\eps_{\Phi}:\eps_{\Phi}^{-1}\D_{\gan}(\om)\subset\L^{2}_{\eps_{\Phi}}(\om)
&\to\L^{2}_{\nu_{\Phi}}(\om),
&
\nu_{\Phi}&:=\tau^{3}_{\Phi}\nu\tau^{0}_{\Phi^{-1}}
=(\det J_{\Phi})\wt{\nu},\\
\A_{1}^{*}=\eps_{\Phi}^{-1}\rot_{\gan}:\R_{\gan}(\om)\subset\L^{2}_{\mu_{\Phi}}(\om)
&\to\L^{2}_{\eps_{\Phi}}(\om),
\end{align*}
which are also densely defined and closed (unbounded) linear operators.
Again, $(\A_{\ell},\A_{\ell}^{*})$ are dual pairs.
Note that the inner products in the weighted Lebesgue spaces
$\L^{2}_{\eps_{\Phi}}(\om)$, $\L^{2}_{\mu_{\Phi}}(\om)$, and $\L^{2}_{\nu_{\Phi}}(\om)$
read explicitely 
\begin{align*}
\scp{\,\cdot\,}{\,\cdot\,}_{\L^{2}_{\eps_{\Phi}}(\om)}
&=\scp{\eps_{\Phi}\,\cdot\,}{\,\cdot\,}_{\L^{2}(\om)}
=\bscp{(\det J_{\Phi})\wt{\eps}J_{\Phi}^{-\top}\,\cdot\,}{J_{\Phi}^{-\top}\,\cdot\,}_{\L^{2}(\om)},\\
\scp{\,\cdot\,}{\,\cdot\,}_{\L^{2}_{\mu_{\Phi}}(\om)}
&=\scp{\mu_{\Phi}\,\cdot\,}{\,\cdot\,}_{\L^{2}(\om)}
=\bscp{(\det J_{\Phi})\wt{\mu}J_{\Phi}^{-\top}\,\cdot\,}{J_{\Phi}^{-\top}\,\cdot\,}_{\L^{2}(\om)},\\
\scp{\,\cdot\,}{\,\cdot\,}_{\L^{2}_{\nu_{\Phi}}(\om)}
&=\scp{\nu_{\Phi}\,\cdot\,}{\,\cdot\,}_{\L^{2}(\om)}
=\bscp{(\det J_{\Phi})\wt{\nu}\,\cdot\,}{\,\cdot\,}_{\L^{2}(\om)},
\end{align*}
and that $\eps_{\Phi}$, $\mu_{\Phi}$, and $\nu_{\Phi}$ are admissible transformations. 
Note that, here, the operators $\A_{2,\Phi}$ and $\A_{2}$ are not needed due to their equivalence 
to $\A_{0,\Phi}^{*}$ and $\A_{0}^{*}$.

\subsection{Unitary Equivalence and Spectrum}
\label{sec:uniequispectrum}

Utilising the pullbacks $\tau^{q}_{\Phi}$ of $\Phi:\om\to\Phi(\om)=\omp$
from Theorem \ref{theo:transtheo} and Corollay \ref{cor:transtheo} 
and the corresponding inverse pullbacks $(\tau^{q}_{\Phi})^{-1}=\tau^{q}_{\Phi^{-1}}$
we compute
\begin{align*}
\tau^{1}_{\Phi}\A_{0,\Phi}u
&=\tau^{1}_{\Phi}\na_{\gatp}u
=\na_{\gat}\tau^{0}_{\Phi}u
=\A_{0}\tau^{0}_{\Phi}u,\\
\tau^{1}_{\Phi}\A_{1,\Phi}E
&=\tau^{1}_{\Phi}\mu^{-1}\rot_{\gatp}E
=\tau^{1}_{\Phi}\mu^{-1}\tau^{2}_{\Phi^{-1}}\tau^{2}_{\Phi}\rot_{\gatp}E
=\mu_{\Phi}^{-1}\rot_{\gat}\tau^{1}_{\Phi}E
=\A_{1}\tau^{1}_{\Phi}E,\\
\tau^{0}_{\Phi}\A_{0,\Phi}^{*}H
&=-\tau^{0}_{\Phi}\nu^{-1}\tau^{3}_{\Phi^{-1}}\tau^{3}_{\Phi}\div_{\ganp}\eps H
=-\nu_{\Phi}^{-1}\div_{\gan}\tau^{2}_{\Phi}\eps H
=-\nu_{\Phi}^{-1}\div_{\gan}\eps_{\Phi}\tau^{1}_{\Phi}H
=\A_{0}^{*}\tau^{1}_{\Phi}H,
\end{align*}
and obtain by symmetry the following result.

\begin{lem}
\label{lem:unitequi1}
It holds
\begin{align*}
\A_{0,\Phi}
&=\tau^{1}_{\Phi^{-1}}\A_{0}\tau^{0}_{\Phi},
&
\A_{1,\Phi}
&=\tau^{1}_{\Phi^{-1}}\A_{1}\tau^{1}_{\Phi},\\
\A_{0,\Phi}^{*}
&=\tau^{0}_{\Phi^{-1}}\A_{0}^{*}\tau^{1}_{\Phi},
&
\A_{1,\Phi}^{*}
&=\tau^{1}_{\Phi^{-1}}\A_{1}^{*}\tau^{1}_{\Phi}.
\end{align*}
\end{lem}

\begin{rem}
\label{rem:unitequi1}
We emphasise that the adjoints of the pullbacks in Lemma \ref{lem:unitequi1} are given by
$$(\tau^{0}_{\Phi})^{*}=\tau^{0}_{\Phi^{-1}}=(\tau^{0}_{\Phi})^{-1},\qquad
(\tau^{1}_{\Phi})^{*}=\tau^{1}_{\Phi^{-1}}=(\tau^{1}_{\Phi})^{-1},$$
apparently in contradiction to the adjoints from Theorem \ref{theo:transtheo}.
This is due to the formulations in weighted $\L^{2}$-spaces.
Note that, e.g., we have to consider 
$\tau^{1}_{\Phi}:\L^{2}_{\eps}(\omp)\to\L^{2}_{\eps_{\Phi}}(\om)$,
which leads with $\tau^{2}_{\Phi^{-1}}\eps_{\Phi}=\eps\tau^{1}_{\Phi^{-1}}$ to
\begin{align*}
\scp{\tau^{1}_{\Phi}E}{\Psi}_{\L^{2}_{\eps_{\Phi}}(\om)}
&=\scp{\tau^{1}_{\Phi}E}{\eps_{\Phi}\Psi}_{\L^{2}(\om)}
=\scp{E}{\tau^{2}_{\Phi^{-1}}\eps_{\Phi}\Psi}_{\L^{2}(\omp)}
=\scp{E}{\tau^{1}_{\Phi^{-1}}\Psi}_{\L^{2}_{\eps}(\omp)},
\end{align*}
i.e., $(\tau^{1}_{\Phi})^{*}=\tau^{1}_{\Phi^{-1}}$.
Analogously, we treat $\tau^{0}_{\Phi}:\L^{2}_{\nu}(\omp)\to\L^{2}_{\nu_{\Phi}}(\om)$.
\end{rem}

\begin{theo}
\label{theo:unitequi2}
$\A_{\ell,\Phi}^{*}\A_{\ell,\Phi}$ and $\A_{\ell,\Phi}\A_{\ell,\Phi}^{*}$
are unitarily equivalent to $\A_{\ell}^{*}\A_{\ell}$ and $\A_{\ell}\A_{\ell}^{*}$,
respectively. More precisely, 
\begin{align*}
\A_{0,\Phi}^{*}\A_{0,\Phi}
&=\tau^{0}_{\Phi^{-1}}\A_{0}^{*}\A_{0}\tau^{0}_{\Phi},
&
\A_{1,\Phi}^{*}\A_{1,\Phi}
&=\tau^{1}_{\Phi^{-1}}\A_{1}^{*}\A_{1}\tau^{1}_{\Phi},\\
\A_{0,\Phi}\A_{0,\Phi}^{*}
&=\tau^{1}_{\Phi^{-1}}\A_{0}\A_{0}^{*}\tau^{1}_{\Phi},
&
\A_{1,\Phi}\A_{1,\Phi}^{*}
&=\tau^{1}_{\Phi^{-1}}\A_{1}\A_{1}^{*}\tau^{1}_{\Phi}.
\end{align*}
Moreover, $\rho\A_{0,\Phi}\A_{0,\Phi}^{*}+\A_{1,\Phi}^{*}\A_{1,\Phi}$
and $\rho\A_{0}\A_{0}^{*}+\A_{1}^{*}\A_{1}$
are unitarily equivalent, i.e.,
$$\rho\A_{0,\Phi}\A_{0,\Phi}^{*}+\A_{1,\Phi}^{*}\A_{1,\Phi}
=\tau^{1}_{\Phi^{-1}}(\rho\A_{0}\A_{0}^{*}+\A_{1}^{*}\A_{1})\tau^{1}_{\Phi}.$$
\end{theo}

\begin{proof}
Apply Lemma \ref{lem:unitequi1}.
\end{proof}

\begin{cor}
\label{cor:unitequi2}
The positive parts of the spectra of
$\A_{\ell,\Phi}^{*}\A_{\ell,\Phi}$, $\A_{\ell,\Phi}\A_{\ell,\Phi}^{*}$,
and $\A_{\ell}^{*}\A_{\ell}$, $\A_{\ell}\A_{\ell}^{*}$ coincide.
More precisely, 
\begin{align*}
\sigma(\A_{\ell,\Phi}^{*}\A_{\ell,\Phi})\setminus\{0\}
=\sigma(\A_{\ell,\Phi}\A_{\ell,\Phi}^{*})\setminus\{0\}
=\sigma(\A_{\ell}\A_{\ell}^{*})\setminus\{0\}
=\sigma(\A_{\ell}^{*}\A_{\ell})\setminus\{0\}.
\end{align*}
Moreover, the positive parts of the spectra of
$\rho\A_{0,\Phi}\A_{0,\Phi}^{*}+\A_{1,\Phi}^{*}\A_{1,\Phi}$,
and $\rho\A_{0}\A_{0}^{*}+\A_{1}^{*}\A_{1}$ coincide, i.e.,
\begin{align*}
\sigma(\rho\A_{0,\Phi}\A_{0,\Phi}^{*}+\A_{1,\Phi}^{*}\A_{1,\Phi})\setminus\{0\}
=\sigma(\rho\A_{0}\A_{0}^{*}+\A_{1}^{*}\A_{1})\setminus\{0\}.
\end{align*}
\end{cor}

\begin{proof}
Recall \eqref{eq:saops} and apply Theorem \ref{theo:unitequi2}.
\end{proof}

\begin{rem}[eigenvectors]
\label{rem:unitequi2}
It holds:
\begin{itemize}
\item
$u$ is an eigenvector of $\A_{0,\Phi}^{*}\A_{0,\Phi}$, if and only if
$\tau^{0}_{\Phi}u$ is an eigenvector of $\A_{0}^{*}\A_{0}$.
\item
$E$ is an eigenvector of $\A_{1,\Phi}^{*}\A_{1,\Phi}$ and $\A_{0,\Phi}\A_{0,\Phi}^{*}$, respectively, if and only if
$\tau^{1}_{\Phi}E$ is an eigenvector of $\A_{1}^{*}\A_{1}$ and $\A_{0}\A_{0}^{*}$, respectively.
\item
$H$ is an eigenvector of $\A_{1,\Phi}\A_{1,\Phi}^{*}$, if and only if
$\tau^{1}_{\Phi}H$ is an eigenvector of $\A_{1}\A_{1}^{*}$.
\end{itemize}
\end{rem}

Note that by Lemma \ref{lem:unitequi1} and Remark \ref{rem:unitequi1} we have, e.g.,
for $F,G\in\L^{2}_{\eps}(\omp)$
\begin{align*}
\scp{\tau^{1}_{\Phi}F}{\tau^{1}_{\Phi}G}_{\L^{2}_{\eps_{\Phi}}(\om)}
&=\scp{F}{\tau^{1}_{\Phi^{-1}}\tau^{1}_{\Phi}G}_{\L^{2}_{\eps}(\omp)}
=\scp{F}{G}_{\L^{2}_{\eps}(\omp)},
\intertext{and for $F,G\in D(\A_{1,\Phi})$}
\scp{\A_{1}\tau^{1}_{\Phi}F}{\A_{1}\tau^{1}_{\Phi}G}_{\L^{2}_{\mu_{\Phi}}(\om)}
&=\scp{\tau^{1}_{\Phi}\A_{1,\Phi}F}{\tau^{1}_{\Phi}\A_{1,\Phi}G}_{\L^{2}_{\mu_{\Phi}}(\om)}
=\scp{\A_{1,\Phi}F}{\A_{1,\Phi}G}_{\L^{2}_{\mu}(\omp)},\\
\scp{\tau^{1}_{\Phi}F}{\tau^{1}_{\Phi}G}_{D(\A_{1})}
&=\scp{F}{G}_{D(\A_{1,\Phi})}.
\end{align*}
For a definition of the inner products $\scp{\,\cdot\,}{\,\cdot\,}_{D(\A_{1})}$
and $\scp{\,\cdot\,}{\,\cdot\,}_{D(\A_{1,\Phi})}$ see the next remark.

Hence we get:

\begin{rem}[isometries and orthonormal bases]
\label{rem:isoonbas}
The transformations $\tau^{q}_{\Phi}$ are isometries.
In particular, othonormal bases are mapped to othonormal bases.
More precisely:
\begin{itemize}
\item
$\tau^{0}_{\Phi}:\L^{2}_{\nu}(\omp)\to\L^{2}_{\nu_{\Phi}}(\om)$
and $\tau^{0}_{\Phi}:D(\A_{0,\Phi})\to D(\A_{0})$,
the latter 
$$D(\A_{0,\Phi})=\H^{1}_{\gatp}(\omp),\qquad
D(\A_{0})=\H^{1}_{\gat}(\om)$$
equipped with the inner products
\begin{align*}
\scp{\,\cdot\,}{\,\cdot\,}_{D(\A_{0,\Phi})}
&=\scp{\,\cdot\,}{\,\cdot\,}_{\L^{2}_{\nu}(\omp)}
+\scp{\A_{0,\Phi}\,\cdot\,}{\A_{0,\Phi}\,\cdot\,}_{\L^{2}_{\eps}(\omp)}
=\scp{\,\cdot\,}{\,\cdot\,}_{\L^{2}_{\nu}(\omp)}
+\scp{\na\,\cdot\,}{\na\,\cdot\,}_{\L^{2}_{\eps}(\omp)},\\
\scp{\,\cdot\,}{\,\cdot\,}_{D(\A_{0})}
&=\scp{\,\cdot\,}{\,\cdot\,}_{\L^{2}_{\nu_{\Phi}}(\om)}
+\scp{\A_{0}\,\cdot\,}{\A_{0}\,\cdot\,}_{\L^{2}_{\eps_{\Phi}}(\om)}
=\scp{\,\cdot\,}{\,\cdot\,}_{\L^{2}_{\nu_{\Phi}}(\om)}
+\scp{\na\,\cdot\,}{\na\,\cdot\,}_{\L^{2}_{\eps_{\Phi}}(\om)},
\end{align*}
are isometries. Hence $\tau^{0}_{\Phi}$ 
maps a $\L^{2}_{\nu}(\omp)$-orthonormal basis
or a $D(\A_{0,\Phi})$-orthonormal basis $\{u_{m}\}$ 
to the $\L^{2}_{\nu_{\Phi}}(\om)$-orthonormal basis 
or the $D(\A_{0})$-orthonormal basis $\{\tau^{0}_{\Phi}u_{m}\}$,
respectively, and vice versa. 
\item 
$\tau^{1}_{\Phi}:\L^{2}_{\eps}(\omp)\to\L^{2}_{\eps_{\Phi}}(\om)$,
$\tau^{1}_{\Phi}:\L^{2}_{\mu}(\omp)\to\L^{2}_{\mu_{\Phi}}(\om)$
and 
$$\tau^{1}_{\Phi}:D(\A_{1,\Phi})\to D(\A_{1}),\quad
\tau^{1}_{\Phi}:D(\A_{1,\Phi}^{*})\to D(\A_{1}^{*}),\quad
\tau^{1}_{\Phi}:D(\A_{0,\Phi}^{*})\to D(\A_{0,\Phi}^{*}),$$
the latter 
\begin{align*}
D(\A_{1,\Phi})&=\R_{\gatp}(\omp),
&
D(\A_{1})&=\R_{\gat}(\om),\\
D(\A_{1,\Phi}^{*})&=\R_{\ganp}(\omp),
&
D(\A_{1}^{*})&=\R_{\gan}(\om),\\
D(\A_{0,\Phi}^{*})&=\eps^{-1}\D_{\ganp}(\omp)
,
&
D(\A_{0}^{*})&=\eps_{\Phi}^{-1}\D_{\gan}(\om)
\end{align*}
equipped with the inner products
\begin{align*}
\scp{\,\cdot\,}{\,\cdot\,}_{D(\A_{1,\Phi})}
&=\scp{\,\cdot\,}{\,\cdot\,}_{\L^{2}_{\eps}(\omp)}
+\scp{\A_{1,\Phi}\,\cdot\,}{\A_{1,\Phi}\,\cdot\,}_{\L^{2}_{\mu}(\omp)}
=\scp{\,\cdot\,}{\,\cdot\,}_{\L^{2}_{\eps}(\omp)}
+\scp{\rot\,\cdot\,}{\rot\,\cdot\,}_{\L^{2}_{\mu^{-1}}(\omp)},\\
\scp{\,\cdot\,}{\,\cdot\,}_{D(\A_{1})}
&=\scp{\,\cdot\,}{\,\cdot\,}_{\L^{2}_{\eps_{\Phi}}(\om)}
+\scp{\A_{1}\,\cdot\,}{\A_{1}\,\cdot\,}_{\L^{2}_{\mu_{\Phi}}(\om)}
=\scp{\,\cdot\,}{\,\cdot\,}_{\L^{2}_{\eps_{\Phi}}(\om)}
+\scp{\rot\,\cdot\,}{\rot\,\cdot\,}_{\L^{2}_{\mu_{\Phi}^{-1}}(\om)},\\
\scp{\,\cdot\,}{\,\cdot\,}_{D(\A_{1,\Phi}^{*})}
&=\scp{\,\cdot\,}{\,\cdot\,}_{\L^{2}_{\mu}(\omp)}
+\scp{\A_{1,\Phi}^{*}\,\cdot\,}{\A_{1,\Phi}^{*}\,\cdot\,}_{\L^{2}_{\eps}(\omp)}
=\scp{\,\cdot\,}{\,\cdot\,}_{\L^{2}_{\mu}(\omp)}
+\scp{\rot\,\cdot\,}{\rot\,\cdot\,}_{\L^{2}_{\eps^{-1}}(\omp)},\\
\scp{\,\cdot\,}{\,\cdot\,}_{D(\A_{1}^{*})}
&=\scp{\,\cdot\,}{\,\cdot\,}_{\L^{2}_{\mu_{\Phi}}(\om)}
+\scp{\A_{1}^{*}\,\cdot\,}{\A_{1}^{*}\,\cdot\,}_{\L^{2}_{\eps_{\Phi}}(\om)}
=\scp{\,\cdot\,}{\,\cdot\,}_{\L^{2}_{\mu_{\Phi}}(\om)}
+\scp{\rot\,\cdot\,}{\rot\,\cdot\,}_{\L^{2}_{\eps_{\Phi}^{-1}}(\om)},\\
\scp{\,\cdot\,}{\,\cdot\,}_{D(\A_{0,\Phi}^{*})}
&=\scp{\,\cdot\,}{\,\cdot\,}_{\L^{2}_{\eps}(\omp)}
+\scp{\A_{0,\Phi}^{*}\,\cdot\,}{\A_{0,\Phi}^{*}\,\cdot\,}_{\L^{2}_{\nu}(\omp)}
=\scp{\,\cdot\,}{\,\cdot\,}_{\L^{2}_{\eps}(\omp)}
+\scp{\div\eps\,\cdot\,}{\div\eps\,\cdot\,}_{\L^{2}_{\nu^{-1}}(\omp)},\\
\scp{\,\cdot\,}{\,\cdot\,}_{D(\A_{0}^{*})}
&=\scp{\,\cdot\,}{\,\cdot\,}_{\L^{2}_{\eps_{\Phi}}(\om)}
+\scp{\A_{0}^{*}\,\cdot\,}{\A_{0}^{*}\,\cdot\,}_{\L^{2}_{\nu_{\Phi}}(\om)}
=\scp{\,\cdot\,}{\,\cdot\,}_{\L^{2}_{\eps_{\Phi}}(\om)}
+\scp{\div\eps_{\Phi}\,\cdot\,}{\div\eps_{\Phi}\,\cdot\,}_{\L^{2}_{\nu_{\Phi}^{-1}}(\om)},
\end{align*}
are isometries. Hence $\tau^{1}_{\Phi}$ 
maps a $\L^{2}_{\eps}(\omp)$-orthonormal basis
or a $D(\A_{1,\Phi})$-orthonormal basis
or a $D(\A_{0,\Phi}^{*})$-orthonormal basis $\{E_{m}\}$
to the $\L^{2}_{\eps_{\Phi}}(\om)$-orthonormal basis 
or the $D(\A_{1})$-orthonormal basis or the $D(\A_{0}^{*})$-orthonormal basis 
$\{\tau^{1}_{\Phi}E_{m}\}$, respectively, and vice versa. 
Analogously, $\tau^{1}_{\Phi}$ 
maps a $\L^{2}_{\mu}(\omp)$-orthonormal basis
or a $D(\A_{1,\Phi}^{*})$-orthonormal basis $\{H_{m}\}$
to the $\L^{2}_{\mu_{\Phi}}(\om)$-orthonormal basis 
or the $D(\A_{1}^{*})$-orthonormal basis 
$\{\tau^{1}_{\Phi}H_{m}\}$, respectively, and vice versa. 
\end{itemize}
\end{rem}

\subsection{Point Spectrum}
\label{sec:pointspectrum}

Now and more precisely \eqref{eq:maxev3}-\eqref{eq:lapev2} read for the domain $\omp$
\begin{align}
\begin{aligned}
\label{eq:maxev4}
\A_{0,\Phi}^{*}\A_{0,\Phi}u=-\nu^{-1}\div_{\ganp}\eps\na_{\gatp}u
&=\lambda_{0}u
&
\text{in }&\L^{2}_{\nu}(\omp),\\
\A_{1,\Phi}^{*}\A_{1,\Phi}E=\eps^{-1}\rot_{\ganp}\mu^{-1}\rot_{\gatp}E
&=\lambda_{1}E
&
\text{in }&\L^{2}_{\eps}(\omp),\\
\A_{0,\Phi}\A_{0,\Phi}^{*}H=-\na_{\gatp}\nu^{-1}\div_{\ganp}\eps H
&=\lambda_{0}H
&
\quad\text{ in }&\L^{2}_{\eps}(\omp)
\end{aligned}
\end{align}
with some eigenvectors
\begin{align*}
u&\in D(\A_{0,\Phi}^{*}\A_{0,\Phi})
=D(\nu^{-1}\div_{\ganp}\eps\na_{\gatp})
=\big\{\psi\in\H^{1}_{\gatp}(\omp)\,:\,\eps\na\psi\in\D_{\ganp}(\omp)\big\},\\
E&\in D(\A_{1,\Phi}^{*}\A_{1,\Phi})
=D(\eps^{-1}\rot_{\ganp}\mu^{-1}\rot_{\gatp})
=\big\{\Psi\in\R_{\gatp}(\omp)\,:\,\mu^{-1}\rot\Psi\in\R_{\ganp}(\omp)\big\},\\
H&\in D(\A_{0,\Phi}\A_{0,\Phi}^{*})
=D(\na_{\gatp}\nu^{-1}\div_{\ganp}\eps)
=\big\{\Psi\in\eps^{-1}\D_{\ganp}(\omp)\,:\,\nu^{-1}\div\eps\Psi\in\H^{1}_{\gatp}(\omp)\big\}.
\end{align*}

\begin{rem}
Note that for the eigenfields $E$ and $H$
a normal and tangential boundary condition is induced by the complex property since
\begin{align*}
E&\in R(\A_{1,\Phi}^{*})\subset N(\A_{0,\Phi}^{*})=N(\nu^{-1}\div_{\ganp}\eps)
=\big\{\Psi\in\L^{2}(\omp)\,:\,\div\eps\Psi=0,\,\n\cdot\eps\Psi|_{\ganp}=0\big\},\\
H&\in R(\A_{0,\Phi})\subset N(\A_{1,\Phi})=N(\mu^{-1}\rot_{\gatp})
=\big\{\Psi\in\L^{2}(\omp)\,:\,\rot\Psi=0,\,\n\times\Psi|_{\gatp}=0\big\},
\end{align*}
respectively.
\end{rem}

We want to discuss \eqref{eq:maxev4} equivalently in $\om$
utilising the pullbacks $\tau^{q}_{\Phi}$ and
Theorem \ref{theo:evderham}, Theorem \ref{theo:unitequi2}, and Corollary \ref{cor:unitequi2}.

\begin{theo}[eigenvalues of the de Rham complex]
\label{theo:evderhamphi}
$\A_{\ell,\Phi}^{*}\A_{\ell,\Phi}$ and $\A_{\ell}^{*}\A_{\ell}$
are unitarily equivalent. The same holds for 
$\A_{\ell,\Phi}\A_{\ell,\Phi}^{*}$ and $\A_{\ell}\A_{\ell}^{*}$.
All these operators are self-adjoint and non-negative and have pure and discrete point spectrum with no accumulation point.
Moreover, the positive parts of the spectra coincide, i.e.,
\begin{align*}
&\qquad\sigma(\mu^{-1}\rot_{\gatp}\eps^{-1}\rot_{\ganp})\setminus\{0\}
=\sigma(\eps^{-1}\rot_{\ganp}\mu^{-1}\rot_{\gatp})\setminus\{0\}\\
&=\sigma(\mu_{\Phi}^{-1}\rot_{\gat}\eps_{\Phi}^{-1}\rot_{\gan})\setminus\{0\}
=\sigma(\eps_{\Phi}^{-1}\rot_{\gan}\mu_{\Phi}^{-1}\rot_{\gat})\setminus\{0\}
=\{\lambda_{1,\Phi,k}\}_{k\in\nat}
\subset(0,\infty)
\intertext{and}
&\qquad\sigma(-\na_{\gatp}\nu^{-1}\div_{\ganp}\eps)\setminus\{0\}
=\sigma(-\nu^{-1}\div_{\ganp}\eps\na_{\gatp})\setminus\{0\}\\
&=\sigma(-\na_{\gat}\nu_{\Phi}^{-1}\div_{\gan}\eps_{\Phi})\setminus\{0\}
=\sigma(-\nu_{\Phi}^{-1}\div_{\gan}\eps_{\Phi}\na_{\gat})\setminus\{0\}
=\{\lambda_{0,\Phi,k}\}_{k\in\nat}
\subset(0,\infty)
\end{align*}
with eigenvalues 
$0<\lambda_{\ell,\Phi,1}\leq\lambda_{\ell,\Phi,2}\leq\dots\leq\lambda_{\ell,\Phi,k-1}\leq\lambda_{\ell,\Phi,k}\leq\dots\to\infty$.
Only finitely many eigenvalues coincide  and they are repeated according to their multiplicity.
\end{theo}

\begin{proof}
Theorem \ref{theo:evderham}, Theorem \ref{theo:unitequi2}, and Corollary \ref{cor:unitequi2} yield
\begin{align*}
&\qquad\sigma(\A_{\ell,\Phi}^{*}\A_{\ell,\Phi})\setminus\{0\}
=\sigma(\A_{\ell,\Phi}\A_{\ell,\Phi}^{*})\setminus\{0\}
=\sigma(\cA_{\ell,\Phi}\cA_{\ell,\Phi}^{*})
=\sigma(\cA_{\ell,\Phi}^{*}\cA_{\ell,\Phi})\\
&=\sigma(\A_{\ell}^{*}\A_{\ell})\setminus\{0\}
=\sigma(\A_{\ell}\A_{\ell}^{*})\setminus\{0\}
=\sigma(\cA_{\ell}\cA_{\ell}^{*})
=\sigma(\cA_{\ell}^{*}\cA_{\ell})
=:\{\lambda_{\ell,\Phi,k}\}_{k\in\nat}
\end{align*}
for $\ell\in\{0,1\}$.
\end{proof}

\begin{rem}
\label{rem:lippairmap}
Note that by the definition of weak Lipschitz pair  it follows directly that $\Phi$ maps a weak Lipschitz pair $(\om,\gat)$
to a weak Lipschitz pair $(\omp,\gatp)$.
\end{rem}

Theorem \ref{theo:unitequi2}, Remark \ref{rem:unitequi2}, and Theorem \ref{theo:evderhamphi} show
that eigenvectors $u$, $E$, and $H$ in \eqref{eq:maxev4} 
for the domain $\omp$ and for the eigenvalues $\lambda_{0}$ and $\lambda_{1}$ 
are mapped to eigenvectors 
$\tau^{0}_{\Phi}u$, $\tau^{1}_{\Phi}E$, and $\tau^{1}_{\Phi}H$ for the domain $\om$
for the same eigenvalues and vice versa. More precisely,
$u$, $E$, and $H$ are eigenvectors in \eqref{eq:maxev4}, if and only if
\begin{align}
\begin{aligned}
\label{eq:maxev5}
\A_{0}^{*}\A_{0}\tau^{0}_{\Phi}u
=-\nu_{\Phi}^{-1}\div_{\gan}\eps_{\Phi}\na_{\gat}\tau^{0}_{\Phi}u
&=\lambda_{0}\tau^{0}_{\Phi}u
&
\text{in }&\L^{2}_{\nu_{\Phi}}(\om),\\
\A_{1}^{*}\A_{1}\tau^{1}_{\Phi}E
=\eps_{\Phi}^{-1}\rot_{\gan}\mu_{\Phi}^{-1}\rot_{\gat}\tau^{1}_{\Phi}E
&=\lambda_{1}\tau^{1}_{\Phi}E
&
\text{in }&\L^{2}_{\eps_{\Phi}}(\om),\\
\A_{0}\A_{0}^{*}\tau^{1}_{\Phi}H
=-\na_{\gat}\nu_{\Phi}^{-1}\div_{\gan}\eps_{\Phi}\tau^{1}_{\Phi}H
&=\lambda_{0}\tau^{1}_{\Phi}H
&
\quad\text{ in }&\L^{2}_{\eps_{\Phi}}(\om)
\end{aligned}
\end{align}
with eigenvectors
\begin{align*}
\tau^{0}_{\Phi}u&\in D(\A_{0}^{*}\A_{0})
=D(\nu_{\Phi}^{-1}\div_{\gan}\eps_{\Phi}\na_{\gat})
=\big\{\psi\in\H^{1}_{\gat}(\om)\,:\,\eps_{\Phi}\na\psi\in\D_{\gan}(\om)\big\},\\
\tau^{1}_{\Phi}E&\in D(\A_{1}^{*}\A_{1})
=D(\eps_{\Phi}^{-1}\rot_{\gan}\mu_{\Phi}^{-1}\rot_{\gat})
=\big\{\Psi\in\R_{\gat}(\om)\,:\,\mu_{\Phi}^{-1}\rot\Psi\in\R_{\gan}(\om)\big\},\\
\tau^{1}_{\Phi}H&\in D(\A_{0}\A_{0}^{*})
=D(\na_{\gat}\nu_{\Phi}^{-1}\div_{\gan}\eps_{\Phi})
=\big\{\Psi\in\eps_{\Phi}^{-1}\D_{\gan}(\om)\,:\,\nu_{\Phi}^{-1}\div\eps_{\Phi}\Psi\in\H^{1}_{\gat}(\om)\big\}.
\end{align*}

\begin{rem}
We have
\begin{align*}
u&\in D(\A_{0,\Phi}^{*}\A_{0,\Phi})
\subset D(\A_{0,\Phi})
=D(\na_{\gatp})
=\H^{1}_{\gatp}(\omp),\\
E&\in D(\A_{1,\Phi}^{*}\A_{1,\Phi})\cap N(\A_{0,\Phi}^{*})
\subset D(\A_{1,\Phi})\cap N(\A_{0,\Phi}^{*})
=D(\rot_{\gatp})\cap N(\div_{\ganp}\eps)\\
&=\big\{\Psi\in\R_{\gatp}(\omp)\cap\eps^{-1}\D_{\ganp}(\omp)\,:\,\div\eps\Psi=0\big\}\\
&=\big\{\Psi\in\R(\omp)\cap\eps^{-1}\D(\omp)\,:\,\div\eps\Psi=0,\,
\n\times\Psi|_{\gatp}=0,\,\n\cdot\eps\Psi|_{\ganp}=0\big\},\\
H&\in D(\A_{0,\Phi}\A_{0,\Phi}^{*})\cap N(\A_{1,\Phi})
\subset D(\A_{0,\Phi}^{*})\cap N(\A_{1,\Phi})
=D(\div_{\ganp}\eps)\cap N(\rot_{\gatp})\\
&=\big\{\Psi\in\R_{\gatp}(\omp)\cap\eps^{-1}\D_{\ganp}(\omp)\,:\,\rot\Psi=0\big\}\\
&=\big\{\Psi\in\R(\omp)\cap\eps^{-1}\D(\omp)\,:\,\rot\Psi=0,\,
\n\times\Psi|_{\gatp}=0,\,\n\cdot\eps\Psi|_{\ganp}=0\big\},
\end{align*}
and for the transformed fields
\begin{align*}
\tau^{0}_{\Phi}u&\in D(\A_{0}^{*}\A_{0})
\subset D(\A_{0})
=D(\na_{\gat})
=\H^{1}_{\gat}(\om),\\
\tau^{1}_{\Phi}E&\in D(\A_{1}^{*}\A_{1})\cap N(\A_{0}^{*})
\subset D(\A_{1})\cap N(\A_{0}^{*})
=D(\rot_{\gat})\cap N(\div_{\gan}\eps_{\Phi})\\
&=\big\{\Psi\in\R_{\gat}(\om)\cap\eps_{\Phi}^{-1}\D_{\gan}(\om)\,:\,\div\eps_{\Phi}\Psi=0\big\}\\
&=\big\{\Psi\in\R(\om)\cap\eps_{\Phi}^{-1}\D(\om)\,:\,\div\eps_{\Phi}\Psi=0,\,
\n\times\Psi|_{\gat}=0,\,\n\cdot\eps_{\Phi}\Psi|_{\gan}=0\big\},\\
\tau^{1}_{\Phi}H&\in D(\A_{0}\A_{0}^{*})\cap N(\A_{1})
\subset D(\A_{0}^{*})\cap N(\A_{1})
=D(\div_{\gan}\eps_{\Phi})\cap N(\rot_{\gat})\\
&=\big\{\Psi\in\R_{\gat}(\om)\cap\eps_{\Phi}^{-1}\D_{\gan}(\om)\,:\,\rot\Psi=0\big\}\\
&=\big\{\Psi\in\R(\om)\cap\eps_{\Phi}^{-1}\D(\om)\,:\,\rot\Psi=0,\,
\n\times\Psi|_{\gat}=0,\,\n\cdot\eps_{\Phi}\Psi|_{\gan}=0\big\}.
\end{align*}
For corresponding variational formulations see the appendix.
\end{rem}

\begin{theo}[eigenvalues of the generalised Laplacian]
\label{theo:evlapphi}
The operators 
$\rho\A_{0,\Phi}\A_{0,\Phi}^{*}+\A_{1,\Phi}^{*}\A_{1,\Phi}$ and
$\rho\A_{0}\A_{0}^{*}+\A_{1}^{*}\A_{1}$ are unitarily equivalent.
Moreover, both are self-adjoint and non-negative and have pure and discrete point spectrum with no accumulation point.
Moreover, the positive parts of the spectra coincide, i.e.,
\begin{align*}
&\qquad\sigma(-\rho\na_{\gatp}\nu^{-1}\div_{\ganp}\eps
+\eps^{-1}\rot_{\ganp}\mu^{-1}\rot_{\gatp})\setminus\{0\}\\
&=\sigma(-\rho\na_{\gat}\nu_{\Phi}^{-1}\div_{\gan}\eps_{\Phi}
+\eps_{\Phi}^{-1}\rot_{\gan}\mu_{\Phi}^{-1}\rot_{\gat})\setminus\{0\}\\
&=\big(\rho\sigma(-\na_{\gat}\nu_{\Phi}^{-1}\div_{\gan}\eps_{\Phi})
\cup\sigma(\eps_{\Phi}^{-1}\rot_{\gan}\mu_{\Phi}^{-1}\rot_{\gat})\big)\setminus\{0\}
=\rho\{\lambda_{0,\Phi,k}\}_{k\in\nat}\cup\{\lambda_{1,\Phi,k}\}_{k\in\nat}.
\end{align*}
Only finitely many eigenvalues coincide  and they are repeated according to their multiplicity.
\end{theo}

\section{Conclusion and Outlook}
\label{sec:conclu}

\subsection{Eigenvalues and Rayleigh Quotients}
\label{sec:evandrayq}

We recall our results on the de Rham eigenvalues which are important 
for the study of their shape derivatives 
(variations of the domain and the boundary conditions
via the Lipschitz maps $\Phi:\om\to\Phi(\om)=\omp$)
in the second part of this paper.
So far we have shown that for bounded weak Lipschitz pairs $(\om,\gat)$
the de Rahm complex has countably many eigenvalues 
$$0<\lambda_{\ell,\Phi,1}\leq\lambda_{\ell,\Phi,2}\leq\dots
\leq\lambda_{\ell,\Phi,k-1}\leq\lambda_{\ell,\Phi,k}\leq\dots\to\infty,\qquad
\ell\in\{0,1\}.$$
For a fixed index $k$ we set
$$\lambda_{0,\Phi}:=\lambda_{0,\Phi,k},\qquad
\lambda_{1,\Phi}:=\lambda_{1,\Phi,k}.$$
Moreover, by Lemma \ref{lem:eigenval2} 
and \eqref{eq:maxev4}, \eqref{eq:maxev5}
the eigenvalues are given by the Rayleigh quotients
of the eigenfields, this is
\begin{align}
\begin{aligned}
\label{eq:eigenvaluesshapederivatives1}
\lambda_{0,\Phi}
&=\frac{\norm{\A_{0,\Phi}u}_{\L^{2}_{\eps}(\omp)}^{2}}
{\norm{u}_{\L^{2}_{\nu}(\omp)}^{2}}
=\frac{\scp{\eps\na_{\gat,\Phi}u}{\na_{\gat,\Phi}u}_{\L^{2}(\omp)}}
{\scp{\nu u}{u}_{\L^{2}(\omp)}}\\
&=\frac{\norm{\A_{0}\tau^{0}_{\Phi}u}_{\L^{2}_{\eps_{\Phi}}(\om)}^{2}}
{\norm{\tau^{0}_{\Phi}u}_{\L^{2}_{\nu_{\Phi}}(\om)}^{2}}
=\frac{\scp{\eps_{\Phi}\na_{\gat}\tau^{0}_{\Phi}u}{\na_{\gat}\tau^{0}_{\Phi}u}_{\L^{2}(\om)}}
{\scp{\nu_{\Phi}\tau^{0}_{\Phi}u}{\tau^{0}_{\Phi}u}_{\L^{2}(\om)}}\\
&=\frac{\norm{\A_{0,\Phi}^{*}H}_{\L^{2}_{\nu}(\omp)}^{2}}
{\norm{H}_{\L^{2}_{\eps}(\omp)}^{2}}
=\frac{\scp{\nu^{-1}\div_{\gan,\Phi}\eps H}{\div_{\gan,\Phi}\eps H}_{\L^{2}(\omp)}}
{\scp{\eps H}{H}_{\L^{2}(\omp)}}\\
&=\frac{\norm{\A_{0}^{*}\tau^{1}_{\Phi}H}_{\L^{2}_{\nu_{\Phi}}(\om)}^{2}}
{\norm{\tau^{1}_{\Phi}H}_{\L^{2}_{\eps_{\Phi}}(\om)}^{2}}
=\frac{\scp{\nu_{\Phi}^{-1}\div_{\gan}\eps_{\Phi}\tau^{1}_{\Phi}H}{\div_{\gan}\eps_{\Phi}\tau^{1}_{\Phi}H}_{\L^{2}(\om)}}
{\scp{\eps_{\Phi}\tau^{1}_{\Phi}H}{\tau^{1}_{\Phi}H}_{\L^{2}(\om)}},\\
\lambda_{1,\Phi}
&=\frac{\norm{\A_{1,\Phi}E}_{\L^{2}_{\mu}(\omp)}^{2}}
{\norm{E}_{\L^{2}_{\eps}(\omp)}^{2}}
=\frac{\scp{\mu^{-1}\rot_{\gat,\Phi}E}{\rot_{\gat,\Phi}E}_{\L^{2}(\omp)}}
{\scp{\eps E}{E}_{\L^{2}(\omp)}}\\
&=\frac{\norm{\A_{1}\tau^{1}_{\Phi}E}_{\L^{2}_{\mu_{\Phi}}(\om)}^{2}}
{\norm{\tau^{1}_{\Phi}E}_{\L^{2}_{\eps_{\Phi}}(\om)}^{2}}
=\frac{\scp{\mu_{\Phi}^{-1}\rot_{\gat}\tau^{1}_{\Phi}E}{\rot_{\gat}\tau^{1}_{\Phi}E}_{\L^{2}(\om)}}
{\scp{\eps_{\Phi}\tau^{1}_{\Phi}E}{\tau^{1}_{\Phi}E}_{\L^{2}(\om)}}
\end{aligned}
\end{align}
with eigenfields $u$, $E$, $H$, and 
$$\tau^{0}_{\Phi}u=\wt{u},\qquad
\tau^{1}_{\Phi}E=J_{\Phi}^{\top}\wt{E},\qquad
\tau^{2}_{\Phi}H=(\adj J_{\Phi})\wt{H},$$
respectively. 
Note that the eigenvalues $\lambda_{\ell,\Phi,k}$ are depending
not only on $\Phi$ (shape of the domain) but also on the mixed boundary conditions imposed on $\gat$ and $\gan$
and on the coefficients $\eps$, $\mu$, and $\nu$, which we do not indicate explicitly in our notations, i.e.,
$$\lambda_{0,\Phi,k}
=\lambda_{0,\Phi,k}(\om,\gat,\eps,\nu),\qquad
\lambda_{1,\Phi,k}
=\lambda_{1,\Phi,k}(\om,\gat,\eps,\mu).$$

\subsection{Heuristic Shape Derivatives of Eigenvalues}
\label{sec:heuev}

In this final subsection we want to conclude with formal computations
to derive shape derivatives of the eigenvalues
assuming\footnote{Note that this assumption is quite strong and, unless one restricts the analysis to suitable families of perturbations $\Phi$, it requires that the eigenvalue under consideration is simple. See Part II of this series of papers for more details concerning multiple eigenvalues.} that  the corresponding (pull-backs of the) eigenvectors are differentiable with respect to $\Phi$.
This means we investigate the behaviour of the eigenvalues under 
variations of the domain and the boundary conditions.
More precisely, we investigate the differentiable dependence of the eigenvalues 
of the de Rahm complex if the domain, i.e., the mapping $\Phi$, 
is changing in certain subsets of bi-Lipschitz transformations. Here the space $\C^{0,1}(\omb,\rt)$ of Lipschitz maps from $ \omb$ to $\rt$   
is endowed with its standard norm.

\begin{theo}
\label{theo:derevformal1}
Let $\eps$, $\mu$, and $\nu$  be of class $C^1$.
Let $u$, $E$, $H$ be normalised eigenfields such that  
$$\norm{\tau^{0}_{\Phi}u}_{\L^{2}_{\nu_{\Phi}}(\om)}
=\norm{u}_{\L^{2}_{\nu}(\omp)}
=\norm{\tau^{1}_{\Phi}E}_{\L^{2}_{\eps_{\Phi}}(\om)}
=\norm{E}_{\L^{2}_{\eps}(\omp)}
=\norm{\tau^{1}_{\Phi}H}_{\L^{2}_{\eps_{\Phi}}(\om)}
=\norm{H}_{\L^{2}_{\eps}(\omp)}=1,$$
and 
\begin{align*}
\lambda_{0,\Phi}
&=\norm{\A_{0}\tau^{0}_{\Phi}u}_{\L^{2}_{\eps_{\Phi}}(\om)}^{2}
=\norm{\na_{\gat}\tau^{0}_{\Phi}u}_{\L^{2}_{\eps_{\Phi}}(\om)}^{2}
=\scp{\eps_{\Phi}\na_{\gat}\tau^{0}_{\Phi}u}{\na_{\gat}\tau^{0}_{\Phi}u}_{\L^{2}(\om)},\\
\lambda_{1,\Phi}
&=\norm{\A_{1}\tau^{1}_{\Phi}E}_{\L^{2}_{\mu_{\Phi}}(\om)}^{2}
=\norm{\mu_{\Phi}^{-1}\rot_{\gat}\tau^{1}_{\Phi}E}_{\L^{2}_{\mu_{\Phi}}(\om)}^{2}
=\scp{\mu_{\Phi}^{-1}\rot_{\gat}\tau^{1}_{\Phi}E}{\rot_{\gat}\tau^{1}_{\Phi}E}_{\L^{2}(\om)},\\
\lambda_{0,\Phi}
&=\norm{\A_{0}^{*}\tau^{1}_{\Phi}H}_{\L^{2}_{\nu_{\Phi}}(\om)}^{2}
=\norm{\nu_{\Phi}^{-1}\div_{\gan}\eps_{\Phi}\tau^{1}_{\Phi}H}_{\L^{2}_{\nu_{\Phi}}(\om)}^{2}
=\scp{\nu_{\Phi}^{-1}\div_{\gan}\eps_{\Phi}\tau^{1}_{\Phi}H}{\div_{\gan}\eps_{\Phi}\tau^{1}_{\Phi}H}_{\L^{2}(\om)},
\end{align*}
cf.~Remark \ref{rem:isoonbas} and \eqref{eq:eigenvaluesshapederivatives1}.
Assume that $\tau^{0}_{\Phi}u$, $\tau^{1}_{\Phi}E$, $\tau^{1}_{\Phi}H$ are differentiable with respect to $\Phi \in \C^{0,1}(\omb,\rt )$.   
Then the directional derivatives of the eigenvalues with respect to a direction $\wt{\Psi}  \in \C^{0,1}(\omb,\rt)$ are given by 
\begin{align*}
\p_{\wt{\Psi}}\lambda_{0,\Phi}
&=\norm{\A_{0}\tau^{0}_{\Phi}u}_{\L^{2}_{(\p_{\wt{\Psi}}\eps_{\Phi})}(\om)}^{2}
-\lambda_{0,\Phi}\norm{\tau^{0}_{\Phi}u}_{\L^{2}_{(\p_{\wt{\Psi}}\nu_{\Phi})}(\om)}^{2}\\
&=\norm{\na_{\gat}\tau^{0}_{\Phi}u}_{\L^{2}_{(\p_{\wt{\Psi}}\eps_{\Phi})}(\om)}^{2}
-\lambda_{0,\Phi}\norm{\tau^{0}_{\Phi}u}_{\L^{2}_{(\p_{\wt{\Psi}}\nu_{\Phi})}(\om)}^{2}\\
&=\bscp{(\p_{\wt{\Psi}}\eps_{\Phi})\na_{\gat}\tau^{0}_{\Phi}u}{\na_{\gat}\tau^{0}_{\Phi}u}_{\L^{2}(\om)}
-\lambda_{0,\Phi}\bscp{(\p_{\wt{\Psi}}\nu_{\Phi})\tau^{0}_{\Phi}u}{\tau^{0}_{\Phi}u}_{\L^{2}(\om)},\\
\p_{\wt{\Psi}}\lambda_{1,\Phi}
&=\norm{\A_{1}\tau^{1}_{\Phi}E}_{\L^{2}_{-(\p_{\wt{\Psi}}\mu_{\Phi})}(\om)}^{2}
-\lambda_{1,\Phi}\norm{\tau^{1}_{\Phi}E}_{\L^{2}_{(\p_{\wt{\Psi}}\eps_{\Phi})}(\om)}^{2}\\
&=\norm{\rot_{\gat}\tau^{1}_{\Phi}E}_{\L^{2}_{(\p_{\wt{\Psi}}\mu_{\Phi}^{-1})}(\om)}^{2}
-\lambda_{1,\Phi}\norm{\tau^{1}_{\Phi}E}_{\L^{2}_{(\p_{\wt{\Psi}}\eps_{\Phi})}(\om)}^{2}\\
&=\bscp{(\p_{\wt{\Psi}}\mu_{\Phi}^{-1})\rot_{\gat}\tau^{1}_{\Phi}E}{\rot_{\gat}\tau^{1}_{\Phi}E}_{\L^{2}(\om)}
-\lambda_{1,\Phi}\bscp{(\p_{\wt{\Psi}}\eps_{\Phi})\tau^{1}_{\Phi}E}{\tau^{1}_{\Phi}E}_{\L^{2}(\om)},\\
\p_{\wt{\Psi}}\lambda_{0,\Phi}
&=\norm{\A_{0}^{*}\tau^{1}_{\Phi}H}_{\L^{2}_{-(\p_{\wt{\Psi}}\nu_{\Phi})}(\om)}^{2}
-\lambda_{0,\Phi}\norm{\tau^{1}_{\Phi}H}_{\L^{2}_{-(\p_{\wt{\Psi}}\eps_{\Phi})}(\om)}^{2}\\
&=\norm{\div_{\gan}\eps_{\Phi}\tau^{1}_{\Phi}H}_{\L^{2}_{(\p_{\wt{\Psi}}\nu_{\Phi}^{-1})}(\om)}^{2}
+\lambda_{0,\Phi}\norm{\tau^{1}_{\Phi}H}_{\L^{2}_{(\p_{\wt{\Psi}}\eps_{\Phi})}(\om)}^{2}\\
&=\bscp{(\p_{\wt{\Psi}}\nu_{\Phi}^{-1})\div_{\gan}\eps_{\Phi}\tau^{1}_{\Phi}H}{\div_{\gan}\eps_{\Phi}\tau^{1}_{\Phi}H}_{\L^{2}(\om)}
+\lambda_{0,\Phi}\bscp{(\p_{\wt{\Psi}}\eps_{\Phi})\tau^{1}_{\Phi}H}{\tau^{1}_{\Phi}H}_{\L^{2}(\om)}.
\end{align*}
Here, we have formally used the norm notation 
although the tensor fields $\pm\p_{\wt{\Psi}}(\cdots)_{\Phi}^{\pm}$ do not necessarily 
generate proper $\L^{2}(\om)$-inner products.
Note that 
$$\p_{\wt{\Psi}}\eps_{\Phi}=-\eps_{\Phi}(\p_{\wt{\Psi}}\eps_{\Phi}^{-1})\eps_{\Phi},\qquad
\p_{\wt{\Psi}}\mu_{\Phi}=-\mu_{\Phi}(\p_{\wt{\Psi}}\mu_{\Phi}^{-1})\mu_{\Phi},\qquad
\p_{\wt{\Psi}}\nu_{\Phi}=-\nu_{\Phi}^{2}\p_{\wt{\Psi}}\nu_{\Phi}^{-1},$$
cf.~Remark \ref{rem:derivatives2}.
Furthermore, we understand terms like $\p_{\wt\Psi}\eps$ in the sense of
$$\p_{\wt\Psi}\eps:=[\p_{\wt\Psi}\eps_{j,m}].$$
\end{theo}

\begin{proof}
We elaborate the computations only for $\lambda_{1,\Phi}$
and postpone the calculations of the remaining cases to Appendix \ref{app:heucomp}.
By \eqref{eq:eigenvaluesshapederivatives1} and the quotient rule we compute
\begin{align*}
(\p_{\wt{\Psi}}\lambda_{1,\Phi})
\scp{\eps_{\Phi}\tau^{1}_{\Phi}E}{\tau^{1}_{\Phi}E}_{\L^{2}(\om)}^{2}
&=\scp{\eps_{\Phi}\tau^{1}_{\Phi}E}{\tau^{1}_{\Phi}E}_{\L^{2}(\om)}
\p_{\wt{\Psi}}\scp{\mu_{\Phi}^{-1}\rot_{\gat}\tau^{1}_{\Phi}E}{\rot_{\gat}\tau^{1}_{\Phi}E}_{\L^{2}(\om)}\\
&\qquad
-\scp{\mu_{\Phi}^{-1}\rot_{\gat}\tau^{1}_{\Phi}E}{\rot_{\gat}\tau^{1}_{\Phi}E}_{\L^{2}(\om)}
\p_{\wt{\Psi}}\scp{\eps_{\Phi}\tau^{1}_{\Phi}E}{\tau^{1}_{\Phi}E}_{\L^{2}(\om)}\\
&=\scp{\eps_{\Phi}\tau^{1}_{\Phi}E}{\tau^{1}_{\Phi}E}_{\L^{2}(\om)}
\Big(\bscp{(\p_{\wt{\Psi}}\mu_{\Phi}^{-1})\rot_{\gat}\tau^{1}_{\Phi}E}{\rot_{\gat}\tau^{1}_{\Phi}E}_{\L^{2}(\om)}\\
&\qquad\qquad
+2\Re\scp{\mu_{\Phi}^{-1}\rot_{\gat}\tau^{1}_{\Phi}E}{\rot_{\gat}\p_{\wt{\Psi}}\tau^{1}_{\Phi}E}_{\L^{2}(\om)}\Big)\\
&\qquad
-\scp{\mu_{\Phi}^{-1}\rot_{\gat}\tau^{1}_{\Phi}E}{\rot_{\gat}\tau^{1}_{\Phi}E}_{\L^{2}(\om)}
\Big(\bscp{(\p_{\wt{\Psi}}\eps_{\Phi})\tau^{1}_{\Phi}E}{\tau^{1}_{\Phi}E}_{\L^{2}(\om)}\\
&\qquad\qquad
+2\Re\scp{\eps_{\Phi}\tau^{1}_{\Phi}E}{\p_{\wt{\Psi}}\tau^{1}_{\Phi}E}_{\L^{2}(\om)}\Big).
\end{align*}
Note that we assume that  $\tau^{1}_{\Phi}E$ is differentiable and hence 
$\p_{\wt{\Psi}}\tau^{1}_{\Phi}E$ exists.
Thus using
$$\scp{\mu_{\Phi}^{-1}\rot_{\gat}\tau^{1}_{\Phi}E}{\rot_{\gat}\tau^{1}_{\Phi}E}_{\L^{2}(\om)}
=\lambda_{1,\Phi}\scp{\eps_{\Phi}\tau^{1}_{\Phi}E}{\tau^{1}_{\Phi}E}_{\L^{2}(\om)}$$
we see
\begin{align*}
(\p_{\wt{\Psi}}\lambda_{1,\Phi})
\scp{\eps_{\Phi}\tau^{1}_{\Phi}E}{\tau^{1}_{\Phi}E}_{\L^{2}(\om)}
&=\bscp{(\p_{\wt{\Psi}}\mu_{\Phi}^{-1})\rot_{\gat}\tau^{1}_{\Phi}E}{\rot_{\gat}\tau^{1}_{\Phi}E}_{\L^{2}(\om)}\\
&\qquad+2\Re\scp{\mu_{\Phi}^{-1}\rot_{\gat}\tau^{1}_{\Phi}E}{\rot_{\gat}\p_{\wt{\Psi}}\tau^{1}_{\Phi}E}_{\L^{2}(\om)}\\
&\qquad
-\lambda_{1,\Phi}
\Big(\bscp{(\p_{\wt{\Psi}}\eps_{\Phi})\tau^{1}_{\Phi}E}{\tau^{1}_{\Phi}E}_{\L^{2}(\om)}
+2\Re\scp{\eps_{\Phi}\tau^{1}_{\Phi}E}{\p_{\wt{\Psi}}\tau^{1}_{\Phi}E}_{\L^{2}(\om)}\Big)\\
&=\bscp{(\p_{\wt{\Psi}}\mu_{\Phi}^{-1})\rot_{\gat}\tau^{1}_{\Phi}E}{\rot_{\gat}\tau^{1}_{\Phi}E}_{\L^{2}(\om)}\\
&\qquad+2\Re\scp{\A_{1}\tau^{1}_{\Phi}E}{\A_{1}\p_{\wt{\Psi}}\tau^{1}_{\Phi}E}_{\L^{2}_{\mu_{\Phi}}(\om)}\\
&\qquad
-\lambda_{1,\Phi}
\Big(\bscp{(\p_{\wt{\Psi}}\eps_{\Phi})\tau^{1}_{\Phi}E}{\tau^{1}_{\Phi}E}_{\L^{2}(\om)}
+2\Re\scp{\tau^{1}_{\Phi}E}{\p_{\wt{\Psi}}\tau^{1}_{\Phi}E}_{\L^{2}_{\eps_{\Phi}}(\om)}\Big)\\
&=\bscp{(\p_{\wt{\Psi}}\mu_{\Phi}^{-1})\rot_{\gat}\tau^{1}_{\Phi}E}{\rot_{\gat}\tau^{1}_{\Phi}E}_{\L^{2}(\om)}\\
&\qquad-\lambda_{1,\Phi}\bscp{(\p_{\wt{\Psi}}\eps_{\Phi})\tau^{1}_{\Phi}E}{\tau^{1}_{\Phi}E}_{\L^{2}(\om)}\\
&\qquad
+2\Re\bscp{\underbrace{(\A_{1}^{*}\A_{1}-\lambda_{1,\Phi})\tau^{1}_{\Phi}E}_{=0}}{\p_{\wt{\Psi}}\tau^{1}_{\Phi}E}_{\L^{2}_{\eps_{\Phi}}(\om)}\\
&=\scp{\rot_{\gat}\tau^{1}_{\Phi}E}{\rot_{\gat}\tau^{1}_{\Phi}E}_{\L^{2}_{(\p_{\wt{\Psi}}\mu_{\Phi}^{-1})}(\om)}
-\lambda_{1,\Phi}\scp{\tau^{1}_{\Phi}E}{\tau^{1}_{\Phi}E}_{\L^{2}_{(\p_{\wt{\Psi}}\eps_{\Phi})}(\om)}.
\end{align*}
Note that $\p_{\wt{\Psi}}\eps_{\Phi}=-\eps_{\Phi}(\p_{\wt{\Psi}}\eps_{\Phi}^{-1})\eps_{\Phi}$.
For a normalised eigenfield $E$ we obtain the assertions.
\end{proof}

By Lemma \ref{lem:derivatives} we have
\begin{align*}
\p_{\wt{\Psi}}\eps_{\Phi}
&=(\det J_{\Phi})J_{\Phi}^{-1}\Big(
\p_{\wt{\Psi}}\wt{\eps}+(\wt{\div\Psi})\wt{\eps}-2\sym(\wt{J_{\Psi}}\wt{\eps})
\Big)J_{\Phi}^{-\top},\\
\p_{\wt{\Psi}}\eps_{\Phi}^{-1}
&=(\det J_{\Phi})^{-1}J_{\Phi}^{\top}\Big(
\p_{\wt{\Psi}}\wt{\eps^{-1}}-(\wt{\div\Psi})\wt{\eps^{-1}}+2\sym(\wt{\eps^{-1}}\wt{J_{\Psi}})
\Big)J_{\Phi},\\
\p_{\wt{\Psi}}\nu_{\Phi}
&=(\det J_{\Phi})\Big(\p_{\wt{\Psi}}\wt{\nu}+(\wt{\div\Psi})\wt{\nu}\Big),\\
\p_{\wt{\Psi}}\nu_{\Phi}^{-1}
&=-(\det J_{\Phi})^{-1}\wt{\nu}^{-2}\Big(\p_{\wt{\Psi}}\wt{\nu}+(\wt{\div\Psi})\wt{\nu}\Big)
=(\det J_{\Phi})^{-1}\Big(\p_{\wt{\Psi}}\wt{\nu^{-1}}-(\wt{\div\Psi})\wt{\nu^{-1}}\Big),
\end{align*}
where the same formulas hold for $\eps$ replaced by $\mu$. Recall that by definition   $2\sym M=M+M^{\top}$ for any square matrix $M$.
Note that, e.g., for $\eps$, we have
$$\p_{\wt{\Psi}}\wt{\eps}_{j,m}=\wt{J_{\eps_{j,m}}}\wt{\Psi} $$
and hence 
$$(\p_{\wt{\Psi}}\wt{\eps}_{j,m})\circ\Phi^{-1}=J_{\eps_{j,m}}\Psi :=\p_{\Psi}\eps_{j,m}.$$

\begin{theo}
\label{theo:derevformal2}
Let the assumptions of Theorem \ref{theo:derevformal1} be satisfied. Then
\begin{align*}
\p_{\wt{\Psi}}\lambda_{0,\Phi}
&=\Bscp{\big(
\p_{\Psi}\eps+(\div\Psi)\eps-2\sym(J_{\Psi}\eps)
\big)\na_{\gatp}u}{\na_{\gatp}u}_{\L^{2}(\omp)}\\
&\qquad-\lambda_{0,\Phi}
\Bscp{\big(\p_{\Psi}\nu+(\div\Psi)\nu\big)u}{u}_{\L^{2}(\omp)},\\
\p_{\wt{\Psi}}\lambda_{1,\Phi}
&=\Bscp{\big(
\p_{\Psi}\mu^{-1}-(\div\Psi)\mu^{-1}+2\sym(\mu^{-1}J_{\Psi})
\big)\rot_{\gatp}E}{\rot_{\gatp}E}_{\L^{2}(\omp)}\\
&\qquad-\lambda_{1,\Phi}\Bscp{\big(
\p_{\Psi}\eps+(\div\Psi)\eps-2\sym(J_{\Psi}\eps)
\big)E}{E}_{\L^{2}(\omp)},\\
\p_{\wt{\Psi}}\lambda_{0,\Phi}
&=\Bscp{\big(\p_{\Psi}\nu^{-1}-(\div\Psi)\nu^{-1}\big)
\div_{\ganp}\eps H}{\div_{\ganp}\eps H}_{\L^{2}(\omp)}\\
&\qquad+\lambda_{0,\Phi}\Bscp{\big(
\p_{\Psi}\eps+(\div\Psi)\eps-2\sym(J_{\Psi}\eps)
\big)H}{H}_{\L^{2}(\omp)}.
\end{align*}
Recalling Lemma \ref{lem:eigenvec1} and with the dual eigenvectors 
\begin{align*}
H^{*}&:=\lambda_{0,\Phi}^{-1/2}\A_{0,\Phi}u=\lambda_{0,\Phi}^{-1/2}\na_{\gatp}u,
&
u&=\lambda_{0,\Phi}^{-1/2}\A_{0,\Phi}^{*}H^{*}=-\lambda_{0,\Phi}^{-1/2}\nu^{-1}\div_{\ganp}\eps H^{*},\\
u^{*}&:=\lambda_{0,\Phi}^{-1/2}\A_{0,\Phi}^{*}H=-\lambda_{0,\Phi}^{-1/2}\nu^{-1}\div_{\ganp}\eps H,
&
H&=\lambda_{0,\Phi}^{-1/2}\A_{0,\Phi}u^{*}=\lambda_{0,\Phi}^{-1/2}\na_{\gatp}u^{*},\\
E^{*}&:=\lambda_{1,\Phi}^{-1/2}\A_{1,\Phi}E=\lambda_{1,\Phi}^{-1/2}\mu^{-1}\rot_{\gatp}E
&
E&:=\lambda_{1,\Phi}^{-1/2}\A_{1,\Phi}^{*}E^{*}=\lambda_{1,\Phi}^{-1/2}\eps^{-1}\rot_{\ganp}E^{*}
\end{align*}
we get the formulas
\begin{align*}
\frac{\p_{\wt{\Psi}}\lambda_{0,\Phi}}{\lambda_{0,\Phi}}
&=\Bscp{\big(
\p_{\Psi}\eps+(\div\Psi)\eps-2\sym(J_{\Psi}\eps)
\big)H^{*}}{H^{*}}_{\L^{2}(\omp)}\\
&\qquad-
\Bscp{\big(\p_{\Psi}\nu+(\div\Psi)\nu\big)u}{u}_{\L^{2}(\omp)},\\
\frac{\p_{\wt{\Psi}}\lambda_{1,\Phi}}{\lambda_{1,\Phi}}
&=-\Bscp{\big(
\p_{\Psi}\mu+(\div\Psi)\mu-2\sym(J_{\Psi}\mu)
\big)E^{*}}{E^{*}}_{\L^{2}(\omp)}\\
&\qquad-\Bscp{\big(
\p_{\Psi}\eps+(\div\Psi)\eps-2\sym(J_{\Psi}\eps)
\big)E}{E}_{\L^{2}(\omp)},\\
\frac{\p_{\wt{\Psi}}\lambda_{0,\Phi}}{\lambda_{0,\Phi}}
&=-\Bscp{\big(\p_{\Psi}\nu+(\div\Psi)\nu\big)
\nu u^{*}}{u^{*}}_{\L^{2}(\omp)}\\
&\qquad+\Bscp{\big(
\p_{\Psi}\eps+(\div\Psi)\eps-2\sym(J_{\Psi}\eps)
\big)H}{H}_{\L^{2}(\omp)}.
\end{align*}
\end{theo}

\begin{proof}
Again we focus on $\lambda_{1,\Phi}$ and refer for $\lambda_{0,\Phi}$ to Appendix \ref{app:heucomp}.
By Theorem \ref{theo:transtheo}, Corollary \ref{cor:transtheo}, and Remark \ref{rem:transtheo1} we see
\begin{align*}
\p_{\wt{\Psi}}\lambda_{1,\Phi}
&=\bscp{(\p_{\wt{\Psi}}\mu_{\Phi}^{-1})\rot_{\gat}\tau^{1}_{\Phi}E}{\rot_{\gat}\tau^{1}_{\Phi}E}_{\L^{2}(\om)}
-\lambda_{1,\Phi}\bscp{(\p_{\wt{\Psi}}\eps_{\Phi})\tau^{1}_{\Phi}E}{\tau^{1}_{\Phi}E}_{\L^{2}(\om)},\\
&=\Bscp{(\det J_{\Phi})^{-1}\big(
\p_{\wt{\Psi}}\wt{\mu^{-1}}-(\wt{\div\Psi})\wt{\mu^{-1}}+2\sym(\wt{\mu^{-1}}\wt{J_{\Psi}})
\big)J_{\Phi}\rot_{\gat}\tau^{1}_{\Phi}E}{J_{\Phi}\rot_{\gat}\tau^{1}_{\Phi}E}_{\L^{2}(\om)}\\
&\qquad-\lambda_{1,\Phi}\Bscp{(\det J_{\Phi})\big(
\p_{\wt{\Psi}}\wt{\eps}+(\wt{\div\Psi})\wt{\eps}-2\sym(\wt{J_{\Psi}}\wt{\eps})
\big)J_{\Phi}^{-\top}\tau^{1}_{\Phi}E}{J_{\Phi}^{-\top}\tau^{1}_{\Phi}E}_{\L^{2}(\om)},\\
&=\Bscp{(\det J_{\Phi})\big(
\p_{\wt{\Psi}}\wt{\mu^{-1}}-(\wt{\div\Psi})\wt{\mu^{-1}}+2\sym(\wt{\mu^{-1}}\wt{J_{\Psi}})
\big)\wt{\rot_{\gatp}E}}{\wt{\rot_{\gatp}E}}_{\L^{2}(\om)}\\
&\qquad-\lambda_{1,\Phi}\Bscp{(\det J_{\Phi})\big(
\p_{\wt{\Psi}}\wt{\eps}+(\wt{\div\Psi})\wt{\eps}-2\sym(\wt{J_{\Psi}}\wt{\eps})
\big)\wt{E}}{\wt{E}}_{\L^{2}(\om)},\\
&=\Bscp{\big(
\p_{\Psi}\mu^{-1}-(\div\Psi)\mu^{-1}+2\sym(\mu^{-1}J_{\Psi})
\big)\rot_{\gatp}E}{\rot_{\gatp}E}_{\L^{2}(\omp)}\\
&\qquad-\lambda_{1,\Phi}\Bscp{\big(
\p_{\Psi}\eps+(\div\Psi)\eps-2\sym(J_{\Psi}\eps)
\big)E}{E}_{\L^{2}(\omp)},\\
&=\Bscp{\big(
-\p_{\Psi}\mu-(\div\Psi)\mu+2\sym(J_{\Psi}\mu)
\big)\A_{1,\Phi}E}{\A_{1,\Phi}E}_{\L^{2}(\omp)}\\
&\qquad-\lambda_{1,\Phi}\Bscp{\big(
\p_{\Psi}\eps+(\div\Psi)\eps-2\sym(J_{\Psi}\eps)
\big)E}{E}_{\L^{2}(\omp)}.
\end{align*}
\end{proof}

In the particular case, where $\eps$, $\mu$, and $\nu$ 
are the identity mappings, Theorem \ref{theo:derevformal2} yields
\begin{align}
\label{eq:derevformal3}
\begin{aligned}
\frac{\p_{\wt{\Psi}}\lambda_{0,\Phi}}{\lambda_{0,\Phi}}
&=\bscp{(\div\Psi-2\sym J_{\Psi})H^{*}}{H^{*}}_{\L^{2}(\omp)}
-\bscp{(\div\Psi)u}{u}_{\L^{2}(\omp)}\\
&=-\bscp{(\symtr J_{\Psi})H^{*}}{H^{*}}_{\L^{2}(\omp)}
-\bscp{(\div\Psi)u}{u}_{\L^{2}(\omp)},\\
\frac{\p_{\wt{\Psi}}\lambda_{1,\Phi}}{\lambda_{1,\Phi}}
&=\bscp{(2\sym J_{\Psi}-\div\Psi)E^{*}}{E^{*}}_{\L^{2}(\omp)}
+\bscp{(2\sym J_{\Psi}-\div\Psi)E}{E}_{\L^{2}(\omp)}\\
&=\bscp{(\symtr J_{\Psi})E^{*}}{E^{*}}_{\L^{2}(\omp)}
+\bscp{(\symtr J_{\Psi})E}{E}_{\L^{2}(\omp)},\\
\frac{\p_{\wt{\Psi}}\lambda_{0,\Phi}}{\lambda_{0,\Phi}}
&=-\bscp{(\div\Psi)u^{*}}{u^{*}}_{\L^{2}(\omp)}
+\bscp{(\div\Psi-2\sym J_{\Psi})H}{H}_{\L^{2}(\omp)}\\
&=-\bscp{(\div\Psi)u^{*}}{u^{*}}_{\L^{2}(\omp)}
-\bscp{(\symtr J_{\Psi})H}{H}_{\L^{2}(\omp)}
\end{aligned}
\end{align}
with $\symtr M:=2\sym M-(\tr M)\cdot\id$, i.e.,
$$\symtr J_{\Psi}
=2\sym J_{\Psi}-(\tr J_{\Psi})\cdot\id
=J_{\Psi}+J_{\Psi}^{\top}-(\div\Psi)\cdot\id.$$
\eqref{eq:derevformal3} are the formulas \eqref{eq:derivintro} from the introduction
with $H=H^{*}$ and $B=E^{*}$.

\section*{Acknowledgments}

The authors are thankful to Prof. Massimo Lanza de Cristoforis for useful discussions and references. 
The first  named author acknowledges support by the project ``Perturbation problems and
asymptotics for elliptic differential equations: variational and potential theoretic methods'' funded by the MUR ``Progetti di Ricerca di Rilevante Interesse
Nazionale'' (PRIN) Bando 2022 grant no. 2022SENJZ3; he is also a member of the Gruppo Nazionale per l'Analisi  Matematica, la Probabilit\`{a} e le loro Applicazioni (GNAMPA) of the Istituto Nazionale di Alta Matematica (INdAM).
The third named author is very thankful for the kind 
invitation of the second author and the great hospitality of the 
Fakult\"at Mathematik, Institut f\"ur Analysis of the Technische Universit\"at Dresden
during his visiting period 
when some problems discussed in this paper where addressed.
He also acknowledges  financial support by ``Fondazione Ing. Aldo Gini'' for  his visit in Dresden.
His work was partially realized during his stay
at the Czech Technical University in Prague with the support of the EXPRO grant No.
20-17749X of the Czech Science Foundation. He has also received support from the French
government under the France 2030 investment plan, as part of the Initiative d’Excellence d’Aix-Marseille Université – A*MIDEX AMX-21-RID-012.


\bibliographystyle{plain} 


\bibliography{lpz1}


\section{Appendix}

\subsection{Computations of Shape Derivatives}
\label{app:heucomp}

Recall from Section \ref{sec:derhamop} the transformed matrices
\begin{align*}
\eps_{\Phi}
=\tau^{2}_{\Phi}\eps\tau^{1}_{\Phi^{-1}}
&=(\det J_{\Phi})J_{\Phi}^{-1}\wt{\eps}J_{\Phi}^{-\top},
&
\nu_{\Phi}
=\tau^{3}_{\Phi}\nu\tau^{0}_{\Phi^{-1}}
&=(\det J_{\Phi})\wt{\nu},\\
\eps_{\Phi}^{-1}
=\tau^{1}_{\Phi}\eps^{-1}\tau^{2}_{\Phi^{-1}}
&=(\det J_{\Phi})^{-1}J_{\Phi}^{\top}\wt{\eps^{-1}}J_{\Phi}.
\end{align*}
Note that generally 
$$\p_{v}f(x)=f'(x)v$$
and that $\p_{v}f(x)=f'(x)v=fv$ holds for bounded linear $f$.

\begin{lem}
\label{lem:derivatives}
Let $k\in\reals$. It holds 
\begin{align*}
\p_{\wt{\Psi}}J_{\Phi}
&=J_{\wt{\Psi}}
=\wt{J_{\Psi}}J_{\Phi},
&
\p_{\wt{\Psi}}J_{\Phi}^{-1}
&=-J_{\Phi}^{-1}J_{\wt{\Psi}}J_{\Phi}^{-1}
=-J_{\Phi}^{-1}\wt{J_{\Psi}},\\
\p_{\wt{\Psi}}J_{\Phi}^{\top}
&=J_{\wt{\Psi}}^{\top}
=J_{\Phi}^{\top}\wt{J_{\Psi}^{\top}},
&
\p_{\wt{\Psi}}J_{\Phi}^{-\top}
&=-J_{\Phi}^{-\top}J_{\wt{\Psi}}^{\top}J_{\Phi}^{-\top}
=-\wt{J_{\Psi}^{\top}}J_{\Phi}^{-\top},
\intertext{and}
\p_{\wt{\Psi}}(\det J_{\Phi})
&=(\det J_{\Phi})\wt{\div\Psi},
&
\p_{\wt{\Psi}}\big((\det J_{\Phi})^{k}\big)
&=k(\det J_{\Phi})^{k}\wt{\div\Psi}.
\end{align*}
Moreover, if $\eps$ and $\nu$ are  of class $C^1$, it holds
\begin{align*}
\p_{\wt{\Psi}}\eps_{\Phi}
&=(\det J_{\Phi})J_{\Phi}^{-1}\Big(
\p_{\wt{\Psi}}\wt{\eps}
+(\wt{\div\Psi})\wt{\eps}
-2\sym(\wt{J_{\Psi}}\wt{\eps})
\Big)J_{\Phi}^{-\top},\\
\p_{\wt{\Psi}}\eps_{\Phi}^{-1}
&=(\det J_{\Phi})^{-1}J_{\Phi}^{\top}\Big(
\p_{\wt{\Psi}}\wt{\eps^{-1}}
-(\wt{\div\Psi})\wt{\eps^{-1}}
+2\sym(\wt{\eps^{-1}}\wt{J_{\Psi}})
\Big)J_{\Phi},\\
\p_{\wt{\Psi}}\nu_{\Phi}
&=(\det J_{\Phi})\Big(\p_{\wt{\Psi}}\wt{\nu}+(\wt{\div\Psi})\wt{\nu}\Big)
=\nu_{\Phi}\Big(\frac{\p_{\wt{\Psi}}\wt{\nu}}{\wt{\nu}}+(\wt{\div\Psi})\Big),\\
\p_{\wt{\Psi}}\nu_{\Phi}^{k}
&=k(\det J_{\Phi})^{k}\wt{\nu}^{k-1}\Big(\p_{\wt{\Psi}}\wt{\nu}+(\wt{\div\Psi})\wt{\nu}\Big)
=k\nu_{\Phi}^{k}\Big(\frac{\p_{\wt{\Psi}}\wt{\nu}}{\wt{\nu}}+(\wt{\div\Psi})\Big).
\end{align*}
\end{lem}

\begin{rem}
\label{rem:derivatives1}
In particular, we have for $\eps=\id$
\begin{align*}
\p_{\wt{\Psi}}\id_{\Phi}
&=(\det J_{\Phi})J_{\Phi}^{-1}\big(\wt{\div\Psi}-2\sym\wt{J_{\Psi}}\big)J_{\Phi}^{-\top}
=-(\det J_{\Phi})J_{\Phi}^{-1}(\wt{\symtr J_{\Psi}})J_{\Phi}^{-\top},\\
\p_{\wt{\Psi}}\id_{\Phi}^{-1}
&=(\det J_{\Phi})^{-1}J_{\Phi}^{\top}\big(-\wt{\div\Psi}+2\sym\wt{J_{\Psi}}\big)J_{\Phi}
=(\det J_{\Phi})^{-1}J_{\Phi}^{\top}(\wt{\symtr J_{\Psi}})J_{\Phi}
\end{align*}
with $\symtr$ from \eqref{eq:derevformal3}, i.e.,
$\symtr J_{\Psi}=2\sym J_{\Psi}-\div\Psi$.
\end{rem}

\begin{rem}
\label{rem:derivatives2}
Note that 
\begin{align*}
\p_{\wt{\Psi}}\eps_{\Phi}^{-1}
&=-\eps_{\Phi}^{-1}(\p_{\wt{\Psi}}\eps_{\Phi})\eps_{\Phi}^{-1},
&
\p_{\wt{\Psi}}\nu_{\Phi}^{-1}
&=-\nu_{\Phi}^{-2}\p_{\wt{\Psi}}\nu_{\Phi},\\
\p_{\wt{\Psi}}\eps_{\Phi}
&=-\eps_{\Phi}(\p_{\wt{\Psi}}\eps_{\Phi}^{-1})\eps_{\Phi},
&
\p_{\wt{\Psi}}\nu_{\Phi}
&=-\nu_{\Phi}^{2}\p_{\wt{\Psi}}\nu_{\Phi}^{-1}.
\end{align*}
Similar formulas hold for $\eps_{\Phi}$ and $\nu_{\Phi}$ replaced by
$\wt{\eps}$ and $\wt{\nu}$, respectively.
\end{rem}

\begin{proof}[Proof of Lemma \ref{lem:derivatives}.]
By the chain rule we have $(\wt{\Psi})'=\wt{\Psi'}\Phi'$, i.e.,
$J_{\wt{\Psi}}=\wt{J_{\Psi}}J_{\Phi}$ and $J_{\wt{\Psi}}J_{\Phi}^{-1}=\wt{J_{\Psi}}$.
Since $\det(\id+sT)=1+s\tr T+O(s^2)$ we get with $\tr J_{\Psi}=\div\Psi$
\begin{align*}
\det J_{\Phi+s\wt{\Psi}}
&=\det(J_{\Phi}+sJ_{\wt{\Psi}})
=(\det J_{\Phi})\det(\id+sJ_{\wt{\Psi}}J_{\Phi}^{-1})
=(\det J_{\Phi})\det(\id+s\wt{J_{\Psi}})\\
&=(\det J_{\Phi})\big(1+s\tr\wt{J_{\Psi}}+O(s^2)\big)
=(\det J_{\Phi})\big(1+s\,\wt{\div\Psi}+O(s^2)\big).
\end{align*}
Moreover, for topological isomorphisms it holds $\p_{H}T^{-1}=(T^{-1})'H=-T^{-1}HT^{-1}$ as 
$$(T+H)^{-1}
=T^{-1}(\id+HT^{-1})^{-1}
=T^{-1}\sum_{n\geq0}(-HT^{-1})^{n}
=T^{-1}-T^{-1}HT^{-1}+O\big(|H|^2\big).$$
Then the first six and the last two derivatives in the lemma are easily computed. Furthermore,
using the latter results we get
\begin{align*}
\p_{\wt{\Psi}}\eps_{\Phi}
=\p_{\wt{\Psi}}\big((\det J_{\Phi})J_{\Phi}^{-1}\wt{\eps}J_{\Phi}^{-\top}\big)
&=\big(\p_{\wt{\Psi}}(\det J_{\Phi})\big)J_{\Phi}^{-1}\wt{\eps}J_{\Phi}^{-\top}
+(\det J_{\Phi})(\p_{\wt{\Psi}}J_{\Phi}^{-1})\wt{\eps}J_{\Phi}^{-\top}\\
&\qquad+(\det J_{\Phi})J_{\Phi}^{-1}(\p_{\wt{\Psi}}\wt{\eps})J_{\Phi}^{-\top}
+(\det J_{\Phi})J_{\Phi}^{-1}\wt{\eps}(\p_{\wt{\Psi}}J_{\Phi}^{-\top})\\
&=\big((\det J_{\Phi})\wt{\div\Psi}\big)J_{\Phi}^{-1}\wt{\eps}J_{\Phi}^{-\top}
-(\det J_{\Phi})J_{\Phi}^{-1}\wt{J_{\Psi}}\wt{\eps}J_{\Phi}^{-\top}\\
&\qquad+(\det J_{\Phi})J_{\Phi}^{-1}(\p_{\wt{\Psi}}\wt{\eps})J_{\Phi}^{-\top}
-(\det J_{\Phi})J_{\Phi}^{-1}\wt{\eps}\wt{J_{\Psi}^{\top}}J_{\Phi}^{-\top}\\
&=(\det J_{\Phi})J_{\Phi}^{-1}\Big((\wt{\div\Psi})\wt{\eps}
-\underbrace{\big(\wt{J_{\Psi}}\wt{\eps}+\wt{\eps}\wt{J_{\Psi}^{\top}}\big)}_{=
2\sym(\wt{J_{\Psi}}\wt{\eps})}+\p_{\wt{\Psi}}\wt{\eps}\Big)J_{\Phi}^{-\top}
\end{align*}
and, using this, by the chain rule
\begin{align*}
\p_{\wt{\Psi}}\eps_{\Phi}^{-1}
&=-\eps_{\Phi}^{-1}(\p_{\wt{\Psi}}\eps_{\Phi})\eps_{\Phi}^{-1}\\
&=-(\det J_{\Phi})^{-1}J_{\Phi}^{\top}\wt{\eps^{-1}}\underbrace{J_{\Phi}J_{\Phi}^{-1}}_{=\id}
\Big((\wt{\div\Psi})\wt{\eps}-2\sym(\wt{J_{\Psi}}\wt{\eps})+\p_{\wt{\Psi}}\wt{\eps}\Big)
\underbrace{J_{\Phi}^{-\top}J_{\Phi}^{\top}}_{=\id}\wt{\eps^{-1}}J_{\Phi}\\
&=-(\det J_{\Phi})^{-1}J_{\Phi}^{\top}
\Big((\wt{\div\Psi})\wt{\eps^{-1}}
-2\underbrace{\wt{\eps^{-1}}\sym(\wt{J_{\Psi}}\wt{\eps})\wt{\eps^{-1}}}_{=\sym(\wt{\eps^{-1}}\wt{J_{\Psi}})}
+\underbrace{\wt{\eps^{-1}}(\p_{\wt{\Psi}}\wt{\eps})\wt{\eps^{-1}}}_{=-\p_{\wt{\Psi}}\wt{\eps^{-1}}}\Big)J_{\Phi},
\end{align*}
finishing the proof. 
\end{proof}

\subsection{Variational Formulations}


For $\A_{\ell}$ and $\A_{\ell,\Phi}$, $\ell\in\{0,1\}$, from Section \ref{sec:pointspectrum} we note
the following variational formulations:
For all 
$$\psi_{\Phi}\in D(\A_{0,\Phi})=\H^{1}_{\gatp}(\omp),\quad
\Psi_{\Phi}\in D(\A_{1,\Phi})=\R_{\gatp}(\omp),\quad
\Theta_{\Phi}\in D(\A_{0,\Phi}^{*})=\eps^{-1}\D_{\ganp}(\omp)$$
it holds
\begin{align*}
\lambda_{0}\scp{\nu u}{\psi_{\Phi}}_{\L^{2}(\omp)}
&=\lambda_{0}\scp{u}{\psi_{\Phi}}_{\L^{2}_{\nu}(\omp)}
=\scp{\A_{0,\Phi}^{*}\A_{0,\Phi}u}{\psi_{\Phi}}_{\L^{2}_{\nu}(\omp)}\\
&=\scp{\A_{0,\Phi}u}{\A_{0,\Phi}\psi_{\Phi}}_{\L^{2}_{\eps}(\omp)}
=\scp{\eps\na_{\gatp}u}{\na_{\gatp}\psi_{\Phi}}_{\L^{2}(\omp)},\\
\lambda_{1}\scp{\eps E}{\Psi_{\Phi}}_{\L^{2}(\omp)}
&=\lambda_{1}\scp{E}{\Psi_{\Phi}}_{\L^{2}_{\eps}(\omp)}
=\scp{\A_{1,\Phi}^{*}\A_{1,\Phi}E}{\Psi_{\Phi}}_{\L^{2}_{\eps}(\omp)}\\
&=\scp{\A_{1,\Phi}E}{\A_{1,\Phi}\Psi_{\Phi}}_{\L^{2}_{\mu}(\omp)}
=\scp{\mu^{-1}\rot_{\gatp}E}{\rot_{\gatp}\Psi_{\Phi}}_{\L^{2}(\omp)},\\
\lambda_{0}\scp{\eps H}{\Theta_{\Phi}}_{\L^{2}(\omp)}
&=\lambda_{0}\scp{H}{\Theta_{\Phi}}_{\L^{2}_{\eps}(\omp)}
=\scp{\A_{0,\Phi}\A_{0,\Phi}^{*}H}{\Theta_{\Phi}}_{\L^{2}_{\eps}(\omp)}\\
&=\scp{\A_{0,\Phi}^{*}H}{\A_{0,\Phi}^{*}\Theta_{\Phi}}_{\L^{2}_{\nu}(\omp)}
=\scp{\nu^{-1}\div_{\ganp}\eps H}{\div_{\ganp}\eps\Theta_{\Phi}}_{\L^{2}(\omp)}.
\end{align*}
For all
$$\psi\in D(\A_{0})=\H^{1}_{\gat}(\om),\quad
\Psi\in D(\A_{1})=\R_{\gat}(\om),\quad
\Theta\in D(\A_{0}^{*})=\eps_{\Phi}^{-1}\D_{\gan}(\om)$$
it holds
\begin{align*}
\lambda_{0}\scp{\nu_{\Phi}\tau^{0}_{\Phi}u}{\psi}_{\L^{2}(\om)}
&=\lambda_{0}\scp{\tau^{0}_{\Phi}u}{\psi}_{\L^{2}_{\nu_{\Phi}}(\om)}
=\scp{\A_{0}^{*}\A_{0}\tau^{0}_{\Phi}u}{\psi}_{\L^{2}_{\nu_{\Phi}}(\om)}\\
&=\scp{\A_{0}\tau^{0}_{\Phi}u}{\A_{0}\psi}_{\L^{2}_{\eps_{\Phi}}(\om)}
=\scp{\eps_{\Phi}\na_{\gat}\tau^{0}_{\Phi}u}{\na_{\gat}\psi}_{\L^{2}(\om)},\\
\lambda_{1}\scp{\eps_{\Phi}\tau^{1}_{\Phi}E}{\Psi}_{\L^{2}(\om)}
&=\lambda_{1}\scp{\tau^{1}_{\Phi}E}{\Psi}_{\L^{2}_{\eps_{\Phi}}(\om)}
=\scp{\A_{1}^{*}\A_{1}\tau^{1}_{\Phi}E}{\Psi}_{\L^{2}_{\eps_{\Phi}}(\om)}\\
&=\scp{\A_{1}\tau^{1}_{\Phi}E}{\A_{1}\Psi}_{\L^{2}_{\mu_{\Phi}}(\om)}
=\scp{\mu_{\Phi}^{-1}\rot_{\gat}\tau^{1}_{\Phi}E}{\rot_{\gat}\Psi}_{\L^{2}(\om)},\\
\lambda_{0}\scp{\eps_{\Phi}\tau^{1}_{\Phi}H}{\Theta}_{\L^{2}(\om)}
&=\lambda_{0}\scp{\tau^{1}_{\Phi}H}{\Theta}_{\L^{2}_{\eps_{\Phi}}(\om)}
=\scp{\A_{0}\A_{0}^{*}\tau^{1}_{\Phi}H}{\Theta}_{\L^{2}_{\eps_{\Phi}}(\om)}\\
&=\scp{\A_{0}^{*}\tau^{1}_{\Phi}H}{\A_{0}^{*}\Theta}_{\L^{2}_{\nu_{\Phi}}(\om)}
=\scp{\nu_{\Phi}^{-1}\div_{\gan}\eps_{\Phi}\tau^{1}_{\Phi}H}{\div_{\gan}\eps_{\Phi}\Theta}_{\L^{2}(\om)}.
\end{align*}
Hence, more explicitly, 
\begin{align*}
\lambda_{0}\bscp{(\det J_{\Phi})\wt{\nu}\tau^{0}_{\Phi}u}{\psi}_{\L^{2}(\om)}
&=\bscp{(\det J_{\Phi})\wt{\eps}J_{\Phi}^{-\top}\na_{\gat}\tau^{0}_{\Phi}u}{J_{\Phi}^{-\top}\na_{\gat}\psi}_{\L^{2}(\om)},\\
\lambda_{1}\bscp{(\det J_{\Phi})\wt{\eps}J_{\Phi}^{-\top}\tau^{1}_{\Phi}E}{J_{\Phi}^{-\top}\Psi}_{\L^{2}(\om)}
&=\bscp{(\det J_{\Phi})^{-1}\wt{\mu}^{-1}J_{\Phi}\rot_{\gat}\tau^{1}_{\Phi}E}{J_{\Phi}\rot_{\gat}\Psi}_{\L^{2}(\om)},\\
\lambda_{0}\bscp{(\det J_{\Phi})\wt{\eps}J_{\Phi}^{-\top}\tau^{1}_{\Phi}H}{J_{\Phi}^{-\top}\Theta}_{\L^{2}(\om)}
&=\bscp{(\det J_{\Phi})^{-1}\wt{\nu}^{-1}\div_{\gan}\eps_{\Phi}\tau^{1}_{\Phi}H}{\div_{\gan}\eps_{\Phi}\Theta}_{\L^{2}(\om)}.
\end{align*}
Note that $\tau^{0}_{\Phi}u=\wt{u}$, $J_{\Phi}^{-\top}\tau^{1}_{\Phi}E=\wt{E}$, and 
$\eps_{\Phi}
=\tau^{2}_{\Phi}\eps\tau^{1}_{\Phi^{-1}}
=(\det J_{\Phi})J_{\Phi}^{-1}\wt{\eps}J_{\Phi}^{-\top}
=(\adj J_{\Phi})\wt{\eps}J_{\Phi}^{-\top}$.
Thus
\begin{align*}
\lambda_{0}\bscp{(\det J_{\Phi})\wt{\nu u}}{\psi}_{\L^{2}(\om)}
&=\bscp{(\det J_{\Phi})\wt{\eps}J_{\Phi}^{-\top}\na_{\gat}\wt{u}}
{J_{\Phi}^{-\top}\na_{\gat}\psi}_{\L^{2}(\om)},\\
\lambda_{1}\bscp{(\det J_{\Phi})\wt{\eps E}}
{J_{\Phi}^{-\top}\Psi}_{\L^{2}(\om)}
&=\bscp{(\det J_{\Phi})^{-1}\wt{\mu}^{-1}J_{\Phi}\rot_{\gat}J_{\Phi}^{\top}\wt{E}}
{J_{\Phi}\rot_{\gat}\Psi}_{\L^{2}(\om)},\\
\lambda_{0}\bscp{(\det J_{\Phi})\wt{\eps H}}{J_{\Phi}^{-\top}\Theta}_{\L^{2}(\om)}
&=\bscp{(\det J_{\Phi})^{-1}\wt{\nu}^{-1}\div_{\gan}(\adj J_{\Phi})\wt{\eps H}}
{\div_{\gan}(\adj J_{\Phi})\wt{\eps}J_{\Phi}^{-\top}\Theta}_{\L^{2}(\om)}.
\end{align*}


\subsection{Some Additional Proofs}

\begin{proof}[\bf Proof of Lemma \ref{lem:densefullbc1}]
Consider the densely defined and closed linear operators
\begin{align*}
\A_{0}:=\na_{\ga}:\H^{1}_{\ga}(\om)\subset\L^{2}(\om)&\to\L^{2}(\om),\\
\A_{1}:=\rot_{\ga}:\R_{\ga}(\om)\subset\L^{2}(\om)&\to\L^{2}(\om),\\
\A_{2}:=\div_{\ga}:\D_{\ga}(\om)\subset\L^{2}(\om)&\to\L^{2}(\om)
\intertext{together with their densely defined and closed adjoints}
\A_{0}^{*}=-\div:\bD(\om)\subset\L^{2}(\om)&\to\L^{2}(\om),\\
\A_{1}^{*}=\rot:\bR(\om)\subset\L^{2}(\om)&\to\L^{2}(\om),\\
\A_{2}^{*}=-\na:\bH^{1}(\om)\subset\L^{2}(\om)&\to\L^{2}(\om)
\end{align*}
and recall that generally $\A_{\ell}^{**}=\ol{\A_{\ell}}=\A_{\ell}$. 
Then, e.g., for the rotor 
\begin{align*}
\R_{\ga}(\om)
&=D(\A_{1})
=D(\A_{1}^{**})\\
&=\big\{\Psi\in\L^{2}(\om)\,:\,
\exists\,\Psi_{\A_{1}^{**}}\in\L^{2}(\om)\quad
\forall\,\Theta\in D(\A_{1}^{*})\quad
\scp{\Psi}{\A_{1}^{*}\Theta}_{\L^{2}(\om)}=\scp{\Psi_{\A_{1}^{**}}}{\Theta}_{\L^{2}(\om)}\big\}\\
&=\big\{\Psi\in\L^{2}(\om)\,:\,
\exists\,\Psi_{\rot}\in\L^{2}(\om)\quad
\forall\,\Theta\in\bR(\om)\quad
\scp{\Psi}{\rot\Theta}_{\L^{2}(\om)}=\scp{\Psi_{\rot}}{\Theta}_{\L^{2}(\om)}\big\}\\
&=\big\{\Psi\in\bR(\om)\,:\,
\forall\,\Theta\in\bR(\om)\quad
\scp{\Psi}{\rot\Theta}_{\L^{2}(\om)}=\scp{\rot\Psi}{\Theta}_{\L^{2}(\om)}\big\},
\end{align*}
finishing the proof.
\end{proof}

\begin{proof}[\bf Longer Proof of Theorem \ref{theo:derevformal1}] 

By \eqref{eq:eigenvaluesshapederivatives1} and the quotient rule we compute
{\small
\begin{align*}
(\p_{\wt{\Psi}}\lambda_{0,\Phi})
\scp{\nu_{\Phi}\tau^{0}_{\Phi}u}{\tau^{0}_{\Phi}u}_{\L^{2}(\om)}^{2}
&=\scp{\nu_{\Phi}\tau^{0}_{\Phi}u}{\tau^{0}_{\Phi}u}_{\L^{2}(\om)}
\p_{\wt{\Psi}}\scp{\eps_{\Phi}\na_{\gat}\tau^{0}_{\Phi}u}{\na_{\gat}\tau^{0}_{\Phi}u}_{\L^{2}(\om)}\\
&\qquad
-\scp{\eps_{\Phi}\na_{\gat}\tau^{0}_{\Phi}u}{\na_{\gat}\tau^{0}_{\Phi}u}_{\L^{2}(\om)}
\p_{\wt{\Psi}}\scp{\nu_{\Phi}\tau^{0}_{\Phi}u}{\tau^{0}_{\Phi}u}_{\L^{2}(\om)}\\
&=\scp{\nu_{\Phi}\tau^{0}_{\Phi}u}{\tau^{0}_{\Phi}u}_{\L^{2}(\om)}
\Big(\bscp{(\p_{\wt{\Psi}}\eps_{\Phi})\na_{\gat}\tau^{0}_{\Phi}u}{\na_{\gat}\tau^{0}_{\Phi}u}_{\L^{2}(\om)}\\
&\qquad\qquad
+2\Re\scp{\eps_{\Phi}\na_{\gat}\tau^{0}_{\Phi}u}{\na_{\gat}\p_{\wt{\Psi}}\tau^{0}_{\Phi}u}_{\L^{2}(\om)}\Big)\\
&\qquad
-\scp{\eps_{\Phi}\na_{\gat}\tau^{0}_{\Phi}u}{\na_{\gat}\tau^{0}_{\Phi}u}_{\L^{2}(\om)}
\Big(\bscp{(\p_{\wt{\Psi}}\nu_{\Phi})\tau^{0}_{\Phi}u}{\tau^{0}_{\Phi}u}_{\L^{2}(\om)}\\
&\qquad\qquad
+2\Re\scp{\nu_{\Phi}\tau^{0}_{\Phi}u}{\p_{\wt{\Psi}}\tau^{0}_{\Phi}u}_{\L^{2}(\om)}\Big),\\
(\p_{\wt{\Psi}}\lambda_{0,\Phi})
\scp{\eps_{\Phi}\tau^{1}_{\Phi}H}{\tau^{1}_{\Phi}H}_{\L^{2}(\om)}^{2}
&=\scp{\eps_{\Phi}\tau^{1}_{\Phi}H}{\tau^{1}_{\Phi}H}_{\L^{2}(\om)}
\p_{\wt{\Psi}}\scp{\nu_{\Phi}^{-1}\div_{\gan}\eps_{\Phi}\tau^{1}_{\Phi}H}{\div_{\gan}\eps_{\Phi}\tau^{1}_{\Phi}H}_{\L^{2}(\om)}\\
&\qquad
-\scp{\nu_{\Phi}^{-1}\div_{\gan}\eps_{\Phi}\tau^{1}_{\Phi}H}{\div_{\gan}\eps_{\Phi}\tau^{1}_{\Phi}H}_{\L^{2}(\om)}
\p_{\wt{\Psi}}\scp{\eps_{\Phi}\tau^{1}_{\Phi}H}{\tau^{1}_{\Phi}H}_{\L^{2}(\om)}\\
&=\scp{\eps_{\Phi}\tau^{1}_{\Phi}H}{\tau^{1}_{\Phi}H}_{\L^{2}(\om)}
\Big(\bscp{(\p_{\wt{\Psi}}\nu_{\Phi}^{-1})\div_{\gan}\eps_{\Phi}\tau^{1}_{\Phi}H}{\div_{\gan}\eps_{\Phi}\tau^{1}_{\Phi}H}_{\L^{2}(\om)}\\
&\qquad\qquad
+2\Re\bscp{\nu_{\Phi}^{-1}\div_{\gan}\eps_{\Phi}\tau^{1}_{\Phi}H}{\div_{\gan}\p_{\wt{\Psi}}(\eps_{\Phi}\tau^{1}_{\Phi}H)}_{\L^{2}(\om)}\Big)\\
&\qquad
-\scp{\nu_{\Phi}^{-1}\div_{\gan}\eps_{\Phi}\tau^{1}_{\Phi}H}{\div_{\gan}\eps_{\Phi}\tau^{1}_{\Phi}H}_{\L^{2}(\om)}
\Big(\bscp{(\p_{\wt{\Psi}}\eps_{\Phi}^{-1})\eps_{\Phi}\tau^{1}_{\Phi}H}{\eps_{\Phi}\tau^{1}_{\Phi}H}_{\L^{2}(\om)}\\
&\qquad\qquad+2\Re\bscp{\tau^{1}_{\Phi}H}{\p_{\wt{\Psi}}(\eps_{\Phi}\tau^{1}_{\Phi}H)}_{\L^{2}(\om)}\Big)
\end{align*}
}
and thus using
\begin{align*}
\scp{\eps_{\Phi}\na_{\gat}\tau^{0}_{\Phi}u}{\na_{\gat}\tau^{0}_{\Phi}u}_{\L^{2}(\om)}
&=\lambda_{0,\Phi}\scp{\nu_{\Phi}\tau^{0}_{\Phi}u}{\tau^{0}_{\Phi}u}_{\L^{2}(\om)},\\
\scp{\nu_{\Phi}^{-1}\div_{\gan}\eps_{\Phi}\tau^{1}_{\Phi}H}{\div_{\gan}\eps_{\Phi}\tau^{1}_{\Phi}H}_{\L^{2}(\om)}
&=\lambda_{0,\Phi}\scp{\eps_{\Phi}\tau^{1}_{\Phi}H}{\tau^{1}_{\Phi}H}_{\L^{2}(\om)}
\end{align*}
we see
{\small
\begin{align*}
(\p_{\wt{\Psi}}\lambda_{0,\Phi})
\scp{\nu_{\Phi}\tau^{0}_{\Phi}u}{\tau^{0}_{\Phi}u}_{\L^{2}(\om)}
&=\bscp{(\p_{\wt{\Psi}}\eps_{\Phi})\na_{\gat}\tau^{0}_{\Phi}u}{\na_{\gat}\tau^{0}_{\Phi}u}_{\L^{2}(\om)}\\
&\qquad+2\Re\scp{\eps_{\Phi}\na_{\gat}\tau^{0}_{\Phi}u}{\na_{\gat}\p_{\wt{\Psi}}\tau^{0}_{\Phi}u}_{\L^{2}(\om)}\\
&\qquad
-\lambda_{0,\Phi}
\Big(\bscp{(\p_{\wt{\Psi}}\nu_{\Phi})\tau^{0}_{\Phi}u}{\tau^{0}_{\Phi}u}_{\L^{2}(\om)}
+2\Re\scp{\nu_{\Phi}\tau^{0}_{\Phi}u}{\p_{\wt{\Psi}}\tau^{0}_{\Phi}u}_{\L^{2}(\om)}\Big)\\
&=\bscp{(\p_{\wt{\Psi}}\eps_{\Phi})\na_{\gat}\tau^{0}_{\Phi}u}{\na_{\gat}\tau^{0}_{\Phi}u}_{\L^{2}(\om)}
+2\Re\scp{\A_{0}\tau^{0}_{\Phi}u}{\A_{0}\p_{\wt{\Psi}}\tau^{0}_{\Phi}u}_{\L^{2}_{\eps_{\Phi}}(\om)}\\
&\qquad
-\lambda_{0,\Phi}
\Big(\bscp{(\p_{\wt{\Psi}}\nu_{\Phi})\tau^{0}_{\Phi}u}{\tau^{0}_{\Phi}u}_{\L^{2}(\om)}
+2\Re\scp{\tau^{0}_{\Phi}u}{\p_{\wt{\Psi}}\tau^{0}_{\Phi}u}_{\L^{2}_{\nu_{\Phi}}(\om)}\Big)\\
&=\bscp{(\p_{\wt{\Psi}}\eps_{\Phi})\na_{\gat}\tau^{0}_{\Phi}u}{\na_{\gat}\tau^{0}_{\Phi}u}_{\L^{2}(\om)}
-\lambda_{0,\Phi}\bscp{(\p_{\wt{\Psi}}\nu_{\Phi})\tau^{0}_{\Phi}u}{\tau^{0}_{\Phi}u}_{\L^{2}(\om)}\\
&\qquad
+2\Re\bscp{\underbrace{(\A_{0}^{*}\A_{0}-\lambda_{0,\Phi})\tau^{0}_{\Phi}u}_{=0}}{\p_{\wt{\Psi}}\tau^{0}_{\Phi}u}_{\L^{2}_{\nu_{\Phi}}(\om)}\\
&=\scp{\na_{\gat}\tau^{0}_{\Phi}u}{\na_{\gat}\tau^{0}_{\Phi}u}_{\L^{2}_{(\p_{\wt{\Psi}}\eps_{\Phi})}(\om)}
-\lambda_{0,\Phi}\scp{\tau^{0}_{\Phi}u}{\tau^{0}_{\Phi}u}_{\L^{2}_{(\p_{\wt{\Psi}}\nu_{\Phi})}(\om)},\\
(\p_{\wt{\Psi}}\lambda_{0,\Phi})
\scp{\eps_{\Phi}\tau^{1}_{\Phi}H}{\tau^{1}_{\Phi}H}_{\L^{2}(\om)}
&=\bscp{(\p_{\wt{\Psi}}\nu_{\Phi}^{-1})\div_{\gan}\eps_{\Phi}\tau^{1}_{\Phi}H}{\div_{\gan}\eps_{\Phi}\tau^{1}_{\Phi}H}_{\L^{2}(\om)}\\
&\qquad+2\Re\bscp{\nu_{\Phi}^{-1}\div_{\gan}\eps_{\Phi}\tau^{1}_{\Phi}H}{\div_{\gan}\p_{\wt{\Psi}}(\eps_{\Phi}\tau^{1}_{\Phi}H)}_{\L^{2}(\om)}\\
&\qquad
-\lambda_{0,\Phi}
\Big(2\Re\bscp{\tau^{1}_{\Phi}H}{\p_{\wt{\Psi}}(\eps_{\Phi}\tau^{1}_{\Phi}H)}_{\L^{2}(\om)}\\
&\qquad\qquad+\bscp{(\p_{\wt{\Psi}}\eps_{\Phi}^{-1})\eps_{\Phi}\tau^{1}_{\Phi}H}{\eps_{\Phi}\tau^{1}_{\Phi}H}_{\L^{2}(\om)}\Big)\\
&=\bscp{(\p_{\wt{\Psi}}\nu_{\Phi}^{-1})\div_{\gan}\eps_{\Phi}\tau^{1}_{\Phi}H}{\div_{\gan}\eps_{\Phi}\tau^{1}_{\Phi}H}_{\L^{2}(\om)}\\
&\qquad+2\Re\bscp{\A_{0}^{*}\tau^{1}_{\Phi}H}{\A_{0}^{*}\eps_{\Phi}^{-1}\p_{\wt{\Psi}}(\eps_{\Phi}\tau^{1}_{\Phi}H)}_{\L^{2}_{\nu_{\Phi}}(\om)}\\
&\qquad
-\lambda_{0,\Phi}
\Big(\bscp{(\p_{\wt{\Psi}}\eps_{\Phi}^{-1})\eps_{\Phi}\tau^{1}_{\Phi}H}{\eps_{\Phi}\tau^{1}_{\Phi}H}_{\L^{2}(\om)}\\
&\qquad\qquad+2\Re\bscp{\tau^{1}_{\Phi}H}{\eps_{\Phi}^{-1}\p_{\wt{\Psi}}(\eps_{\Phi}\tau^{1}_{\Phi}H)}_{\L^{2}_{\eps_{\Phi}}(\om)}\Big)\\
&=\bscp{(\p_{\wt{\Psi}}\nu_{\Phi}^{-1})\div_{\gan}\eps_{\Phi}\tau^{1}_{\Phi}H}{\div_{\gan}\eps_{\Phi}\tau^{1}_{\Phi}H}_{\L^{2}(\om)}\\
&\qquad-\lambda_{0,\Phi}\bscp{(\p_{\wt{\Psi}}\eps_{\Phi}^{-1})\eps_{\Phi}\tau^{1}_{\Phi}H}{\eps_{\Phi}\tau^{1}_{\Phi}H}_{\L^{2}(\om)}\\
&\qquad
+2\Re\bscp{\underbrace{(\A_{0}\A_{0}^{*}-\lambda_{0,\Phi})\tau^{1}_{\Phi}H}_{=0}}{\eps_{\Phi}^{-1}\p_{\wt{\Psi}}(\eps_{\Phi}\tau^{1}_{\Phi}H)}_{\L^{2}_{\eps_{\Phi}}(\om)}\\
&=\scp{\div_{\gan}\eps_{\Phi}\tau^{1}_{\Phi}H}{\div_{\gan}\eps_{\Phi}\tau^{1}_{\Phi}H}_{\L^{2}_{(\p_{\wt{\Psi}}\nu_{\Phi}^{-1})}(\om)}
+\lambda_{0,\Phi}\scp{\tau^{1}_{\Phi}H}{\tau^{1}_{\Phi}H}_{\L^{2}_{(\p_{\wt{\Psi}}\eps_{\Phi})}(\om)}.
\end{align*}
}
Note that 
$\p_{\wt{\Psi}}\eps_{\Phi}=-\eps_{\Phi}(\p_{\wt{\Psi}}\eps_{\Phi}^{-1})\eps_{\Phi}$
by Remark \ref{rem:derivatives2}.
Therefore, for normalised eigenfields $u$, $E$, and $H$ we obtain the assertions.
\end{proof}

\begin{proof}[\bf Longer Proof of Theorem \ref{theo:derevformal2}]
By Theorem \ref{theo:transtheo}, Corollary \ref{cor:transtheo}, Remark \ref{rem:transtheo1} and Lemma~\ref{lem:derivatives} we see
{\small
\begin{align*}
\p_{\wt{\Psi}}\lambda_{0,\Phi}
&=\bscp{(\p_{\wt{\Psi}}\eps_{\Phi})\na_{\gat}\tau^{0}_{\Phi}u}{\na_{\gat}\tau^{0}_{\Phi}u}_{\L^{2}(\om)}
-\lambda_{0,\Phi}\bscp{(\p_{\wt{\Psi}}\nu_{\Phi})\tau^{0}_{\Phi}u}{\tau^{0}_{\Phi}u}_{\L^{2}(\om)}\\
&=\Bscp{(\det J_{\Phi})\big(
\p_{\wt{\Psi}}\wt{\eps}+(\wt{\div\Psi})\wt{\eps}-2\sym(\wt{J_{\Psi}}\wt{\eps})
\big)J_{\Phi}^{-\top}\na_{\gat}\tau^{0}_{\Phi}u}{J_{\Phi}^{-\top}\na_{\gat}\tau^{0}_{\Phi}u}_{\L^{2}(\om)}\\
&\qquad-\lambda_{0,\Phi}
\Bscp{(\det J_{\Phi})\big(\p_{\wt{\Psi}}\wt{\nu}+(\wt{\div\Psi})\wt{\nu}\big)\tau^{0}_{\Phi}u}{\tau^{0}_{\Phi}u}_{\L^{2}(\om)}\\
&=\Bscp{(\det J_{\Phi})\big(
\p_{\wt{\Psi}}\wt{\eps}+(\wt{\div\Psi})\wt{\eps}-2\sym(\wt{J_{\Psi}}\wt{\eps})
\big)\wt{\na_{\gatp}u}}{\wt{\na_{\gatp}u}}_{\L^{2}(\om)}\\
&\qquad-\lambda_{0,\Phi}
\Bscp{(\det J_{\Phi})\big(\p_{\wt{\Psi}}\wt{\nu}+(\wt{\div\Psi})\wt{\nu}\big)\wt{u}}{\wt{u}}_{\L^{2}(\om)}\\
&=\Bscp{\big(
\p_{\Psi}\eps+(\div\Psi)\eps-2\sym(J_{\Psi}\eps)
\big)\na_{\gatp}u}{\na_{\gatp}u}_{\L^{2}(\omp)}\\
&\qquad-\lambda_{0,\Phi}
\Bscp{\big(\p_{\Psi}\nu+(\div\Psi)\nu\big)u}{u}_{\L^{2}(\omp)},\\
\p_{\wt{\Psi}}\lambda_{1,\Phi}
&=\bscp{(\p_{\wt{\Psi}}\mu_{\Phi}^{-1})\rot_{\gat}\tau^{1}_{\Phi}E}{\rot_{\gat}\tau^{1}_{\Phi}E}_{\L^{2}(\om)}
-\lambda_{1,\Phi}\bscp{(\p_{\wt{\Psi}}\eps_{\Phi})\tau^{1}_{\Phi}E}{\tau^{1}_{\Phi}E}_{\L^{2}(\om)},\\
&=\Bscp{(\det J_{\Phi})^{-1}\big(
\p_{\wt{\Psi}}\wt{\mu^{-1}}-(\wt{\div\Psi})\wt{\mu^{-1}}+2\sym(\wt{\mu^{-1}}\wt{J_{\Psi}})
\big)J_{\Phi}\rot_{\gat}\tau^{1}_{\Phi}E}{J_{\Phi}\rot_{\gat}\tau^{1}_{\Phi}E}_{\L^{2}(\om)}\\
&\qquad-\lambda_{1,\Phi}\Bscp{(\det J_{\Phi})\big(
\p_{\wt{\Psi}}\wt{\eps}+(\wt{\div\Psi})\wt{\eps}-2\sym(\wt{J_{\Psi}}\wt{\eps})
\big)J_{\Phi}^{-\top}\tau^{1}_{\Phi}E}{J_{\Phi}^{-\top}\tau^{1}_{\Phi}E}_{\L^{2}(\om)},\\
&=\Bscp{(\det J_{\Phi})\big(
\p_{\wt{\Psi}}\wt{\mu^{-1}}-(\wt{\div\Psi})\wt{\mu^{-1}}+2\sym(\wt{\mu^{-1}}\wt{J_{\Psi}})
\big)\wt{\rot_{\gatp}E}}{\wt{\rot_{\gatp}E}}_{\L^{2}(\om)}\\
&\qquad-\lambda_{1,\Phi}\Bscp{(\det J_{\Phi})\big(
\p_{\wt{\Psi}}\wt{\eps}+(\wt{\div\Psi})\wt{\eps}-2\sym(\wt{J_{\Psi}}\wt{\eps})
\big)\wt{E}}{\wt{E}}_{\L^{2}(\om)},\\
&=\Bscp{\big(
\p_{\Psi}\mu^{-1}-(\div\Psi)\mu^{-1}+2\sym(\mu^{-1}J_{\Psi})
\big)\rot_{\gatp}E}{\rot_{\gatp}E}_{\L^{2}(\omp)}\\
&\qquad-\lambda_{1,\Phi}\Bscp{\big(
\p_{\Psi}\eps+(\div\Psi)\eps-2\sym(J_{\Psi}\eps)
\big)E}{E}_{\L^{2}(\omp)},\\
\p_{\wt{\Psi}}\lambda_{0,\Phi}
&=\bscp{(\p_{\wt{\Psi}}\nu_{\Phi}^{-1})\div_{\gan}\eps_{\Phi}\tau^{1}_{\Phi}H}{\div_{\gan}\eps_{\Phi}\tau^{1}_{\Phi}H}_{\L^{2}(\om)}
+\lambda_{0,\Phi}\bscp{(\p_{\wt{\Psi}}\eps_{\Phi})\tau^{1}_{\Phi}H}{\tau^{1}_{\Phi}H}_{\L^{2}(\om)}\\
&=\Bscp{(\det J_{\Phi})^{-1}\big(\p_{\wt{\Psi}}\wt{\nu^{-1}}-(\wt{\div\Psi})\wt{\nu^{-1}}\big)
\div_{\gan}\eps_{\Phi}\tau^{1}_{\Phi}H}{\div_{\gan}\eps_{\Phi}\tau^{1}_{\Phi}H}_{\L^{2}(\om)}\\
&\qquad+\lambda_{0,\Phi}\Bscp{(\det J_{\Phi})\big(
\p_{\wt{\Psi}}\wt{\eps}+(\wt{\div\Psi})\wt{\eps}-2\sym(\wt{J_{\Psi}}\wt{\eps})
\big)J_{\Phi}^{-\top}\tau^{1}_{\Phi}H}{J_{\Phi}^{-\top}\tau^{1}_{\Phi}H}_{\L^{2}(\om)}\\
&=\Bscp{(\det J_{\Phi})\big(\p_{\wt{\Psi}}\wt{\nu^{-1}}-(\wt{\div\Psi})\wt{\nu^{-1}}\big)
\wt{\div_{\ganp}\eps H}}{\wt{\div_{\ganp}\eps H}}_{\L^{2}(\om)}\\
&\qquad+\lambda_{0,\Phi}\Bscp{(\det J_{\Phi})\big(
\p_{\wt{\Psi}}\wt{\eps}+(\wt{\div\Psi})\wt{\eps}-2\sym(\wt{J_{\Psi}}\wt{\eps})
\big)\wt{H}}{\wt{H}}_{\L^{2}(\om)}\\
&=\Bscp{\big(\p_{\Psi}\nu^{-1}-(\div\Psi)\nu^{-1}\big)
\div_{\ganp}\eps H}{\div_{\ganp}\eps H}_{\L^{2}(\omp)}\\
&\qquad+\lambda_{0,\Phi}\Bscp{\big(
\p_{\Psi}\eps+(\div\Psi)\eps-2\sym(J_{\Psi}\eps)
\big)H}{H}_{\L^{2}(\omp)},
\end{align*}
}
finishing the proof.
\end{proof}


\vspace*{5mm}
\hrule
\vspace*{3mm}


\end{document}